\documentclass[12pt]{article}
\usepackage[margin=.88in]{geometry}

\usepackage{multirow}
\usepackage{amsmath, amssymb,amsthm} 
\usepackage{bm}
\usepackage[pdftex]{graphicx}
\usepackage{algorithm, algorithmic}
\usepackage{comment}
\usepackage{caption}
\usepackage{subcaption}
\usepackage{enumitem}
\usepackage{natbib}
\usepackage{diagbox}
\usepackage{titletoc}

\usepackage{hyperref}
\hypersetup{
	colorlinks=true,
	linkcolor=blue,
	filecolor=blue,      
	urlcolor=red,
	citecolor=blue,
}

\DeclareMathAlphabet\mathbfcal{OMS}{cmsy}{b}{n}

\newcommand{\Var}{\text{Var}}

\newcommand{\argmin}{\operatornamewithlimits{argmin}}

\renewcommand{\(}{\left(}
\renewcommand{\)}{\right)}

\newcommand{\be}{\begin{equation}}
	\newcommand{\ee}{\end{equation}}
 \newcommand{\bs}{\begin{split}}
	\newcommand{\es}{\end{split}}
\newcommand{\bea}{\begin{eqnarray}}
	\newcommand{\eea}{\end{eqnarray}}
\newcommand{\beas}{\begin{eqnarray*}}
	\newcommand{\eeas}{\end{eqnarray*}}

\usepackage{color}

\newcommand{\T}{^\top}
\newcommand{\E}{\mathbb{E}}
\newcommand{\PP}{\mathbb{P}}
\newcommand{\R}{\mathbb{R}}

\newcommand{\bcA}{{\mathbfcal{A}}}
\newcommand{\bcM}{{\mathbfcal{M}}}

\newcommand{\bcT}{{\mathbfcal{T}}}

\newcommand{\bcK}{{\mathbfcal{K}}}

\newtheorem*{theorem*}{Theorem}
\newtheorem{theorem}{Theorem}
\newtheorem{conjecture}{Conjecture}
\newtheorem{lemma}{Lemma}
\newtheorem{corollary}{Corollary}
\newtheorem{proposition}{Proposition}

\newtheorem{remark}{Remark}

\usepackage{setspace}
\title{Statistical and Computational Limits for Cumulant Tensor Inference: A Regime in Which Efficient Detection Is Harder Than Efficient Estimation}

\author{Runshi Tang\footnote{Department of Statistics, University of Wisconsin-Madison}, ~ 
Yuefeng Han\footnote{Department of Applied and Computational Mathematics and Statistics, University of Notre Dame}, ~ and ~ 
Anru R. Zhang\footnote{Department of Biostatistics \& Bioinformatics and Department of Computer Science, Duke University}}

\date{}

\begin{document}

\maketitle

\begin{abstract}
We study estimation and detection of high-order cumulant tensors from $n$ i.i.d. observations of a $p$-dimensional random vector, with performance measured in the tensor spectral norm. Under sub-Gaussianity, we show that the minimax rate for estimating the order-$d$ cumulant tensor is $\sqrt{p/n}\wedge 1$. In contrast to covariance estimation, the natural plug-in estimator obtained by replacing population moments in the moment--cumulant formula with sample moments, which we call the sample cumulant tensor, is generally not rate-optimal for $d\ge 3$, and we construct an estimator that attains the minimax rate. 
On the computational side, we study testing whether the $d$-th order cumulant tensor vanishes after whitening. 
We provide evidence, under a model-specific low-degree conjecture, that detection is computationally hard when $n\ll p^{d/2}$. 
At the same time, we identify a regime in which an efficiently computable estimator achieves error smaller than the separation at which polynomial-time tests can reliably distinguish the null from the alternative under the same conjecture. 
This reveals an unusual reverse detection--estimation gap: computationally efficient detection can be harder than computationally efficient estimation. 
The underlying reason is that the relevant loss, tensor spectral norm, is itself NP-hard to compute, creating a new form of computational--statistical gap. 
\end{abstract}

\section{Introduction}\label{sec:intro}

High-order cumulants play a central role in modern statistics and learning theory. Beyond the covariance matrix, which captures only second-order dependence, higher-order cumulant tensors encode non-Gaussian structure that underlies fundamental tasks such as 
latent variable estimation \citep{anandkumar2014tensor}, 
method-of-moments learning \citep{pearson1894contributions}, 
distribution-free covariance inference \citep{browne1984asymptotically}, 
higher-order asymptotic approximations \citep{javed2024edgeworth}, 
and multi-view representation learning \citep{anandkumar2012method}. In many of these problems, high-order tensors serve to identify hidden factors or to distinguish structured alternatives from Gaussian noise. Motivated by these developments, this paper studies the statistical and computational limits of estimation and detection of high-order cumulant tensors, without imposing additional structural assumptions such as sparsity or low rankness.

Suppose we observe i.i.d.\ samples $X_1,\ldots,X_n \in \mathbb{R}^p$ from an unknown distribution of a random vector $X$. The main object of this paper is the $d$-th order cumulant tensor, which, whenever the moment generating function exists in a neighborhood of the origin, can be expressed as
\[
\bcK_d = \nabla^d \log\!\left(\mathbb{E}\, e^{t^\top X}\right)\big|_{t=0}.
\]
Note that $\bcK_1$ and $\bcK_2$ are the mean vector and covariance matrix, respectively. Higher-order cumulants, of order $d \geq 3$, provide a canonical characterization of non-Gaussianity: all cumulants of order $d\ge 3$ vanish for Gaussian distributions, and cumulants are additive under independent summation. These properties make $\bcK_d$ a natural object both in statistical theory and in computational pipelines that first whiten second-order structure and then extract higher-order deviations from Gaussianity. Moment tensors remain closely related through the moment--cumulant formula. Although the main discussion focuses on cumulant tensors, the corresponding arguments and algorithms extend naturally to moment tensors. 

A key feature of our setting is that we measure estimation error in the tensor spectral norm (operator norm),
\[
\|\bcT\|
:=
\sup_{u_k: \|u_k\|=1, k\in[d]}
\Big\langle \bcT,\ \bigotimes_{k\in[d]} u_k\Big\rangle,
\]
the natural analogue of matrix operator norm. This loss is the benchmark in covariance estimation because it uniformly controls directional quadratic forms and directly governs eigenvalue, eigenspace, and PCA perturbation. In the tensor setting, spectral norm likewise controls worst-case multilinear functionals, underlies perturbation guarantees for spectral and tensor power methods, and connects naturally to testing problems that distinguish Gaussian-like distributions from structured alternatives. The Frobenius norm instead simply aggregates error over all tensor entries.

A concrete motivation for spectral-norm estimation of unstructured cumulant tensors arises in cumulant-based projection pursuit \citep{jones1987projection,loperfido2018skewness}, which seeks directions exhibiting strong non-Gaussianity. 
For a centered and whitened random vector $X$, the $d$-th order cumulant of any projection $u^\top X$ is given by $\langle\bcK_d,u^{\otimes d}\rangle$ for $\|u\|=1$, encompassing skewness ($d=3$), excess kurtosis ($d=4$), and higher-order measures of non-Gaussianity. 
Thus, estimating $\bcK_d$ in spectral norm provides uniform control of these projection indices across all directions. 
Unlike structured settings, projection pursuit can be formulated without assuming any structure in the cumulant tensor, motivating the study of unstructured high-order cumulant estimation in spectral norm.

When $d=2$, covariance estimation in spectral norm has been extensively studied: the sample covariance achieves the optimal rate $\sqrt{p/n}$ for $n\gtrsim p$ under sub-Gaussian assumptions, while for heavy-tailed distributions, suitably truncated versions attain the same rate \citep{keUserFriendlyCovarianceEstimation2019,koltchinskii2017concentration}. For higher-order tensors, however, two striking differences arise.

\begin{enumerate}
    \item Unlike the covariance case, the naive sample cumulant estimator is no longer rate-optimal for $d\ge 3$, even under sub-Gaussianity.

    \item For $d\ge 3$, computing, certifying, or even approximating the tensor spectral norm is NP-hard in general \citep{hillarMostTensorProblems2013a}. This computational obstruction fundamentally changes the relationship between estimation and testing, and makes it impossible to directly transfer many matrix-based intuitions to the tensor setting.
\end{enumerate}

These two observations motivate a fundamental question: \emph{What is the optimal computationally feasible approach to estimating high-order cumulants in tensor spectral norm?} To probe the computational barriers underlying this estimation task, we study an associated detection problem. Our results uncover a rich computational--statistical structure that differs qualitatively from both covariance estimation and standard planted-signal problems; see Sections~\ref{sec:related-work} and \ref{sec_rev_gap} for further discussion.

\paragraph{Overview of contributions.}

We summarize our main results below, with precise statements and assumptions provided in the corresponding theorems in the subsequent sections.

\smallskip
\noindent\textbf{(i) Statistical limits.}
We establish minimax lower bounds for estimating $\bcK_d$ in tensor spectral norm over broad classes of sub-Gaussian distributions. Specifically, the optimal statistical rate is
\[
\sqrt{p/n}\wedge 1,
\]
up to constants depending only on $d$. We then show that the sample cumulant tensor, namely the plug-in estimator obtained from the sample moment tensors through the moment--cumulant relation (formally defined in Section \ref{sec:notation}), generally fails to achieve this rate for $d\ge 3$: under $\beta$-subexponential tails with $0<\beta\le d$, its operator-norm error is bounded, with high probability,
\[
\max\Big\{\frac{p^{d/\beta}}{n},\ \sqrt{\frac{p}{n}}\Big\}.
\]
Moreover, we construct a statistically optimal cumulant estimator, showing that the minimax benchmark $\sqrt{p/n}$ is statistically attainable. Finally, we show that detection and consistent estimation have the same statistical limit. This establishes the statistical landscape of the problem.

\smallskip
\noindent\textbf{(ii) A ``reverse'' detection--estimation gap.}
We next turn to the computational aspect of the problem. To understand when efficient procedures may or may not exist, we study the associated detection problem: test whether the $d$-th order cumulant tensor is bounded away from zero in tensor spectral norm. We show, under the low-degree Conjecture~\ref{conj_cumulant_low_degree}, that detection is computationally hard in the regime
\[
n \ll p^{d/2}. 
\]

Beyond this fixed-signal regime, our low-degree analysis further identifies a family of testing problems in which the alternative is separated from the null, in tensor spectral norm, by
\[
\sqrt{\frac{p^{d/2}}{n\operatorname{polylog}(n)}},
\]
yet, under the same conjecture, no polynomial-time test distinguishes the two hypotheses.
In contrast, the sample cumulant tensor is an order-$d$ polynomial estimator, and under sub-Gaussian assumptions achieves operator-norm error
\[
p^{d/2}/n,
\qquad\text{when } n \ll p^{d-1}.
\]
Ignoring logarithmic factors, we have
\[
    \frac{p^{d/2}}{n}
    \ll
    \sqrt{\frac{p^{d/2}}{n\operatorname{polylog}(n)}}
\]
whenever $n$ is polynomially larger than $p^{d/2}$. Thus, in the intermediate regime $p^{d/2}\ll n\ll p^{d-1}$, when $n$ exceeds $p^{d/2}$ by a polynomial factor, the estimation error of a polynomial-time estimator is asymptotically smaller than the separation at which polynomial-time detection remains hard under the low-degree conjecture. This establishes a \emph{reverse detection--estimation gap: computationally efficient detection can be strictly harder than computationally efficient estimation}. This is in contrast to the more familiar detection--recovery gap in planted problems, where recovery is typically the harder task \cite{schramm2022computational,luoComputationalLowerBounds2024,mao2023detection,bresler2023detection}.

The underlying reason is that an efficiently computable estimator does not automatically yield an efficiently computable test in our setting. Indeed, turning an estimator into a test requires certifying the size of the estimated tensor in spectral norm, and computing or approximating that norm is NP-hard in general. Thus, the usual reduction from estimation to testing does not preserve computational efficiency here.

\smallskip
\noindent\textbf{(iii) Open problems.}
Our results leave several interesting questions open. Most notably, we conjecture that the low-degree polynomial threshold $n \asymp p^{d/2}$ marks the onset of efficient consistent estimation, and that in the intermediate regime $p^{d/2} \ll n \ll p^{d-1}$, the computationally optimal estimation error is strictly larger than the statistical minimax rate. If true, this would yield a nearly complete computational--statistical phase diagram for cumulant tensor estimation in spectral norm.

What makes these questions unusual is that the main difficulty does not arise solely from the size of the parameter space or from the complexity of the underlying distribution class. Rather, it is built into the loss function itself. In many open problems on computationally constrained estimation, the target loss is easy to evaluate once an estimator is given, and the main challenge is to determine whether one can compute an estimator attaining the statistically optimal rate. Here, by contrast, the natural loss---the tensor spectral norm---is itself NP-hard to compute or even approximate in general. As a result, the standard reduction from estimation to testing is no longer algorithmically transparent: even if one can efficiently construct an estimator, it may still be computationally intractable to certify that the estimator is accurate in the relevant norm.

\smallskip
\noindent
{\bf (iv) Non-Gaussian lower-bound constructions for cumulants.}
A key technical contribution is the construction of least favorable non-Gaussian distributions for cumulant tensor estimation. Gaussian alternatives, though standard in minimax lower-bound arguments, are inapplicable here since all cumulants of order three and higher vanish under Gaussianity. We instead design non-Gaussian perturbations that induce controlled high-order cumulants while keeping the resulting distributions statistically close and preserving the required tail and structural assumptions.

\bigskip

These results suggest the phase diagram in Figure \ref{fig_1}.
The computational regimes in this informal summary suppress logarithmic factors and use polynomial separation between the displayed scales.
Detection-hardness statements are conditional on the model-specific low-degree Conjecture~\ref{conj_cumulant_low_degree}.

\begin{enumerate}
    \item When $n\ll p$, consistent estimation and detection are both statistically impossible.

    \item When $p\ll n$, consistent estimation and detection are statistically possible, but the computational perspective depends further on the sample size:
    \begin{enumerate}
        \item When $p\ll n\ll p^{d/2}$, we prove that detection is computationally infeasible under a model-specific low-degree conjecture, and we separately conjecture that consistent estimation is computationally infeasible.

        \item When $p^{d/2}\ll n\ll p^{d-1}$, there exist computationally efficient procedures for detection and consistent estimation, and we conjecture that the computationally optimal estimation error is ${p^{d/2}} / {n}$, which is strictly worse than the minimax statistical rate.

        \item When $p^{d-1}\ll n$, the sample cumulant tensor already achieves the statistically optimal rate $\sqrt{p / n}$.
    \end{enumerate}
\end{enumerate}

\begin{figure}
    \centering
    \includegraphics[width=0.9\linewidth]{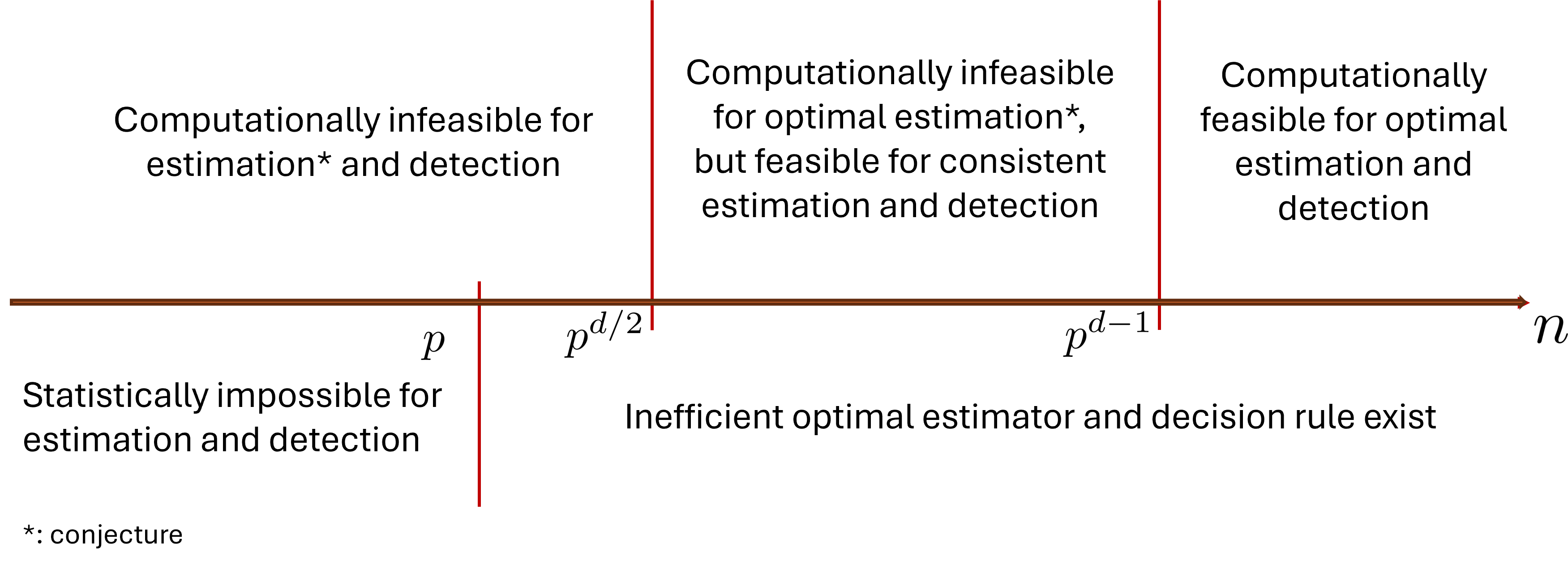}
    \caption{Phase diagram of statistical and computational limits in cumulant tensor estimation and detection}
    \label{fig_1}
\end{figure}

The rest of the paper is organized as follows. Section \ref{sec:related-work} reviews related work. Section \ref{sec:notation} introduces notation and preliminaries. Section \ref{sec_statistical_limit} establishes the statistical limits of cumulant tensor estimation and detection, including minimax lower bounds, upper bounds for sample cumulants, and a statistically optimal estimator. Section \ref{sec_computational_limit} proves low-degree likelihood-ratio bounds, derives conditional computational detection lower bounds, and develops polynomial-time testing procedures based on Frobenius-norm surrogates. Section \ref{sec_open_problem} discusses the resulting open problems and conjectures for cumulant estimation. Section \ref{sec_moment_extensions} summarizes how the results extend to moment tensors, with formal moment statements deferred to Appendix~\ref{app_moment_results}.
The proofs deferred from the main text are collected in Appendix~\ref{app_proofs}.

\subsection{Related Work}\label{sec:related-work}

We next discuss several lines of work related to this paper.

\paragraph{Detection--estimation gap}
The relationship among detection, estimation, and recovery has long been studied from both classical statistical and more recent computational perspectives. For example, \cite{ingster2010detection} develops sharp minimax detection boundaries in sparse regression and shows that procedures based on estimating the regression coefficients in $\ell_2$ norm generally require stronger signals than those needed for optimal detection. In sparse heterogeneous mixtures, \cite{donoho2004higher} characterizes the sharp detection boundary in the rare/weak regime and shows that higher criticism attains this boundary adaptively, thereby illustrating that global detection can succeed even when the underlying sparse signals are too weak for reliable identification of the affected coordinates. In network models, \cite{decelle2011asymptotic} predicts detectability phase transitions in the stochastic block model, while \cite{abbe2016exact,abbe2015community} establishes sharp thresholds for exact and partial recovery, again highlighting that different inferential goals can have different thresholds. More recently, \cite{bresler2023detection,mao2023detection} studies computational gaps between detection and stronger inferential tasks, including recovery and refutation, in planted-signal problems. Most of the works concern settings in which recovery or estimation is harder than detection; by contrast, our setting exhibits the opposite phenomenon, where detection can be strictly harder than estimation.

\paragraph{Independent component analysis (ICA)}

In the ICA model, one observes a random vector of the form $X = AS$, where $A$ is an unknown deterministic mixing matrix and $S$ has entrywise independent coordinates. A key step in many ICA algorithms is to exploit \emph{fourth-order} information, since after whitening, non-Gaussianity is often identified through deviations of the 4th moment (or 4th cumulant) from that of a Gaussian distribution. Consequently, both detection and estimation in ICA are intimately connected to the estimation of 4th moments and cumulants.

Recent works \cite{auddy2025large,szekely_learning_2024} establish sharp computational--statistical gaps for ICA: when $n<p$, reliable detection is statistically impossible, whereas for $n<p^2$ it becomes computationally hard under broad algorithmic frameworks. These thresholds are consistent with the phase transitions suggested by our general theory when specialized to $d=4$.

Despite this close connection, ICA fundamentally differs from the setting considered here, since the relevant 4th-order tensor in ICA exhibits strong algebraic structure induced by independence and the mixing model. In particular, the structured nature of the parameter often enables reductions from estimation to detection. In contrast, our work studies high-order cumulants in a fully general (nonparametric) setting without imposing structural constraints such as independence, low rank, or sparsity. Technically, our low-degree polynomial lower bound in Section~\ref{sec_hypothesis_test} can be viewed as a generalization of the low-degree polynomial construction developed in \cite{szekely_learning_2024}.

\paragraph{Tensor PCA} 

Tensor PCA concerns the estimation of a rank-one signal tensor embedded in noise. A canonical formulation assumes an observation model $Y = X + Z,$ 
where $X = a^{\otimes d}$ and $a\in\R^p$ is an unknown vector and $Z$ is a noise tensor (e.g., Gaussian). The computational hardness of tensor PCA has been extensively studied using a variety of frameworks, including reductions \citep{zhang2018tensor}, the low-degree polynomial framework \citep{kunisky2019notes}, memory- or communication-bounded models \citep{dudeja_statistical-computational_2024}, and the statistical query framework \citep{dudeja2021statistical}. These works collectively demonstrate the existence of a broad computationally hard regime for detection and estimation when the signal strength is below the algorithmic threshold, typically scaling as $\|X\| \ll p^{d/4}$ in the rank-one spiked tensor model.

Our setting differs in that we do not assume a spiked model or low-rank structure for the moment/cumulant tensor. 
Tensor PCA provides useful intuition: if one imposes a rank-one structure on the order-$d$ cumulant tensor (i.e., $\bcK_d \propto u^{\otimes d}$), then the tensor PCA detection threshold corresponds to $n \ll p^{d/2}$ in our sample-based formulation. As in ICA, the rank-one assumption enables a direct reduction from estimation to detection.

\paragraph{Robust covariance estimation} Robust covariance estimation under weak moment or heavy-tailed assumptions has been studied extensively.
\cite{catoni2016pac} develops PAC-Bayesian and $M$-estimation type procedures that estimate the covariance operator via uniform control of quadratic forms.
\cite{wei2017estimation} proposes computationally tractable covariance estimators for heavy-tailed distributions with strong nonasymptotic guarantees, complementing truncation- and median-of-means based approaches.
\cite{minsker2018sub} constructs sub-Gaussian covariance estimators achieving sharp spectral-norm guarantees under weak tail assumptions, demonstrating how truncation and robust mean ideas yield optimal operator-norm performance in the matrix setting.
\cite{keUserFriendlyCovarianceEstimation2019} provides practically implementable, heavy-tail-robust covariance estimators with user-friendly tuning and nonasymptotic spectral-norm bounds.
\cite{mendelson2020robust} establishes sharp robust covariance estimation under an $L_4$--$L_2$ norm equivalence (finite-kurtosis type) condition.

\paragraph{Polynomial chaos and concentration} Beyond second-order structure, higher-order cumulant and moment tensors are closely tied to polynomial chaoses.
The decoupling inequalities of \cite{kwapien1987decoupling} form a classical tool for reducing dependent multilinear forms to independent counterparts in high-degree chaos analysis.
\cite{latala2006estimates} proves two-sided moment and tail bounds for Gaussian chaoses of arbitrary order, a foundational reference for modern polynomial and tensor concentration results.
\cite{adamczak2015concentration} derives concentration inequalities for non-Lipschitz functions with bounded higher derivatives, yielding multilevel bounds for polynomials beyond the Lipschitz-concentration regime.
\cite{gotzeConcentrationInequalitiesPolynomials2021} establishes general concentration inequalities for polynomial functionals of independent random variables, providing tools directly relevant to high-degree moment and tensor analysis. Most closely related to our analysis of empirical tensors, \cite{al-ghattas_sharp_2025} establishes sharp concentration bounds for simple random tensors, or equivalently empirical high-order moment tensors, under Gaussian and sub-Gaussian assumptions. Their results provide an important benchmark for the operator-norm behavior of sample moment tensors, while our work focuses on cumulant tensors and their statistical and computational limits under tensor spectral norm.

\paragraph{Moment methods} Moment-based methods are also central to statistical inference in Gaussian mixture models (GMMs). 
\cite{WuYang2020GaussianMixturesDMM} gives a denoised method-of-moments procedure via projection to truncated moment space and proves statistical guarantees and adaptive optimality. 
\cite{doss2020optimal} highlights the role of moment tensors of the mixing distribution. 
\cite{pereira_tensor_2022} derives explicit formulas and discusses efficient implicit computations for GMM moment tensors. 
\cite{khouja2022tensor} proves that parameters of spherical Gaussian mixtures can be learned from empirical moments by forming symmetric moment tensors, and gives a linear-algebra-based decomposition algorithm. 
\cite{kottler2025method} studies how to recover Gaussian-mixture parameters in practice using the method of moments.

\paragraph{Computational hardness} 
From the algorithmic perspective, tensor norm computation and related inference problems are known to be computationally hard.
\cite{hillarMostTensorProblems2013a} shows that many basic tensor problems, including deciding and approximating the tensor spectral norm, are NP-hard.
\cite{hopkins2015tensor} analyzes Tensor PCA through sum-of-squares (SoS) proofs, demonstrating nontrivial recovery and certification guarantees and characterizing SoS power in spiked tensor models.
\cite{hopkins2017power} develops a general theory connecting SoS methods to spectral algorithms built from low-degree polynomial matrices for detecting planted structure, including tensor problems.
\cite{barak2016noisy} applies SoS relaxations to noisy tensor completion, illustrating how SoS can certify polynomial objectives closely related to tensor norms.
\cite{potechin2020machinery} provides general machinery for proving SoS lower bounds and derives SoS lower bounds for Tensor PCA, complementing upper-bound results on random tensors.
\cite{wein2023average} studies average-case complexity for random tensor decomposition within the low-degree framework, identifying regimes where low-degree methods succeed or fail.
\cite{brennan2020statistical} shows a close connection between statistical query algorithms and low-degree tests, supporting low-degree as a computational model for many high-dimensional testing problems.
\cite{dudeja2021statistical} establishes statistical query lower bounds for Tensor PCA, providing hardness evidence under another widely used restricted-algorithm framework.
\cite{hopkins2017efficient} develops a low-degree/SoS perspective for certain Bayesian inference tasks, including community detection. \cite{luo2020open} highlights hypergraphic planted clique detection as a central average-case hardness conjecture for tensor problems and raises the question of whether its hardness is computationally equivalent to that of classical planted clique detection.
Our computational lower-bound construction is closely related to moment-matched NGCA, for which statistical-query and sum-of-squares lower bounds have been established \citep{diakonikolas2023sqngca,diakonikolas2024sosngca}. Unlike these works, we study general cumulant tensor estimation and leverage detection hardness to establish a conditional reverse detection--estimation gap under the tensor spectral norm.

Unlike previous studies of empirical moment tensors or structured higher-order inference problems, our analysis focuses on cumulant tensors over general distribution classes without structural assumptions. The moment–cumulant transformation introduces additional cancellations that require separate analysis, while the computational hardness of tensor spectral norm gives rise to a qualitatively different relationship between detection and estimation.

\section{Preliminaries and Notation}\label{sec:notation}

\paragraph{Notation}

We use $c, C, C'$, etc.\ to denote positive constants whose values may vary from line to line, and use subscripts to indicate their dependence on parameters. For example, $C_{\beta,d}$ denotes a constant that depends only on $\beta$ and $d$, and not on any other parameters such as $p$. 
For a positive integer $d$, write $[d]=\{1,\ldots,d\}$. Vectors are denoted by lowercase letters, matrices by uppercase letters, and tensors by calligraphic letters such as $\bcT\in\R^{p_1\times\cdots\times p_d}$. For $x\in\R^p$, let $\|x\|$ be its Euclidean norm, and let $\mathbb S^{p-1}=\{x\in\R^p:\|x\|=1\}$ be the unit sphere. For any random variable or random vector $X$, $Y$, we write $X \overset{d}{\sim} Y$ if they have the same distribution. For a matrix $A$, let $\|A\|$ and $\|A\|_F$ denote its spectral norm and Frobenius norm, respectively. For a tensor $\bcT\in(\R^p)^{\otimes d}$, its spectral norm is
\[
\|\bcT\|:=\sup_{u_1,\ldots,u_d\in\mathbb S^{p-1}}\langle \bcT, u_1\otimes\cdots\otimes u_d\rangle. 
\]
We write $\|\bcT\|_F$ for the tensor Frobenius norm.    
For a random vector $X\in\R^p$ and an integer $r\ge 1$, let $\bcM_r=\E[X^{\otimes r}]$ denote the $r$-th moment tensor and let $\bcK_r$ denote the $r$-th cumulant tensor. For indices $i_1,\ldots,i_r\in[p]$, the corresponding entries are written as $(\bcM_r)_{i_1,\ldots,i_r}$ and $(\bcK_r)_{i_1,\ldots,i_r}$. 

For a real-valued random variable $Y$ and $\beta>0$, define the Orlicz norm by 
\[
    \|Y\|_{\Phi_\beta}:=\inf\{t>0:\E\exp(|Y|^\beta/t^\beta)\le 2\}. 
\]
We call a random vector $X\in\R^p$ $\beta$-subexponential for some $\beta>0$ if
\[
\|v^\top X\|_{\Phi_\beta}
:=\inf\Big\{t>0:\ \mathbb{E}\exp\big(|v^\top X|^\beta/t^\beta\big)\le 2\Big\}
\le 1,
\qquad \forall v \in \mathbb{S}^{p-1}.
\]
In particular, $\beta=2$ corresponds to the sub-Gaussian case and $\beta=1$ to the sub-exponential case.

\paragraph{Sample cumulant tensor} 
We briefly review the definitions of moment and cumulant tensors. 
Let $X=(X_1,\ldots,X_p)^\top\in\R^p$ be a random vector with finite moments up to order $d$.
For {$1\le r\le d$}, the $r$-th moment tensor is $\bcM_r=\E[X^{\otimes r}]\in (\R^{p})^{\otimes r}$, 
whose entries are $(\bcM_r)_{i_1,\ldots,i_r}=\E[X_{i_1}\cdots X_{i_r}]$. 

We define the cumulant tensors recursively from the moment tensors through the moment--cumulant relationship.
Specifically, set $\bcK_1:=\bcM_1$ and, for $r=2,\ldots,d$, define
\[
\bcK_r:=\bcM_r-G_r(\bcK_1,\ldots,\bcK_{r-1}),
\]
where $G_r$ is the universal tensor-valued polynomial formed by the terms involving two or more blocks in the moment--cumulant formula given in Section~\ref{sec_moment_extensions}.
Each such term involves only cumulants of orders strictly smaller than $r$.
When the moment generating function exists in a neighborhood of the origin, these tensors agree with the derivatives of the cumulant generating function at zero.
In particular, $\bcK_1=\E[X]$, and when $X$ is centered, $\bcK_2$ and $\bcK_3$ coincide with the second and third central moment tensors.

Equivalently, the moment--cumulant relationship can be written using tensor-valued polynomials $F_d$ and $G_d$, depending only on order $d$, as
\begin{equation}\label{eq_relation_cumulant_moment}
\bcK_d = \bcM_d + F_d(\bcM_1,\ldots,\bcM_{d-1}),
\qquad
\bcM_d = \bcK_d + G_d(\bcK_1,\ldots,\bcK_{d-1}).
\end{equation}
Given i.i.d.\ samples $X_1,\ldots,X_n$, let
\[
\widehat{\bcM}_r = \frac{1}{n}\sum_{i=1}^n X_i^{\otimes r}, \qquad 1\le r\le d,
\]
denote the sample moment tensors. We then define the order-$d$ \emph{sample cumulant tensor} by the plug-in rule
\begin{equation}\label{eq_sample_cumulant_tensor}
\widehat{\bcK}_d = \widehat{\bcM}_d + F_d(\widehat{\bcM}_1,\ldots,\widehat{\bcM}_{d-1}).
\end{equation}

\section{Statistical Limits}\label{sec_statistical_limit} 

\subsection{Lower Bounds for Estimation and Detection}\label{sec_lower_bound} 

We begin by establishing the statistical minimax lower bound for estimating high-order cumulant tensors in spectral norm.

\begin{theorem}[Estimation lower bound]
\label{thm_lower_bound_cumulant_general_d}
Let $d \geq 3$ and $p > C_1$ for some absolute constant $C_1>0$.
Define the class of sub-Gaussian random vectors as
\[
\mathcal{F}
=
\Big\{
X \in \mathbb{R}^p :
\|v^\top X\|_{\Phi_2} \leq 2,\ \forall v \in \mathbb{S}^{p-1}
\Big\}. 
\]
Then for any estimator based on $n$ i.i.d.\ samples, we have
\[
\inf_{\widehat{\bcK}} \sup_{X \in \mathcal{F}}
\E \big\| \widehat{\bcK} - \bcK_d(X) \big\|
\;\ge\;
c_d \Big( \sqrt{\frac{p}{n}} \wedge 1 \Big),
\]
where $\bcK_d(X)$ denotes the $d$-th order cumulant tensor of $X$ and $c_d>0$ is a constant depending only on $d$.
\end{theorem}
Note that all cumulants of order three and higher vanish under Gaussianity. Thus, the proof of Theorem~\ref{thm_lower_bound_cumulant_general_d} requires a non-Gaussian construction. 
To this end, we construct bounded perturbations of a Gaussian density that match every Gaussian moment of order below $d$ and produce an exactly specified rank-one perturbation of the order-$d$ moment and cumulant tensors.
The complete proof is given in Supplement Section~\ref{sec_proof_lower_bound}. 

Next, we turn to the associated cumulant detection problem
\beas
H_0:\ \bcK_d = 0
\qquad\text{vs.}\qquad
H_1:\ \|\bcK_d\| \ge \rho.
\eeas
Theorem~\ref{thm_lower_bound_cumulant_general_d} shows that consistent estimation is impossible when $n\lesssim p$. We now show that the same scaling is also necessary for reliable testing: if $n\ll p$, then no test can distinguish the null from alternatives separated by less than order $\sqrt{p/n}\wedge 1$ in spectral norm.

\begin{theorem}[Detection lower bound]
\label{thm_lower_bound_detection}
Let $d\ge 3$, and let
\[
\mathcal F
=
\Big\{
X\in\R^p:\ \|v^\top X\|_{\Phi_2}\le 2,\ \forall v\in \mathbb S^{p-1}
\Big\}.
\]
For $\rho>0$, define the cumulant testing problem
\[
H_0^{(K)}:\ \bcK_d(X)=0
\qquad \text{vs.}\qquad
H_1^{(K)}(\rho):\ \|\bcK_d(X)\|\ge \rho,
\]
with $X\in\mathcal F$. Let
\[
\mathfrak R_n^{(K)}(\rho)
:=
\inf_{\psi}
\left\{
\sup_{X\in\mathcal F:\,\bcK_d(X)=0}\PP_X^{\otimes n}(\psi=1)
+
\sup_{X\in\mathcal F:\,\|\bcK_d(X)\|\ge \rho}\PP_X^{\otimes n}(\psi=0)
\right\},
\]
where the infimum is over all tests $\psi\in\{0,1\}$.
Then there exist constants $c_d,\eta_d>0$, depending only on $d$, such that
\[
\mathfrak R_n^{(K)}(\rho)\ge \eta_d
\qquad\text{whenever}\qquad
\rho \le c_d\Big(\sqrt{p/n}\wedge 1\Big).
\]
\end{theorem}

Thus, cumulant detection and consistent cumulant estimation share the same statistical lower bound: up to constants, the critical statistical scale for both tasks is $\sqrt{p/n}\wedge 1$.

\subsection{Sample Cumulant Tensors Can Be Suboptimal}\label{sec_sample_moments}

We next study the most natural baseline estimator for cumulant tensors in the absence of additional structural assumptions: the sample cumulant tensor $\widehat{\bcK}_d$ defined in Section~\ref{sec:notation} by \eqref{eq_sample_cumulant_tensor}. The following result gives a spectral-norm deviation bound for this estimator.

\begin{theorem}[Upper bound for sample cumulant tensors]\label{thm_upper_bound_sample_cumulants}
Let $X_1,\ldots,X_n \stackrel{\mathrm{iid}}{\sim} X \in \mathbb{R}^p$ be a $\beta$-subexponential random vector for some $0<\beta\le d$, and let $\widehat{\bcK}_d$ be the sample cumulant tensor defined in \eqref{eq_sample_cumulant_tensor}. Then, with probability at least $1-\exp(-c_{\beta,d}p)$,
\[
\big\|\widehat{\bcK}_d-\bcK_d\big\|
\le
C_{\beta,d}\max\Big\{\frac{p^{d/\beta}}{n},\ \sqrt{\frac{p}{n}}\Big\},
\]
where $C_{\beta,d},c_{\beta,d}>0$ depend only on $\beta$ and $d$.
\end{theorem}

The bound contains two terms. The term $\sqrt{p/n}$ is the classical spectral-norm fluctuation term. The term $p^{d/\beta}/n$, however, is specific to higher-order empirical tensors and can be much larger when $d\ge 3$. We show in the next theorem that this term is not an artifact of the proof. 

\begin{theorem}[Lower bound for sample cumulant tensors]
\label{thm_sample_cumulant_lower_general_beta}
Fix an integer $d\ge3$, $0<\beta\le d$, and $\zeta>0$.
There exist constants $c_{\beta,d,\zeta},C_d>0$ such that, for all sufficiently large $p$ and $n$ greater than $C_d$, there is a $\beta$-subexponential random vector $X\in\R^p$ satisfying
\begin{equation}\label{eq_sample_cumulant_lower_general_beta}
\PP_X^{\otimes n}\left(
\|\widehat{\bcK}_d-\bcK_d(X)\|
\ge c_{\beta,d,\zeta}\max\left\{\frac{p^{d/\beta}}n,\sqrt{\frac pn}\right\}
\right)
>e^{-\zeta p}.
\end{equation}
\end{theorem}

Theorem~\ref{thm_sample_cumulant_lower_general_beta} matches the upper bound in Theorem~\ref{thm_upper_bound_sample_cumulants}, but only at an exponentially small probability level. It does not imply the same lower bound in expectation. In fact, such an expectation bound cannot hold in general when $0<\beta<2$. Under the $\beta$-subexponential assumption, we have $\E\|X\|^d\le C_{\beta,d}p^{d/2}$. By the moment--cumulant formula and H\"older's inequality, this further implies $\sup_X\E\|\widehat{\bcK}_d-\bcK_d(X)\|\le C_{\beta,d}p^{d/2}$, uniformly in $n$, where the supremum is over all $\beta$-subexponential distributions. For fixed $n$, this is $o(p^{d/\beta}/n)$ as $p\to\infty$. The next theorem gives an expectation lower bound that recovers the rate $p^{d/\beta}/n$ when $\beta\ge2$ and establishes a different rate when $\beta<2$.

\begin{theorem}[Expectation lower bound for sample cumulant tensors]
\label{prop_sample_cumulant_expectation_lower_general_beta}
Fix an integer $d\ge3$ and $0<\beta\le d$.
There exist constants $c_{\beta,d},C_d>0$ such that, for all sufficiently large $p$ and $n$ greater than $C_d$, there is a $\beta$-subexponential random vector $X\in\R^p$ satisfying
\begin{equation}\label{eq_sample_cumulant_expectation_lower_general_beta}
\E\|\widehat{\bcK}_d-\bcK_d(X)\|
\ge
\frac{c_{\beta,d}}n
\max\left\{
p^{d/(\beta\vee2)}[\log(1+n)]^{d(1/\beta-1/2)_+},
[\log(1+n)]^{d/\beta}
\right\}.
\end{equation}
Here $x_+=\max\{x,0\}$.
\end{theorem}

Theorem~\ref{prop_sample_cumulant_expectation_lower_general_beta} gives two different lower bounds depending on $\beta$. Assuming $\log(1+n)\lesssim p$, it gives
\[
\E\|\widehat{\bcK}_d-\bcK_d(X)\|
\gtrsim
\begin{cases}
p^{d/2}[\log(1+n)]^{d/\beta-d/2}/n, & 0<\beta<2,\\
p^{d/\beta}/n, & 2\le\beta\le d.
\end{cases}
\]
Consequently, consistency of the sample cumulant estimator in expected spectral norm requires these rates to vanish. 
In the intermediate regime $p\ll n\ll p^{2d/(\beta\vee2)-1}$ for $0<\beta<d$, these lower bounds are asymptotically larger than $\sqrt{p/n}$, demonstrating the statistical suboptimality of the sample cumulant estimator. In particular, for $0<\beta\le2$, this regime reduces to $p\ll n\ll p^{d-1}$.

We finally make two remarks. First, this phenomenon is closely related to the behavior of higher-order sample moment tensors. For instance, \cite{al-ghattas_sharp_2025,al-ghattas_estimation_2025} show that the sample moment tensor for Gaussian observations has spectral-norm error at least of order $\max\{p^{d/2}/n,\sqrt{p/n}\}$. Similar deviation behavior also appears in sharp lower bounds for scalar sub-Weibull sums \citep{zhang2020lower,kuchibhotla2022moving,bong2023tight}. However, these results do not directly imply the corresponding bounds for sample cumulants, since contributions from lower-order sample moments may cancel the leading fluctuation. Our Bernoulli constructions in the proof explicitly control these contributions. See Supplement Section~\ref{sec_proof_sample_cumulant_lower} for details.

Second, the contrast with covariance estimation is worth emphasizing. When $d=2$, the sample cumulant is the sample covariance tensor, whose spectral-norm error under sub-Gaussian tails has the classical rate $\max\{p/n,\sqrt{p/n}\}$; in particular, it is $\sqrt{p/n}$ when $n\gtrsim p$ \citep{koltchinskii2017concentration}. For higher-order cumulants, however, Theorem~\ref{prop_sample_cumulant_expectation_lower_general_beta} shows that the additional term $p^{d/2}/n$ cannot generally be removed. This confirms that the sample cumulant estimator can be statistically suboptimal. 

\subsection{Statistically Optimal Estimator}\label{sec_optimal_estimator}

We now construct a cumulant estimator whose operator-norm error matches the minimax lower bound in Theorem~\ref{thm_lower_bound_cumulant_general_d}. 
The key idea is to avoid directly using the empirical moment tensors, whose operator norm can be too large in high dimensions. Instead, for each moment order $r\le d$, we first estimate only one-dimensional diagonal moment functionals $\langle \bcM_r,u^{\otimes r}\rangle$ on a finite net of the sphere, and then recover a tensor consistent with these directional estimates through a convex feasibility problem. The cumulant estimator is finally obtained by plugging these regularized moment estimators into the moment--cumulant identity.

Let
\[
    \bcM_r=\E X^{\otimes r},\qquad r=1,\ldots,d,
\]
be the population moment tensors. For each $r\le d$, let $\mathbb T_r^{\rm sym}$ denote the space of symmetric order-$r$ tensors in $(\mathbb R^p)^{\otimes r}$. We construct an estimator $\widehat{\bcM}_r^\star\in\mathbb T_r^{\rm sym}$ of $\bcM_r$ as follows.

First fix a unit vector $u\in\mathbb S^{p-1}$. We estimate the scalar diagonal moment
\[
    m_{r,u}:=\langle \bcM_r,u^{\otimes r}\rangle
    =
    \E\langle X,u\rangle^r.
\]
The estimator of $m_{r,u}$ is chosen according to the sample-size regime.
Fix $\zeta=3+\log(1+4d)$ and a constant $c_{\beta,\zeta}\ge2$.
For $n>p$, define the truncation radius by
\[
    R_r=\big(c_{\beta,\zeta}\log(en/p)\big)^{1/\beta}.
\]
Set
\[
\widetilde m_{r,u}
=
\begin{cases}
\displaystyle
\frac1n\sum_{i=1}^n \langle X_i,u\rangle^r,
& n>\max\{p,p^{2r/\beta-1}\}, \\[0.9em]
\displaystyle
\frac1n\sum_{i=1}^n \langle X_i,u\rangle^r
\mathbf 1\{|\langle X_i,u\rangle|\le R_r\},
& p<n\le p^{2r/\beta-1}, \\[0.9em]
0,
& n\le p.
\end{cases}
\]
Thus, in the large-sample regime we use the ordinary empirical directional moment; in the intermediate regime we truncate the one-dimensional projection to control heavy tails with exponentially small failure probability; and in the low-sample regime the target rate is constant, so the zero estimator is sufficient.

Supplement Theorem~\ref{thm_rate_op_moment_multilinear_form_estimator}, establishes the following concentration of $\widetilde m_{r,u}$: 
for each fixed $r\le d$ and $u\in\mathbb S^{p-1}$, with probability at least $1-2\exp(-\zeta p)$,
\[
    |\widetilde m_{r,u}-m_{r,u}|
    \le
    \tau_r:=C_{\beta,r,d}\big(\sqrt{p/n}\wedge1\big),
\]
where $C_{\beta,r,d}$ is a sufficiently large constant depending only on $\beta,r,d$.

We now turn these directional estimates into a tensor estimator. Let $\mathcal N_r$ be an $\varepsilon_r$-net of $\mathbb S^{p-1}$ with
\[
    \varepsilon_r=\frac{1}{2r},
    \qquad
    |\mathcal N_r|\le (1+4r)^p.
\]
Define the feasible set
\[
\mathcal S_r
=
\Big\{
\bcA\in\mathbb T_r^{\rm sym}:
\big|
\langle \bcA,u^{\otimes r}\rangle
-
\widetilde m_{r,u}
\big|
\le \tau_r,
\quad \forall u\in\mathcal N_r
\Big\}.
\]
We then define the regularized moment estimator as the minimum-Frobenius-norm element of this feasible set:
\[
    \widehat{\bcM}_r^\star
    \in
    \argmin_{\bcA\in\mathcal S_r}
    \|\bcA\|_F^2.
\]
If $\mathcal S_r$ is empty, we set $\widehat{\bcM}_r^\star=0$ by convention. 

We finally define the cumulant estimator by plugging these regularized moment estimators into the identity \eqref{eq_relation_cumulant_moment}:
\[
    \widehat{\bcK}_d^\star
    =
    \widehat{\bcM}_d^\star
    +
    F_d(\widehat{\bcM}_1^\star,\ldots,\widehat{\bcM}_{d-1}^\star).
\]

\begin{theorem}[Statistically optimal cumulant estimator]\label{thm_rate_op_cumulant_estimator}
Fix an integer $d\ge3$ and $\beta>0$.
Let $X_1,\ldots,X_n\stackrel{\mathrm{iid}}{\sim}X\in\mathbb R^p$ be a $\beta$-subexponential random vector, and let $\widehat{\bcK}_d^\star$ be the cumulant estimator constructed above with $\tau_r=C_{\beta,r,d}(\sqrt{p/n}\wedge1)$.
Then, for all $p,n\ge1$, with probability at least $1-C_d\exp(-c_{\beta,d}p)$,
\[
    \|\widehat{\bcK}_d^\star-\bcK_d\|
    \le C_{\beta,d}\big(\sqrt{p/n}\wedge1\big).
\]
Moreover, for all $p,n\ge1$,
\[
    \E\|\widehat{\bcK}_d^\star-\bcK_d\|
    \le C_{\beta,d}\big(\sqrt{p/n}\wedge1\big).
\]
\end{theorem}

The expected-risk bound in Theorem~\ref{thm_rate_op_cumulant_estimator} matches the lower bound in Theorem~\ref{thm_lower_bound_cumulant_general_d} for $\beta \leq 2$.
Thus, the estimator attains the exact statistical minimax rate.

However, the estimator is not computationally efficient in high dimensions, because the convex program contains $|\mathcal N_r|=\exp(O_r(p))$ constraints for each $r\le d$.
This motivates the computational questions studied next in Section~\ref{sec_computational_limit}.

\subsection{Statistically Optimal Detection}\label{sec_statistical_detection}

We now turn from estimation to detection. In general, consistent estimation implies statistical detectability; see, for example, page 80 of \cite{tsybakov2009introduction} and page 95 of \cite{rigollet2023high}. Indeed, if an estimator $\hat\mu$ consistently estimates a scalar parameter $\mu$, then the test
\[
H_0:\ \mu=0
\qquad\text{vs.}\qquad
H_1:\ |\mu|>c
\]
can be carried out using the rule $\xi=\mathbb I_{\{|\hat\mu|>c/2\}}$. Consistency of $\hat\mu$ guarantees that both the type I and type II errors vanish asymptotically.

Applying this principle to our setting, the statistical upper bounds for cumulant estimation in Theorem~\ref{thm_rate_op_cumulant_estimator} immediately imply detectability whenever $n\gg p$. In particular, the testing problem
\beas
H_0:\ \bcK_d = 0
\qquad\text{vs.}\qquad
H_1:\ \|\bcK_d\| \ge \kappa
\eeas
is statistically solvable in the regime $n\gg p$. Combined with the lower bounds in Theorems~\ref{thm_lower_bound_cumulant_general_d} and \ref{thm_lower_bound_detection}, this establishes that cumulant detection and consistent cumulant estimation have the same statistical threshold. 

\section{Computational Limits}\label{sec_computational_limit}

In the previous section, we characterized the statistical limits of estimating $d$-th order cumulant tensors and obtained the minimax statistically optimal rate $\sqrt{p/n}\wedge1$. Achieving this rate, however, requires solving a convex optimization problem with exponentially many constraints, which renders the resulting estimator computationally infeasible in high-dimensional settings. In the following sections, we investigate the computational limits of cumulant tensor detection and estimation. 

\subsection{Low-degree Polynomial Framework}\label{sec_ldp_framework}

To study computational complexity, we adopt the low-degree polynomial (LDP) framework. Originating from the analysis of the sum-of-squares (SoS) hierarchy \citep{hopkins2017power, hopkins2017efficient, barak2019nearly}, the LDP framework has become a central tool for characterizing computational barriers in high-dimensional statistics. It has been successfully applied to a wide range of problems, including planted clique detection \citep{barak2019nearly}, tensor PCA \citep{hopkins2018statistical}, tensor regression \citep{luo2024tensor}, among many others. See \cite[Section 2]{maleki2026highdimensional} for a survey.

Consider a hypothesis testing problem with null distribution $\mathbb Q$ and alternative distribution $\mathbb P$. Under standard regularity conditions, the likelihood ratio test is asymptotically optimal by the Neyman--Pearson lemma. The LDP framework studies the projection of the likelihood ratio $L:= \frac{d\mathbb P}{d\mathbb Q}$ onto the space of polynomials of degree at most $D$. Specifically, it considers the norm $\|L_{\leq D}\| := \| \mathcal{P}_{\leq D} L\|, $
where $\mathcal P_{\le D}$ denotes the orthogonal projection onto degree-$D$ polynomials and the norm is defined by $\|f\| := (\E_{X\sim \mathbb Q}[f(X)^2])^{1/2}$. 

Intuitively, $L_{\le D}$ captures the optimal degree-$D$ polynomial statistic for distinguishing $\mathbb P$ from $\mathbb Q$, and the quantity $\|L_{\le D}\|$ measures the extent to which LDPs can separate the two distributions. 
When $\|L_{\le D}\|^2=O(1)$ as the dimension grows, every polynomial $f$ of degree at most $D$ satisfies
\[
\big|\mathbb E_{\mathbb P}f-\mathbb E_{\mathbb Q}f\big|
\le
\sqrt{\|L_{\le D}\|^2-1}\,
\sqrt{\operatorname{Var}_{\mathbb Q}(f)}.
\]
Thus, no such polynomial can achieve a diverging mean difference relative to its standard deviation under $\mathbb Q$. 
Under the low-degree polynomial conjecture \citep{hopkins2018statistical,kunisky2019notes}, boundedness of $\|L_{\le D}\|$ for sufficiently large polylogarithmic degree $D$ further suggests that no polynomial-time algorithm can strongly distinguish $\mathbb P$ from $\mathbb Q$.

A broad class of algorithmic paradigms---including spectral methods, approximate message passing, statistical query algorithms, and other polynomial-time procedures---have been shown to be captured by LDPs. Consequently, the LDP framework provides a rigorous characterization of the computational limits of these algorithmic classes \citep{wein2025computational}. 
Recent counterexamples show that bounded low-degree likelihood ratio norms do not, by themselves, rule out all efficient algorithms in complete generality \citep{buhai2025quasi}. Nevertheless, the LDP framework continues to serve as a robust organizing principle for understanding statistical-computational tradeoffs across a broad range of structured inference problems \citep{holmgren2020counterexamples,buhai2025quasi,wein2025computational}.
We refer the interested reader to the surveys \citep{kunisky2019notes, wein2025computational,maleki2026highdimensional} for comprehensive treatments of the LDP framework and its implications.

\subsection{Computational Detection Lower Bound}\label{sec_hypothesis_test}

In this subsection, we apply the LDP framework described in Section \ref{sec_ldp_framework} to study the computational hardness of the following detection problem: given i.i.d. samples from a random vector, determine whether its $d$-th order cumulant tensor is zero. 

To this end, we construct a family of distributions whose lower-order cumulants match those of a Gaussian distribution, while exhibiting a nontrivial $d$-th order cumulant. This construction ensures that any algorithm must exploit genuinely high-order information to distinguish the alternative from the null. Specifically, let 
\[
    X_i = \sqrt{a/p} W_i U + Z_i,\quad  i\in[n],  
\]
where $Z_i\overset{\mathrm{iid}}{\sim}\mathcal N(0,I_p)$, $W_i$ are i.i.d. random variables drawn from the distribution constructed in Lemma~\ref{lemma_existance_1d}, and $U$ is drawn uniformly from $\{\pm1\}^p$, independently of all $W_i$ and $Z_i$. The same unobserved direction $U$ is shared across all observations, which are i.i.d. conditional on $U$. The parameter $a>0$ is known and may depend on $p,n$.
Conditional on $U$, the cumulants of $X_i$ satisfy
\[
\bcK_1(X_i\mid U)=0,\qquad
\bcK_2(X_i\mid U)=I_p+\frac{a}{p}UU^\top,
\]
\[
\bcK_m(X_i\mid U)=0\quad (3\le m<d),\qquad
\bcK_d(X_i\mid U)=c_d\left(\frac{a}{p}\right)^{d/2}U^{\otimes d},
\]
where $c_d=\kappa_d(W_i)\ne0$ depends only on $d$. Let
\[
S=I_p-\frac{a}{1+a+\sqrt{1+a}}\frac{UU^\top}{p}
\]
be the whitening matrix for the conditional covariance of $X_i$. Setting $Y_i=SX_i$, the first $d-1$ conditional cumulants of $Y_i$ coincide with those of a standard Gaussian vector, while
\[
\bcK_d(Y_i\mid U)=c_d\left(\frac{a}{p}\right)^{d/2}(SU)^{\otimes d}.
\]
Since $\|SU\|^2=p/(1+a)$, we obtain
\[
\|\bcK_d(Y_i\mid U)\|
=|c_d|\left(\frac{a}{1+a}\right)^{d/2}.
\]
Thus, for fixed $a>0$, the conditional cumulant tensor has a positive spectral norm depending only on $d$ and $a$.

For each $u\in\{\pm1\}^p$, let $P_{u,a}$ denote the law of $S_u(\sqrt{a/p}\,Wu+Z)$, where $S_u$ is the whitening matrix $S$ with $U=u$. We consider the following testing problem for the joint distribution of the observations:
\begin{equation}\label{eq_H1H0}
\begin{aligned}
H_0:\quad &(Y_1,\ldots,Y_n)\sim\mathbb Q_{p,n}:=\mathcal N(0,I_p)^{\otimes n},\\
H_1:\quad &(Y_1,\ldots,Y_n)\sim\mathbb P_{p,n,a}:=2^{-p}\sum_{u\in\{\pm1\}^p}P_{u,a}^{\otimes n}.
\end{aligned}
\end{equation}
The alternative is a mixture of conditional i.i.d. distributions, with the same latent direction shared across observations. The nonzero cumulant tensor above belongs to each conditional law $P_{u,a}$. Let $L=d\mathbb P_{p,n,a}/d\mathbb Q_{p,n}$. The low-degree likelihood ratio is controlled by the following lemma.
\begin{lemma}[{LDP norm upper bound}]\label{lemma_ldp} 
$\|L_{\leq D}\|^2 \leq \sum_{0\leq m \leq D} \(\frac{a}{1+a} \frac{C_d m^4 n^{1/d}}{p^{1/2}}\)^{m}. $
\end{lemma}
As a consequence, we obtain the following low-degree bound.
\begin{theorem}[Low-degree likelihood-ratio bound]
\label{thm:ldp-hardness-detection}
Fix $d \ge 3$ and consider the testing problem \eqref{eq_H1H0}. 
Then there exists a constant $c_d>0$, depending only on $d$, such that for any $\varepsilon>0$, if
\[
p \ge n^{2/d+\varepsilon},
\qquad
D \le c_d\, n^{\varepsilon/8},
\]
then the degree-$D$ low-degree likelihood ratio satisfies
\[
\|L_{\le D}\| = O(1).
\]
\end{theorem}
Following the low-degree conjecture commonly used in the literature \citep{hopkins2018statistical,kunisky2019notes,wein2025computational}, we formulate a model-specific version for our testing problem.
\begin{conjecture}[Model-specific low-degree conjecture]\label{conj_cumulant_low_degree}
Fix $d\ge3$ and the scalar distribution of $W$ above. There exists a constant $c_d^{\mathrm{LD}}>0$ such that, for any parameter sequences $n=n_p\le p^C$ and $a=a_p>0$, if
\[
\big\|(L_p)_{\le D_p}\big\|=O(1),
\qquad
D_p=\lfloor c_d^{\mathrm{LD}}(\log p)^2\rfloor,
\]
where $L_p=d\mathbb P_{p,n_p,a_p}/d\mathbb Q_{p,n_p}$, then no sequence of tests generated by a single, possibly randomized polynomial-time algorithm satisfies
\[
\mathbb Q_{p,n_p}(\phi_p=1)+\mathbb P_{p,n_p,a_p}(\phi_p=0)\to0.
\]
\end{conjecture}

\begin{remark}
For the testing problem considered here, the conjecture is further supported by statistical-query and sum-of-squares lower bounds for closely related moment-matched NGCA models \citep{diakonikolas2023sqngca,diakonikolas2024sosngca}.
\end{remark}

\begin{corollary}[Conditional computational detection lower bound]\label{cor_cumulant_detection_conditional}
Assume Conjecture~\ref{conj_cumulant_low_degree}. Fix $d\ge3$, $a>0$, and $\varepsilon>0$. If $p\ge n^{2/d+\varepsilon}$, then no polynomial-time test strongly distinguishes the null and alternative in \eqref{eq_H1H0}. 
\end{corollary}

\subsection{A ``Reverse'' Detection--Estimation Gap}\label{sec_rev_gap}

A useful distinction in this area is between \emph{detection} and \emph{estimation} (or recovery). 
In many planted or signal-plus-noise models, these two tasks can exhibit different computational thresholds. 
Recent work in the LDP framework has made this distinction precise. 
In particular, \cite{schramm2022computational} developed general LDP lower bounds for estimation and recovery, and showed that for problems such as planted submatrix and planted dense subgraph, estimation can be computationally hard even when detection is easy. \cite{luoComputationalLowerBounds2024} identified a setting in which subtensor detection is easier than estimation. See also \cite{mao2023detection} for an explicit detection--estimation gap in planted dense cycles, and \cite{bresler2023detection} for reduction-based evidence of such a gap in planted dense subgraph. 
By contrast, in some models the two thresholds coincide; for example, in the spiked Wigner model, detection and estimation undergo the same phase transition \citep{el2020fundamental}. 

In the above examples, estimation is shown to be harder than detection, which is consistent with the intuition that recovering a signal should require more information than merely detecting its presence.
The phenomenon described below goes in the \emph{opposite} direction.
We identify a regime in which an efficiently computable estimator achieves an error asymptotically smaller than a signal scale at which the low-degree likelihood-ratio norm remains bounded.
Under Conjecture~\ref{conj_cumulant_low_degree}, this yields a regime in which computationally efficient detection is strictly harder than computationally efficient estimation.

Fix $0<\varepsilon<d/2-1$ and consider $n=\lfloor p^{d/2+\varepsilon}\rfloor$, so that $p<n<p^{d-1}$ for sufficiently large $p$.
In Lemma~\ref{lemma_ldp}, choose
\[
a=b_d\frac{p^{1/2}}{n^{1/d}(\log n)^8},
\]
where $b_d>0$ is a sufficiently small constant depending only on $d$.
For $D_p=\lfloor c_d^{\mathrm{LD}}(\log p)^2\rfloor$ as in Conjecture~\ref{conj_cumulant_low_degree}, we have
\[
\frac{a}{1+a}\frac{C_dD_p^4n^{1/d}}{p^{1/2}}
\le C_db_d(c_d^{\mathrm{LD}})^4
\left(\frac{\log p}{\log n}\right)^8
\le\frac12
\]
for sufficiently small $b_d$.
Consequently, Lemma~\ref{lemma_ldp} gives $\|L_{\le D_p}\|^2\le2$.

Since $a\to0$, the conditional cumulant signal satisfies
\[
D_{\operatorname{det}}
:=\|\bcK_d(Y_i\mid U)\|
=|c_d|\left(\frac{a}{1+a}\right)^{d/2}
\asymp_d\sqrt{\frac{p^{d/2}}{n(\log n)^{8d}}}.
\]
On the other hand, Theorem~\ref{thm_upper_bound_sample_cumulants}, after rescaling by a uniform sub-Gaussian constant depending only on $d$, shows that the sample cumulant tensor satisfies the spectral-norm error bound
\[
\|\widehat{\bcK}_d-\bcK_d\|
\lesssim_d\frac{p^{d/2}}{n}
:=D_{\operatorname{est}}
\]
with probability tending to one under the null and uniformly over the conditional alternatives.
In particular,
\[
\frac{D_{\operatorname{est}}}{D_{\operatorname{det}}}
=O_{d,\varepsilon}\big(p^{-\varepsilon/2}(\log p)^{4d}\big)
\longrightarrow0.
\]
Thus, both scales vanish, and the sample cumulant tensor is an efficiently computable estimator whose error is asymptotically smaller than the alternative signal.

\begin{corollary}[Conditional reverse detection--estimation gap]\label{cor_reverse_gap_conditional}
For the above choice of $n$ and $a$, the sample cumulant tensor is a polynomial-time estimator with spectral-norm error at most $D_{\operatorname{est}}=o(D_{\operatorname{det}})$ with probability tending to one under the null and uniformly over the conditional alternatives.
Under Conjecture~\ref{conj_cumulant_low_degree}, no polynomial-time test strongly distinguishes the null and alternative in \eqref{eq_H1H0} at the signal scale $D_{\operatorname{det}}$.
\end{corollary}

This may initially appear counterintuitive. The key point is that we are comparing \emph{computationally efficient} procedures. 
Although the estimator $\widehat{\bcK}_d$ can be computed efficiently, converting it into an efficient detector is not straightforward. 
A natural decision rule would require certifying whether $\|\widehat{\bcK}_d\|$ exceeds a threshold, or more generally whether $\widehat{\bcK}_d$ is close to the true tensor in tensor spectral norm. 
However, these certification steps involve tensor spectral norms, whose exact computation is NP-hard in general \citep{hillarMostTensorProblems2013a}. 
Thus an efficiently computable estimator does not automatically yield an efficiently computable test. 
See Section \ref{sec_open_problem} for further discussion.

Finally, if computational constraints are removed, estimation remains at least as hard as detection, as discussed in Section \ref{sec_statistical_detection}: any estimator with sufficiently small error can be converted into a test (see, for example, page 80 of \cite{tsybakov2009introduction} and page 95 of \cite{rigollet2023high}).
The point here is precisely that such a reduction need not preserve computational efficiency.

Our subsequent work \citep{tang2026reversegaps} further investigates how reverse detection--estimation gaps can be used to establish computational lower bounds for norm approximation. While the present paper focuses on the statistical and computational limits of cumulant tensor inference and identifies the reverse gap, that work develops a general reduction and establishes additional reverse gaps for other matrix and tensor norms.

\subsection{Computationally Efficient Detection of Cumulants}
\label{sec_detection_methods}

We now construct a computationally efficient test for detecting a nonzero cumulant tensor. Let $Y_1,\ldots,Y_n$ be i.i.d.\ copies of a centered and whitened random vector $Y\in\R^p$, satisfying $\E Y=0$ and $\E[YY^\top]=I_p$. Let $\bcK_r$ denote the $r$-th cumulant tensor of $Y$. For a prescribed separation $\rho>0$, we consider
\begin{equation}\label{eq_test_cumulant}
H_0:\ \bcK_d=0
\qquad\text{vs.}\qquad
H_1(\rho):\ \|\bcK_d\|\ge\rho.
\end{equation}

To avoid computing the tensor spectral norm, we construct a test based on estimating $\|\bcK_d\|_F^2$. Since $\|\bcK_d\|_F\ge\|\bcK_d\|$, the Frobenius norm provides a suitable surrogate for detection.

Assuming that the lower-order cumulants $\bcK_1,\ldots,\bcK_{d-1}$ are known and fixed under both hypotheses, we can construct the known degree-$d$ polynomial feature map by relationship \eqref{eq_relation_cumulant_moment},
\[
\Psi_d(y)
=
y^{\otimes d}
-
G_d(\bcK_1,\ldots,\bcK_{d-1}),
\]
which satisfies
\begin{equation}\label{eq_feature_unbiased_cumulant}
\E[\Psi_d(Y)]=\bcK_d.
\end{equation}
More generally, the following analysis applies to any known polynomial feature map satisfying the identity \eqref{eq_feature_unbiased_cumulant}.

Define the U-statistic
\begin{equation}\label{eq_Un_cumulant}
V_n^{(\mathrm{cumu})}
:=
\frac{1}{n(n-1)}
\sum_{i\neq j}
\langle \Psi_d(Y_i),\Psi_d(Y_j)\rangle.
\end{equation}
By independence,
\[
\E V_n^{(\mathrm{cumu})}
=
\Big\langle \E[\Psi_d(Y_1)],\E[\Psi_d(Y_2)]\Big\rangle
=
\|\bcK_d\|_F^2.
\]
The following theorem establishes the detection guarantee of the resulting polynomial-time test.

\begin{theorem}[Polynomial-time consistent test for vanishing $d$-th order cumulant tensor]\label{thm_poly_test_cumulant}
Fix an integer $d\ge3$ and let $Y_1,\ldots,Y_n$, $n\ge2$, be i.i.d.\ copies of a centered and whitened random vector $Y\in\R^p$ with finite moments through order $d$.
Let $\Psi_d$ be a known degree-$d$ tensor-valued polynomial map satisfying \eqref{eq_feature_unbiased_cumulant} and $\E\|\Psi_d(Y)\|_F^2<\infty$.
For the testing problem \eqref{eq_test_cumulant}, define
$
\phi_{n,\rho}^{(\mathrm{cumu})}
=
\mathbf1_{\{V_n^{(\mathrm{cumu})}>\rho^2/2\}},
$
where $V_n^{(\mathrm{cumu})}$ is given by \eqref{eq_Un_cumulant}.
Assume 
\begin{equation}\label{eq_kernel_second_moment_cumulant}
\E\Big[\langle \Psi_d(Y_1),\Psi_d(Y_2)\rangle^2\Big]
\le C_0 p^d
\end{equation}
holds with the same constant $C_0>0$ under both hypotheses.
Then
\begin{equation}\label{eq_shrinking_cumulant_testing_risk}
\sup_{H_0}\PP\big(\phi_{n,\rho}^{(\mathrm{cumu})}=1\big)
+
\sup_{H_1(\rho)}\PP\big(\phi_{n,\rho}^{(\mathrm{cumu})}=0\big)
\le
C\left\{\frac{p^{d/2}}{n\rho^2}+\frac{p^d}{n^2\rho^4}\right\},
\end{equation}
where $C$ depends only on $C_0$, and the suprema range over distributions satisfying the stated assumptions.
\end{theorem}

\begin{remark}\label{remark_assumption_cumulant}
For the canonical feature map defined above, the second-moment condition \eqref{eq_kernel_second_moment_cumulant} holds under the sub-Gaussian assumption, with $C_0$ depending only on $d$ and the sub-Gaussian constant.
In particular, it holds uniformly for the conditional alternative distributions constructed in Section~\ref{sec_hypothesis_test}. 
Moreover, our computational lower-bound construction also satisfies the assumption of known lower-order cumulants, since $\bcK_1=0$, $\bcK_2=I_p$, and $\bcK_m=0$ for $3\le m<d$ under both hypotheses. 
Thus, the upper and conditional lower bounds apply to the same family of distributions.
\end{remark}

Theorem~\ref{thm_poly_test_cumulant} shows that, under the stated assumptions, computationally efficient detection is possible for separations $\rho_{p,n}\gg p^{d/4}/\sqrt{n}$, allowing the signal strength to vanish.
In particular, for any fixed separation $\rho>0$, consistent detection is achievable whenever $n\gg p^{d/2}$.
This complements the low-degree lower bound in Section~\ref{sec_hypothesis_test} conditional on Conjecture~\ref{conj_cumulant_low_degree}.

\section{Open Problems}\label{sec_open_problem}

The previous section gave low-degree evidence and conditional lower bounds for cumulant detection, together with a polynomial-time upper bound. A natural next question is to understand the computational limits of cumulant estimation, and in particular to establish computational lower bounds. Compared with detection, however, this question is substantially more subtle. Existing approaches from the literature do not directly apply here, because estimation under tensor spectral norm has several features that differ qualitatively from the standard settings in which computational lower bounds are well understood. As a result, the computational complexity of cumulant estimation in our problem remains largely open.

\subsection{Consistent Estimation}

In the literature, the computational hardness of estimation problems is often established by reduction to the corresponding hypothesis testing problem. This reduction can be made formal (see, e.g., Section~17 of \cite{brennanReducibilityStatisticalComputationalGaps2020}), but the underlying logic is straightforward and the same as the one discussed in Section \ref{sec_statistical_detection}: if an estimator $\hat\mu$ consistently estimates a scalar parameter $\mu$, then one can test 
$H_0: \mu = 0$ vs. $H_1: \mu > c$ 
using the decision rule $\xi = \mathbb I_{\{|\hat\mu| > c/2\}}$. Consistency of $\hat\mu$ guarantees that $\xi$ converges in probability to the correct decision, implying that the computational complexity of testing provides a lower bound for that of estimation.

This reasoning, however, does not directly apply to cumulant tensor estimation. The key obstruction is that the natural loss function of interest---the tensor spectral norm---is itself NP-hard to compute or even approximate. Consequently, even if a consistent estimator $\widehat{\bcK}$ of a cumulant tensor $\bcK$ can be computed efficiently, the corresponding decision rule $\xi = I_{\{\|\widehat \bcK\| > c/2\}}$ may remain computationally intractable. 

Despite this difficulty, there is a striking alignment between the computational threshold for hypothesis testing established in Section~\ref{sec_hypothesis_test} and the upper bound in Theorem~\ref{thm_upper_bound_sample_cumulants}. 
Specifically, when $n \gg p^{d/2}$, the sample cumulant tensor provides a computationally efficient consistent estimator, whereas when $n \ll p^{d/2}$, cumulant detection becomes computationally hard. 
This motivates the following conjecture.

\begin{conjecture}[Computational limit of uniform consistency]\label{conj_consistent}
Fix an integer $d\ge3$, and let $\mathcal F_p$ denote the class of probability distributions $P$ on $\R^p$ satisfying $\sup_{u\in\mathbb S^{p-1}}\|u^\top X\|_{\Phi_2}\le2$ for $X\sim P$.
For any sequence of positive integers $n=n_p$ satisfying $n_p=o(p^{d/2})$ as $p\to\infty$, there is no estimator $\widehat{\bcK}_{p,n}$ generated by a single, possibly randomized algorithm running in time polynomial in $(n,p)$ that is uniformly consistent for $\bcK_d(P)$ over $\mathcal F_p$ in tensor spectral norm.
\end{conjecture}

Here, uniform consistency means that, for every fixed $\varepsilon>0$,
\[
\sup_{P\in\mathcal F_p}
\PP_P\left(
\left\|\widehat{\bcK}_{p,n_p}-\bcK_d(P)\right\|>\varepsilon
\right)
\longrightarrow0
\qquad\text{as }p\to\infty.
\]

Proving Conjecture~\ref{conj_consistent} is challenging, primarily due to the tensor spectral norm. To the best of our knowledge, very few works directly study computational lower bounds for estimation problems involving losses that are themselves hard to compute, making average-case reductions difficult to construct. Moreover, existing frameworks for establishing computational hardness of estimation---such as the LDP or statistical query frameworks---typically require the loss to decompose into a sum of coordinate-wise losses (e.g., the Frobenius norm), and therefore do not apply here. Additional obstacles arise from the fact that the parameter space has dimension $p^d$, which complicates the application of memory- or communication-constrained models. 

If Conjecture~\ref{conj_consistent} holds, then combined with the results of Section~\ref{sec_sample_moments}, it yields a complete computational-statistical phase diagram for consistent cumulant estimation. 
\begin{itemize}
\item When $n \ll p$, the problem is statistically impossible. 
\item When $p \ll n \ll p^{d/2}$, consistent estimation is statistically possible but computationally intractable. 
\item When $n \gg p^{d/2}$, the sample cumulant tensor is a computationally efficient consistent estimator. 
\end{itemize}

One piece of indirect evidence supporting Conjecture \ref{conj_consistent} is that, if one attempts to adapt the spectrally truncated moment estimator, which is consistent for heavy-tailed covariance estimation whenever $n \gg p$, then a naive bias--variance decomposition for higher-order cumulants fails to deliver the desired rate. Specifically, the truncation bias and the stochastic fluctuation cannot be controlled simultaneously at the target scale. 

\subsection{Optimal Estimation}

Recall from Section~\ref{sec_statistical_limit} that the minimax-optimal error rate for estimating $d$-th order cumulant tensors in spectral norm is $\sqrt{p/n}$. In contrast, when $p^{d/2} \ll n \ll p^{d-1}$, the sample cumulant tensor estimator attains an error rate of $p^{d/2}/n$, which is strictly suboptimal. In many statistical problems, once consistent estimation becomes possible, one can further refine the estimator to achieve the optimal rate---often by using a consistent estimator as initialization for an iterative refinement procedure. Examples of this paradigm appear in tensor problems with low-rank structure \citep{zhang2018tensor, tang2025revisit}.

In the cumulant estimation setting, however, we do not impose any additional structural constraints such as low rank, which significantly limits opportunities for algorithmic refinement beyond the sample cumulant tensor. A related line of work concerns covariance estimation in spectral norm under heavy-tailed distributions, where suitably truncated sample covariance estimators are known to be optimal \citep{wei2017estimation,keUserFriendlyCovarianceEstimation2019}. These techniques rely crucially on spectral properties of matrices and do not extend to higher-order tensors, highlighting a fundamental gap in current methodology.

On the other hand, there are notable exceptions where a parameter can be consistently estimated but cannot be estimated at the optimal rate by any efficient algorithm \citep{luoComputationalLowerBounds2024, zhangLowerBoundsPerformance2014}. This leads us to the following conjecture.

\begin{conjecture}[Computationally optimal estimation rate]\label{conj_optimal}
Fix an integer $d\ge3$, and let $\mathcal F_p$ be the class of sub-Gaussian distributions defined in Conjecture~\ref{conj_consistent}.
For any polynomial-time computable sequence of positive integers $n=n_p$ satisfying $p^{d/2}\ll n_p\ll p^{d-1}$, there is no estimator $\widehat{\bcK}_{p,n}$ generated by a single, possibly randomized algorithm running in time polynomial in $(n,p)$ such that
\[
\sup_{P\in\mathcal F_p}\E_P\left\|\widehat{\bcK}_{p,n_p}-\bcK_d(P)\right\|=o\left(\frac{p^{d/2}}{n_p}\right)
\]
as $p\to\infty$.
\end{conjecture}

The next proposition shows that the conjectured computational limit for consistent estimation also determines the optimal estimation rate in the intermediate regime.

\begin{proposition}[Connection between estimation conjectures]\label{prop_cumulant_estimation_reduction}
If Conjecture~\ref{conj_consistent} holds, then Conjecture~\ref{conj_optimal} also holds.
\end{proposition}

The proof uses a reduction from rate improvement to consistent estimation, without passing through detection or evaluating tensor spectral norms.
See Appendix~\ref{app_proof_cumulant_reduction} for details.

\section{Extensions to Moment Tensors}\label{sec_moment_extensions}

Although this work focuses on cumulant tensors, most of the results have direct moment analogues. To make this transfer explicit, we now recall the general partition formulas behind the moment--cumulant relation stated in Section~\ref{sec:notation}.

Let $\mathcal P([d])$ denote the collection of all partitions of $[d]$. A partition $\pi\in\mathcal P([d])$ is a collection of nonempty, pairwise disjoint subsets $\pi=\{B_1,\ldots,B_m\}$, called the blocks of $\pi$, whose union is $[d]$. For a block $B=\{j_1,\ldots,j_r\}\subset[d]$, we write $i_B:=(i_{j_1},\ldots,i_{j_r})$, where the indices are taken in increasing order. The moment--cumulant relationship used to define the cumulants in Section~\ref{sec:notation} is \citep{speed1983cumulants}
\[
(\bcM_d)_{i_1,\ldots,i_d}
=
\sum_{\pi\in\mathcal P([d])}
\prod_{B\in\pi} (\bcK_{|B|})_{i_B}.
\]
The partition with a single block contributes $\bcK_d$; the remaining partitions give $G_d(\bcK_1,\ldots,\bcK_{d-1})$ in the recursive definition.
Applying M\"obius inversion to this relationship gives the equivalent expression \citep{speed1983cumulants}
\begin{equation}\label{eq_cumulant_moment_inversion}
(\bcK_d)_{i_1,\ldots,i_d}
=
\sum_{\pi\in\mathcal P([d])}
\mu(\pi)\prod_{B\in\pi} (\bcM_{|B|})_{i_B},
\qquad
\mu(\pi)=(-1)^{|\pi|-1}(|\pi|-1)!.
\end{equation}

In particular, these partition formulas imply the polynomial identities in \eqref{eq_relation_cumulant_moment}, showing that the $d$-th moment and cumulant tensors differ by polynomials involving only lower-order tensors.
The upper deviation bounds and the net-based statistically optimal estimator for moment tensors follow from analogous arguments.
For the minimax estimation lower bound, we use the same Gaussian-perturbation construction as for cumulants, which preserves all lower-order moments and produces the same separation in the $d$-th moment and cumulant tensors.
The detection lower bound similarly extends to odd-order moments, while the zero-moment null is degenerate for even orders.
The formal moment results and their proofs are collected in Appendix~\ref{app_moment_results} and~\ref{app_proofs}, respectively.

\section{Discussion}

In this paper, we studied the estimation and detection of high-order cumulant tensors under tensor spectral norm, and uncovered a computational--statistical landscape that is qualitatively different from the matrix case and from more classical planted problems. On the statistical side, we showed that the minimax-optimal estimation rate is $\sqrt{p/n}\wedge 1$, while the naive sample cumulant tensor is generally suboptimal for $d\ge 3$. 
On the computational side, we establish low-degree likelihood-ratio bounds.
Under a model-specific low-degree conjecture, these bounds imply computational detection hardness and a reverse detection--estimation gap.

Our results leave open the computational complexity of cumulant estimation itself. In many high-dimensional problems, computational lower bounds for estimation are derived through reductions from testing. In the present setting, however, such reductions are no longer algorithmically transparent, because the relevant loss function is the tensor spectral norm, which is NP-hard to compute or approximate in general. As a result, even if an estimator can be computed efficiently, certifying that it achieves small error in the target norm may itself be computationally intractable.

This suggests several directions for future work. A first goal is to determine whether the threshold $n\asymp p^{d/2}$ indeed marks the onset of computationally efficient consistent cumulant estimation. A second goal is to characterize the computationally optimal estimation error in the intermediate regime $p^{d/2}\ll n\ll p^{d-1}$, where the statistical minimax rate and the best known efficient estimators do not match. More broadly, our results point to a new class of open problems in which the computational difficulty is driven not only by the statistical model, but also by the intrinsic hardness of the loss used to evaluate performance.

\section*{Acknowledgments}

The authors thank Wei Biao Wu and Cheng Mao for helpful discussions.

\bibliography{reference}
\bibliographystyle{plain}

\appendix

\clearpage

\startcontents[supplement]

\section*{Supplementary Material}

This supplement contains additional results for moment tensors and the proofs deferred from the main text.
Section~\ref{app_moment_results} states the moment-tensor analogues of our estimation and detection results.
These statements are presented first because several subsequent proofs establish the moment and cumulant versions together.
Section~\ref{app_proofs} contains the deferred proofs, with subsection titles identifying the corresponding theorems or lemmas.
Section~\ref{app_lemmas} collects auxiliary lemmas.

\subsection*{Contents}
\printcontents[supplement]{}{1}[2]{}

\clearpage

\section{Formal Results for Moment Tensors}\label{app_moment_results}

This appendix collects the formal moment-tensor analogues. Throughout this section, $\bcM_d=\E X^{\otimes d}$ denotes the $d$-th moment tensor. 

\begin{theorem}[Lower bound for estimating $d$-th order moment tensors]
\label{thm_moment_lower_bound_appendix}
Let $d \geq 3$ and $p>C_1$ for some absolute constant $C_1>0$.
Let
\[
\mathcal{F}
=
\Big\{
X\in\mathbb{R}^p:
\|v^\top X\|_{\Phi_2}\le 2,\ \forall v\in\mathbb S^{p-1}
\Big\}.
\]
Then for any estimator based on $n$ i.i.d.\ samples,
\[
\inf_{\widehat{\bcM}}\sup_{X\in\mathcal F}
\E\big\|\widehat{\bcM}-\E X^{\otimes d}\big\|
\ge
c_d\Big(\sqrt{\frac pn}\wedge 1\Big),
\]
where $c_d>0$ depends only on $d$.
\end{theorem}

\begin{theorem}[Statistical lower bound for odd-order moment detection]
\label{thm_moment_detection_appendix}
Let $d\ge3$ be odd, and let
\[
\mathcal F
=
\Big\{
X\in\R^p:\ \|v^\top X\|_{\Phi_2}\le 2,\ \forall v\in \mathbb S^{p-1}
\Big\}.
\]
For $\rho>0$, consider
\[
H_0^{(M)}:\ \E X^{\otimes d}=0
\qquad \text{vs.}\qquad
H_1^{(M)}(\rho):\ \|\E X^{\otimes d}\|\ge \rho,
\]
with $X\in\mathcal F$. Let $\mathfrak R_n^{(M)}(\rho)$ be the corresponding minimax testing risk. Then there exist constants $c_d,\eta_d>0$, depending only on $d$, such that
\[
\mathfrak R_n^{(M)}(\rho)\ge \eta_d
\qquad\text{whenever}\qquad
\rho\le c_d\Big(\sqrt{p/n}\wedge 1\Big).
\]
\end{theorem}

\begin{remark}[Even-order moment detection is degenerate]
If $d$ is even and $\E X^{\otimes d}=0$, then $X=0$ almost surely. Hence the even-order moment testing problem with null $\bcM_d=0$ is degenerate.
\end{remark}

\begin{theorem}[Upper bound for sample moment tensors]
\label{thm_upper_bound_sample_moments_appendix}
Let $X_1,\ldots,X_n \stackrel{\mathrm{iid}}{\sim} X \in \mathbb{R}^p$ be a $\beta$-subexponential random vector for some $0<\beta\le d$. 
Then, with probability at least $1-\exp(-c_{\beta,d}p)$, 
\[
\Big\|\frac{1}{n}\sum_{i=1}^n X_i^{\otimes d} - \mathbb{E}[X^{\otimes d}]\Big\|
\le
C_{\beta,d}\max\Big\{\frac{p^{d/\beta}}{n},\ \sqrt{\frac{p}{n}}\Big\},
\]
and
\[
\Big\|\frac{1}{n}\sum_{i=1}^n (X_i-\bar X)^{\otimes d}
- \mathbb{E}[(X-\mu)^{\otimes d}]\Big\|
\le
C_{\beta,d}\max\Big\{\frac{p^{d/\beta}}{n},\ \sqrt{\frac{p}{n}}\Big\},
\]
where $\bar X=n^{-1}\sum_{i=1}^n X_i$ and $\mu=\mathbb E X$.
\end{theorem}

\begin{theorem}[Estimation of a multilinear moment form]
\label{thm_rate_op_moment_multilinear_form_estimator}
Fix an integer $d\ge1$, $\beta>0$, and $\zeta>0$.
Let $X_1,\ldots,X_n\stackrel{\mathrm{iid}}{\sim}X\in\mathbb R^p$ be a $\beta$-subexponential random vector, let $\bcM=\E X^{\otimes d}$, and fix $u_1,\ldots,u_d\in\mathbb S^{p-1}$.
Define
\[
\widetilde\bcM(u_1,\ldots,u_d)=
\begin{cases}
\displaystyle \frac1n\sum_{i=1}^n X_i^{\otimes d},
& n>\max\{p,p^{2d/\beta-1}\},\\[0.6em]
\displaystyle \frac1n\sum_{i=1}^n\bigotimes_{j=1}^d
\big(X_i I_{\{|\langle X_i,u_j\rangle|\le R\}}\big),
& p<n\le p^{2d/\beta-1},\\[0.6em]
0,&n\le p,
\end{cases}
\]
where $R=(c_{\beta,\zeta}\log(en/p))^{1/\beta}$ is defined for $n>p$, and $c_{\beta,\zeta}\ge2$ is fixed independently of $p,n$.
Then, with probability at least $1-2\exp(-\zeta p)$,
\[
\left|\left\langle\widetilde\bcM(u_1,\ldots,u_d)-\bcM,
\bigotimes_{j=1}^d u_j\right\rangle\right|
\le C_{\beta,d,\zeta}\big(\sqrt{p/n}\wedge1\big),
\]
where $C_{\beta,d,\zeta}$ depends only on $\beta,d,\zeta$.
\end{theorem}

\begin{theorem}[Convex formulation for a statistically optimal moment estimator]
\label{thm_rate_op_moment_estimator_appendix}
Suppose the conditions in Theorem~\ref{thm_rate_op_moment_multilinear_form_estimator} hold, with $\zeta=3+\log(1+4d)$.
Let $\mathcal N$ be a deterministic $1/(2d)$-net of $\mathbb S^{p-1}$ satisfying $|\mathcal N|\le(1+4d)^p$, and let $\mathbb T_d^{\rm sym}$ be the space of symmetric order-$d$ tensors.
For $u\in\mathcal N$, set $\widetilde m_u=\langle\widetilde\bcM(u,\ldots,u),u^{\otimes d}\rangle$ and $\tau=A_{\beta,d}(\sqrt{p/n}\wedge1)$, where $A_{\beta,d}$ is sufficiently large.
Define
\[
\mathcal S=
\left\{\bcA\in\mathbb T_d^{\rm sym}:
|\langle\bcA,u^{\otimes d}\rangle-\widetilde m_u|\le\tau,
\quad u\in\mathcal N\right\}.
\]
If $\mathcal S\ne\varnothing$, let
\[
\widehat\bcM=\argmin_{\bcA\in\mathcal S}\|\bcA\|_F^2;
\]
otherwise set $\widehat\bcM=0$.
Then, for all $p,n\ge1$, with probability at least $1-2\exp(-c_{\beta,d}p)$,
\[
\|\widehat\bcM-\E X^{\otimes d}\|
\le C_{\beta,d}(\sqrt{p/n}\wedge1).
\]
Moreover, for all $p,n\ge1$ and every fixed $q\ge1$,
\[
\E\|\widehat\bcM-\E X^{\otimes d}\|^q
\le C_{\beta,d,q}(\sqrt{p/n}\wedge1)^q.
\]
All constants depend only on $\beta,d$ (and $q$).
\end{theorem}

\begin{theorem}[Polynomial-time consistent test for vanishing $d$-th moment]
\label{thm_poly_test_moment_corrected}
Suppose $d$ is odd. Let $X_1,\ldots,X_n$ be i.i.d.\ copies of a random vector $X\in\R^p$ with sub-Gaussian norm $K$ for some constant $K>0$. Consider
\[
H_0:\ \bcM_d=\E[X^{\otimes d}]=0
\qquad\text{vs.}\qquad
H_1:\ \|\bcM_d\|\ge \kappa,
\]
where $\kappa>0$ is fixed, and define
\[
V_n:=\frac{1}{n(n-1)}\sum_{i\neq j}(X_i^\top X_j)^d.
\]
Assume in addition that every alternative under $H_1$ satisfies
\begin{equation}\label{eq_zeta1_assumption_new}
\Var \big(\langle \bcM_d, X^{\otimes d}\rangle\big)
\le C_1 p^{d/2},
\end{equation}
for some constant $C_1>0$. Let
\[
\tau_n=A\frac{p^{d/2}\log n}{n},
\]
where $A>0$ is a sufficiently large constant, and define $\phi_n=\mathbf 1_{\{V_n>\tau_n\}}$.
Then for fixed $d$:
\begin{enumerate}
    \item $\phi_n$ is computable in $O(n^2p)$ arithmetic operations.
    \item Under $H_0$, $\PP_{H_0}(\phi_n=1)\to0$ whenever $p^{d/2}\log n/n\to0$.
    \item Under $H_1$, $\PP_{H_1}(\phi_n=0)\to0$ whenever $p^{d/2}\log n/n\to0$.
\end{enumerate}
In particular, under \eqref{eq_zeta1_assumption_new}, the test is polynomial-time computable and consistent whenever $n\gg p^{d/2}\log n$.
\end{theorem}

\section{Proofs}\label{app_proofs}

\subsection{Proof of Theorem \ref{thm_moment_lower_bound_appendix} and Theorem \ref{thm_lower_bound_cumulant_general_d}}
\label{sec_proof_lower_bound}

\begin{proof}[Proof of Theorem~\ref{thm_moment_lower_bound_appendix} and Theorem~\ref{thm_lower_bound_cumulant_general_d}]
We construct a common family for both estimation lower bounds and the detection lower bounds.
The steps are a bounded perturbation with exact Gaussian moment matching, a Gaussian Hermite correlation calculation, a Rademacher overlap bound, and a reduction from estimation to testing.
Let $G\sim N(0,1)$, and write $H_k$ for the $k$th probabilists' Hermite polynomial, with $H_0\equiv1$.
Choose $d+1$ distinct real numbers $x_0,\ldots,x_d$.
Since $H_i$ is monic of degree $i$, the matrix $(H_i(x_j))_{0\le i,j\le d}$ has determinant $\prod_{0\le j<k\le d}(x_k-x_j)\ne0$.
For sufficiently small $\eta>0$, the intervals $I_j=(x_j-\eta,x_j+\eta)$ are pairwise disjoint.
Let
\[
A_{ij}=\E[H_i(G)\mathbf1_{\{G\in I_j\}}],
\qquad
q_j=\PP(G\in I_j)>0,
\qquad 0\le i,j\le d.
\]
For each $i,j$, continuity of $H_i$ gives
\[
\frac{A_{ij}}{q_j}
=\E[H_i(G)\mid G\in I_j]
\longrightarrow H_i(x_j)
\qquad\text{as }\eta\downarrow0.
\]
The determinant is continuous in the matrix entries, so $A$ is invertible after fixing a sufficiently small $\eta$.
Let $e_d=(0,\ldots,0,1)^\top\in\R^{d+1}$, put $c=A^{-1}e_d$, and write $B=\max_{0\le j\le d}|c_j|>0$.
Define
\[
h(x)=\frac1B\sum_{j=0}^d c_j\mathbf1_{\{x\in I_j\}},
\qquad
b_d=\frac1B.
\]
The disjointness of the intervals implies $\|h\|_\infty=1$, and $Ac=e_d$ gives
\[
\E[H_k(G)h(G)]=0,\qquad 0\le k<d,
\qquad
\E[H_d(G)h(G)]=b_d>0.
\]
All choices in this construction depend only on $d$.

Let $Q=N(0,I_p)$ and $Z\sim Q$.
For $u\in\mathbb S^{p-1}$ and $0<\theta\le1/4$, define $P_{u,\theta}$ by
\[
\frac{dP_{u,\theta}}{dQ}(x)=1+\theta h(u^\top x),
\]
and let $X_{u,\theta}\sim P_{u,\theta}$.
Since $\E h(G)=0$ and $|\theta h|\le1/4$, the function $1+\theta h(u^\top x)$ is at least $3/4$ and integrates to one under $Q$.
It therefore defines a probability density.
For every $v\in\mathbb S^{p-1}$, the density bound and the Gaussian exponential moment yield
\[
\E\exp\left(\frac{(v^\top X_{u,\theta})^2}{4}\right)
\le\frac54\,\E\exp\left(\frac{(v^\top Z)^2}{4}\right)
=\frac54\sqrt2
<2.
\]
This is precisely the inequality required for $\|v^\top X_{u,\theta}\|_{\Phi_2}\le2$.
The same Gaussian calculation shows that $Q\in\mathcal F$, so the null law and all the alternative laws belong to $\mathcal F$.

In orthogonal coordinates along $u$, the density factorizes, so
\[
X_{u,\theta}\overset{d}{=}Su+Z_\perp,
\]
where $S$ has density $(1+\theta h(s))\phi(s)$, $Z_\perp\sim N(0,I_p-uu^\top)$ is independent of $S$, and $\phi$ is the standard Gaussian density.
Take $G$ independent of $Z_\perp$, so that $Gu+Z_\perp\sim Q$.
Because $H_0,\ldots,H_k$ form a basis of the polynomials of degree at most $k$, the orthogonality identities give
\[
\E S^k=\E G^k,\qquad 0\le k<d.
\]
Moreover, $H_d(s)-s^d$ has degree less than $d$, and hence
\[
\E S^d-\E G^d
=\theta\E[G^dh(G)]
=\theta\E[H_d(G)h(G)]
=\theta b_d.
\]
Expanding $(Su+Z_\perp)^{\otimes k}$ and using independence shows that its moments agree with those of $(Gu+Z_\perp)^{\otimes k}$ for $k<d$.
For $k=d$, every term involving fewer than $d$ factors of $S$ still agrees with its Gaussian counterpart.
The only remaining term is $(\E S^d-\E G^d)u^{\otimes d}$, which proves
\[
\bcM_d(X_{u,\theta})-\bcM_d(Z)=\theta b_d u^{\otimes d}.
\]
In the order-$d$ moment--cumulant identity, every term other than the order-$d$ moment itself depends only on moments of orders below $d$.
Those moments agree under $P_{u,\theta}$ and $Q$, while $\bcK_d(Z)=0$ for $d\ge3$.
We have therefore proved
\begin{equation}\label{eq_lower_bound_gaussian_signals}
\bcM_d(X_{u,\theta})-\bcM_d(Z)
=\bcK_d(X_{u,\theta})
=\theta b_d u^{\otimes d}.
\end{equation}

We next bound the Gaussian correlation of $h$.
We use the Hermite basis property in the following precise form: the functions $H_k/\sqrt{k!}$, $k\ge0$, form an orthonormal basis of $L^2(N(0,1))$.
Since $h$ is bounded, this property applies to $h$ and gives
\[
h(z)=\sum_{k=d}^\infty a_k\frac{H_k(z)}{\sqrt{k!}},
\qquad
a_k=\frac{\E[h(G)H_k(G)]}{\sqrt{k!}},
\qquad
\sum_{k=d}^\infty a_k^2=\E h(G)^2\le1,
\]
where the series converges in $L^2(N(0,1))$.
For jointly standard Gaussian variables $G_1,G_2$ with correlation $r\in[-1,1]$, the Hermite generating function gives
\[
\E\left[e^{sG_1-s^2/2}e^{tG_2-t^2/2}\right]=e^{rst}.
\]
Differentiating this identity $j$ times in $s$ and $k$ times in $t$ at zero yields
\[
\E[H_j(G_1)H_k(G_2)]
=\mathbf1_{\{j=k\}}\,k!\,r^k.
\]
These differentiations under the expectation are justified because a Gaussian vector has finite exponential moments of every linear functional, and polynomial factors are integrable under any fixed such exponential tilt.
Applying the identity first to finite Hermite sums and then passing to the $L^2$ limit gives
\[
C(r):=\E[h(G_1)h(G_2)]
=\sum_{k=d}^\infty a_k^2r^k.
\]
Indeed, Cauchy--Schwarz bounds the error from replacing either factor by its partial sum by its $L^2$ approximation error times the $L^2$ norm of the other factor.
The displayed series is absolutely convergent because $\sum_{k=d}^\infty a_k^2\le1$.
It follows that, for every $r\in[-1,1]$,
\[
|C(r)|
\le |r|^d\sum_{k=d}^\infty a_k^2
\le |r|^d
\le r^2,
\]
where the last step uses $d\ge3$ and $|r|\le1$.
This bound does not require $d$ or $h$ to have a particular parity.

Let $\pi_p$ be the uniform distribution on $\{-p^{-1/2},p^{-1/2}\}^p$, and define
\[
\overline P_\theta=\E_{u\sim\pi_p}P_{u,\theta}^{\otimes n},
\qquad
L_u=\frac{dP_{u,\theta}^{\otimes n}}{dQ^{\otimes n}}.
\]
For probability measures $P\ll Q_0$, write $\chi^2(P,Q_0)=\E_{Q_0}[(dP/dQ_0-1)^2]$.
The prior $\pi_p$ has finite support, so expansion of the square and independence of the $n$ observations under $Q^{\otimes n}$ give
\[
\begin{aligned}
1+\chi^2(\overline P_\theta,Q^{\otimes n})
&=\E_{u,v\sim\pi_p}\E_{Q^{\otimes n}}[L_uL_v]\\
&=\E_{u,v\sim\pi_p}
\left(\E_{Z\sim Q}[(1+\theta h(u^\top Z))(1+\theta h(v^\top Z))]\right)^n\\
&=\E_{u,v\sim\pi_p}
\left(1+\theta^2C(u^\top v)\right)^n.
\end{aligned}
\]
Here $u,v$ are independent draws from $\pi_p$, and the terms linear in $\theta$ vanish because $\E h(G)=0$.
Since $1+\theta^2C(r)\ge1-\theta^2>0$ and $1+x\le e^x$ for every real $x$, we obtain
\[
1+\chi^2(\overline P_\theta,Q^{\otimes n})
\le\E\exp(n\theta^2R^2),
\qquad R=u^\top v.
\]
The overlap has the distribution
\[
R=\frac1p\sum_{j=1}^p\xi_j,
\]
where the $\xi_j$ are independent signs with probabilities $1/2$ each.
For every real $t$,
\[
\E e^{tR}
=\left(\cosh(t/p)\right)^p
\le\exp\left(\frac{t^2}{2p}\right).
\]
To verify the inequality, the series for $\cosh$ and $(2j)!\ge2^jj!$ give $\cosh x\le\sum_{j\ge0}(x^2/2)^j/j!=e^{x^2/2}$ for all real $x$.
For $\lambda\ge0$, introduce an independent $G'\sim N(0,1)$.
The Gaussian moment generating function gives $e^{\lambda r^2}=\E_{G'}e^{\sqrt{2\lambda}\,rG'}$, so, whenever $2\lambda<p$,
\[
\begin{aligned}
\E e^{\lambda R^2}
&=\E_{G'}\E_R e^{\sqrt{2\lambda}\,G'R}\\
&\le\E_{G'}e^{\lambda(G')^2/p}\\
&=\left(1-\frac{2\lambda}{p}\right)^{-1/2}.
\end{aligned}
\]
The interchange is valid by Tonelli's theorem because the integrand is nonnegative, and the final equality follows by integrating the Gaussian density.
Choose
\[
\theta=\frac14\left(\sqrt{\frac pn}\wedge1\right).
\]
Then
\[
n\theta^2=\frac1{16}\min\{p,n\}\le\frac p{16}.
\]
Taking $\lambda=n\theta^2$ therefore proves
\[
\chi^2(\overline P_\theta,Q^{\otimes n})
\le\left(1-\frac18\right)^{-1/2}-1
=\sqrt{\frac87}-1.
\]

Finally, write $L=d\overline P_\theta/dQ^{\otimes n}$.
The definition of total variation and Cauchy--Schwarz imply
\[
\operatorname{TV}(\overline P_\theta,Q^{\otimes n})
=\frac12\E_{Q^{\otimes n}}|L-1|
\le\frac12\sqrt{\E_{Q^{\otimes n}}(L-1)^2}
=\frac12\sqrt{\chi^2(\overline P_\theta,Q^{\otimes n})}.
\]
For the rejection event $A=\{\psi=1\}$ of any test,
\begin{equation}\label{eq_lower_bound_testing_risk}
\begin{aligned}
Q^{\otimes n}(\psi=1)+\overline P_\theta(\psi=0)
&=1-\big(\overline P_\theta(A)-Q^{\otimes n}(A)\big)\\
&\ge1-\operatorname{TV}(\overline P_\theta,Q^{\otimes n})\\
&\ge1-\frac12\sqrt{\sqrt{8/7}-1}\\
&>\frac12.
\end{aligned}
\end{equation}

We now derive the estimation lower bounds.
Put $a=\theta b_d$, with $\theta$ as above.
Let $\bcT$ denote either the order-$d$ moment tensor or the order-$d$ cumulant tensor, and write $\bcT(P)$ for the corresponding tensor of a random vector with law $P$.
For every unit vector $u$, the definition of the tensor spectral norm gives $\|u^{\otimes d}\|=1$.
In either case, \eqref{eq_lower_bound_gaussian_signals} gives
\[
\|\bcT(P_{u,\theta})-\bcT(Q)\|=a,\qquad u\in\mathbb S^{p-1}.
\]
Fix an estimator $\widehat{\bcT}$ and denote its worst-case expected error over $\mathcal F$ by
\[
R_{\widehat{\bcT}}
=\sup_{X\in\mathcal F}\E\|\widehat{\bcT}-\bcT(X)\|.
\]
Consider the test
\[
\psi=\mathbf1_{\{\|\widehat{\bcT}-\bcT(Q)\|\ge a/2\}}.
\]
Under $Q^{\otimes n}$, Markov's inequality gives
\[
Q^{\otimes n}(\psi=1)
\le\frac2a\,\E_{Q^{\otimes n}}\|\widehat{\bcT}-\bcT(Q)\|
\le\frac{2R_{\widehat{\bcT}}}{a}.
\]
Under $P_{u,\theta}^{\otimes n}$, the event $\{\psi=0\}$ and the triangle inequality imply
\[
\|\widehat{\bcT}-\bcT(P_{u,\theta})\|
\ge
\|\bcT(P_{u,\theta})-\bcT(Q)\|
-\|\widehat{\bcT}-\bcT(Q)\|
>\frac a2.
\]
Another application of Markov's inequality therefore gives
\[
P_{u,\theta}^{\otimes n}(\psi=0)
\le\frac{2R_{\widehat{\bcT}}}{a},
\qquad u\in\mathbb S^{p-1}.
\]
Averaging the last inequality over $\pi_p$ and using \eqref{eq_lower_bound_testing_risk} yields
\[
\frac12
\le Q^{\otimes n}(\psi=1)+\overline P_\theta(\psi=0)
\le\frac{4R_{\widehat{\bcT}}}{a}.
\]
Hence every estimator satisfies
\[
R_{\widehat{\bcT}}
\ge\frac a8
=\frac{b_d}{32}\left(\sqrt{\frac pn}\wedge1\right).
\]
Taking the infimum over estimators proves both stated estimation lower bounds.
\end{proof}

\subsection{Proof of Theorem \ref{thm_sample_cumulant_lower_general_beta} and \ref{prop_sample_cumulant_expectation_lower_general_beta}}\label{sec_proof_sample_cumulant_lower}

The two results follow from the following stronger confidence lower bound.

\begin{theorem}[Unified confidence lower bound for sample cumulant tensors]
\label{thm_sample_cumulant_master_general_beta}
Fix an integer $d\ge3$ and $0<\beta\le d$.
There exist constants $c_{\beta,d}^{\star},C_d>0$ and $p_0(d)$ such that, for every $p\ge p_0(d)$, $n\ge C_d$, and $0<\delta\le1/16$, there is a $\beta$-subexponential random vector $X\in\R^p$ satisfying
\begin{equation}\label{eq_sample_cumulant_master_general_beta}
\begin{aligned}
\PP_X^{\otimes n}\Bigg(
\|\widehat{\bcK}_d-\bcK_d(X)\|
\ge c_{\beta,d}^{\star}
\max\Bigg\{
&\frac{p^{d/(\beta\vee2)}[\log(1+n)]^{d(1/\beta-1/2)_+}}n,\\
&\frac{[\log(1+n/\delta)]^{d/\beta}}n,\,
\sqrt{\frac{\log(1/\delta)}n}
\Bigg\}
\Bigg)
\ge\frac74\delta.
\end{aligned}
\end{equation}
\end{theorem}

\begin{proof}
We first establish Theorem~\ref{thm_sample_cumulant_master_general_beta}, then deduce Theorems~\ref{prop_sample_cumulant_expectation_lower_general_beta} and~\ref{thm_sample_cumulant_lower_general_beta}.

Write
\[
\mathcal F_{\beta,p}
=
\left\{
X\in\R^p:
\sup_{v\in\mathbb S^{p-1}}
\|v^\top X\|_{\Phi_\beta}\le1
\right\},
\]
and abbreviate
$\Delta(X)=\|\widehat{\bcK}_d-\bcK_d(X)\|$.
For $0<\delta\le1/16$, define the three target scales in \eqref{eq_sample_cumulant_master_general_beta}
\[
\begin{aligned}
\rho_1&=
\frac{p^{d/(\beta\vee2)}[\log(1+n)]^{d(1/\beta-1/2)_+}}n,\\
\rho_2&=
\frac{[\log(1+n/\delta)]^{d/\beta}}n,\\
\rho_3&=
\sqrt{\frac{\log(1/\delta)}n}.
\end{aligned}
\]
We construct distributions in $\mathcal F_{\beta,p}$ for which the error exceeds a constant multiple of each relevant scale with probability at least $7\delta/4$.

\medskip
\noindent\textbf{1. Bernoulli cumulants and rank-one tensor identities.}

Let $\mathcal P([d])$ denote the set of partitions of $[d]$, and let $g_d(t)$ be the $d$th cumulant of a Bernoulli random variable with success probability $t$.
Every positive-order raw moment of this variable equals $t$, so the moment--cumulant formula \eqref{eq_cumulant_moment_inversion} gives
\[
g_d(t)
=
\sum_{\pi\in\mathcal P([d])}
(-1)^{|\pi|-1}(|\pi|-1)!\,t^{|\pi|},
\qquad 0\le t\le1.
\]
The unique partition consisting of a single block is $\pi=\{[d]\}$, which contributes $t$ to the sum.
Define
\[
D_d=
\sum_{\substack{\pi\in\mathcal P([d])\\|\pi|\ge2}}|\pi|!,
\qquad
N_d=\max\{2,\lceil4D_d\rceil\}.
\]
For $0\le t\le1$, the remaining partition terms satisfy
\begin{equation}\label{eq_bernoulli_cumulant_master_bounds}
|g_d(t)-t|\le D_dt^2,
\qquad
|g_d'(t)-1|\le D_dt.
\end{equation}
Indeed, for every integer $k\ge2$ and $t\in[0,1]$,
$t^k\le t^2$, $(k-1)!\le k!$, and $t^{k-1}\le t$.
Applying these inequalities term by term to the polynomial and its derivative proves both bounds.
In particular, whenever $n\ge N_d$,
\begin{equation}\label{eq_bernoulli_cumulant_master_one}
g_d(1/n)
\ge\frac1n-\frac{D_d}{n^2}
\ge\frac3{4n}.
\end{equation}

Let $e_1=(1,0,\ldots,0)^\top\in\R^p$.
For any $q\in(0,1)$, put
\[
A_q=[\log(1+1/q)]^{1/\beta},
\qquad
X=A_qBe_1,
\qquad B\sim\operatorname{Bernoulli}(q).
\]
For every $v\in\mathbb S^{p-1}$,
\[
\E\exp(|v^\top X|^\beta)
\le1-q+q\exp(A_q^\beta)
=2.
\]
Hence $X\in\mathcal F_{\beta,p}$.

For independent copies $B_1,\ldots,B_n$, let $M=\sum_{i=1}^nB_i$.
The population and empirical raw moment tensors of order $r$ are
\[
\bcM_r(X)=A_q^rq\,e_1^{\otimes r},
\qquad
\widehat{\bcM}_r=A_q^r\frac Mn e_1^{\otimes r}.
\]
Applying the moment--cumulant formula gives
\[
\bcK_d(X)=A_q^dg_d(q)e_1^{\otimes d},
\qquad
\widehat{\bcK}_d=A_q^dg_d(M/n)e_1^{\otimes d}.
\]
Since $\|e_1^{\otimes d}\|=1$, we have the exact identity
\begin{equation}\label{eq_sample_bernoulli_cumulant_error}
\Delta(X)=A_q^d|g_d(M/n)-g_d(q)|.
\end{equation}

We also record a related identity for a single empirical outlier.
If an empirical distribution assigns mass $1/n$ to $y\in\R^p$ and the remaining mass to zero, then its $r$th raw moment tensor is $n^{-1}y^{\otimes r}$.
The moment--cumulant formula therefore yields
\begin{equation}\label{eq_master_single_outlier_cumulant}
\widehat{\bcK}_d=g_d(1/n)y^{\otimes d},
\qquad
\|\widehat{\bcK}_d\|
=|g_d(1/n)|\|y\|_2^d.
\end{equation}
The second equality follows from
$\|y^{\otimes d}\|=\|y\|_2^d$, since contraction against unit vectors is bounded by $\|y\|_2^d$, with equality when all vectors are chosen parallel to $y$.

\medskip
\noindent\textbf{2. Dimension-dependent lower bound at constant probability.}

Fix $n\ge N_d$, put $L=\log(1+n)$, and let $G\sim N(0,I_p)$.
Define
\[
a=
\begin{cases}
4^{-1/\beta}L^{1/\beta-1/2},
&0<\beta\le2,\\[3pt]
\left(8\,2^{(\beta-2)/2}\right)^{-1/\beta}
p^{1/\beta-1/2},
&2<\beta\le d,
\end{cases}
\]
and
\[
Y=
\begin{cases}
aG,
&0<\beta\le2,\\
aG\mathbf1_{\{\|G\|_2^2\le2p\}},
&2<\beta\le d.
\end{cases}
\]
Let $B\sim\operatorname{Bernoulli}(1/n)$ be independent of $G$, and set $X=BY$.

We first verify that $X\in\mathcal F_{\beta,p}$.
Fix $v\in\mathbb S^{p-1}$ and write $Z=v^\top G\sim N(0,1)$.

If $0<\beta\le2$, put $r=\beta/2\in(0,1]$.
The weighted arithmetic--geometric mean inequality
$u^{1-r}w^r\le(1-r)u+rw$ for $u,w\ge0$ gives
\[
|v^\top Y|^\beta
=\frac14L^{1-r}(Z^2)^r
\le\frac{(1-r)L+rZ^2}{4}.
\]
For a standard Gaussian variable, direct integration of its density gives
$\E\exp(hZ^2)=(1-2h)^{-1/2}$ for $h<1/2$.
Since $r/4\le1/4$, this identity implies
\[
\begin{aligned}
\E\exp(|v^\top Y|^\beta)
&\le
\exp\left(\frac{(1-r)L}{4}\right)
\E\exp(rZ^2/4)\\
&=(1+n)^{(1-r)/4}(1-r/2)^{-1/2}\\
&\le\sqrt2(1+n)^{1/4}
\le1+n,
\end{aligned}
\]
where the last inequality holds for $n\ge2$.

If $2<\beta\le d$, then on $\{\|G\|_2^2\le2p\}$,
$|Z|\le\sqrt{2p}$, and hence
\[
|v^\top Y|^\beta
\le a^\beta(2p)^{(\beta-2)/2}Z^2
=\frac{Z^2}{8}.
\]
Outside this event, the left-hand side is zero, so the inequality holds everywhere.
Consequently,
\[
\E\exp(|v^\top Y|^\beta)
\le\E\exp(Z^2/8)
=(1-1/4)^{-1/2}
<2\le1+n.
\]

In either case, independence of $B$ and $Y$ gives
\[
\E\exp(|v^\top X|^\beta)
=1-\frac1n+\frac1n\E\exp(|v^\top Y|^\beta)
\le2.
\]
Thus $X\in\mathcal F_{\beta,p}$, and both constructions have moments of every order.

Next, let $m_j=\E|Z|^j$ for $Z\sim N(0,1)$ and $1\le j\le d$.
For unit vectors $v_1,\ldots,v_j$, H\"older's inequality in the form
$\E\prod_{\ell=1}^j|U_\ell|
\le\prod_{\ell=1}^j(\E|U_\ell|^j)^{1/j}$
gives
\[
\begin{aligned}
\left|\E\prod_{\ell=1}^jv_\ell^\top X\right|
&\le\frac{a^j}{n}
\E\prod_{\ell=1}^j|v_\ell^\top G|\\
&\le\frac{a^j}{n}
\prod_{\ell=1}^j
\left(\E|v_\ell^\top G|^j\right)^{1/j}\\
&=\frac{a^jm_j}{n}.
\end{aligned}
\]
The first inequality applies to both constructions because the truncation indicator is at most one.
Define
\[
H_d=
\sum_{\pi\in\mathcal P([d])}
(|\pi|-1)!\prod_{J\in\pi}m_{|J|}.
\]
Applying the moment--cumulant formula to arbitrary multilinear contractions, and using $\sum_{J\in\pi}|J|=d$ and $n^{-|\pi|}\le n^{-1}$, yields
\begin{equation}\label{eq_master_population_cumulant_bound}
\|\bcK_d(X)\|
\le\frac{H_da^d}{n}.
\end{equation}

Generate the sample using independent copies $(B_i,G_i)$, and write $M=\sum_{i=1}^nB_i$.
The probability of exactly one Bernoulli success is
\[
\PP(M=1)=(1-1/n)^{n-1}\ge e^{-1}.
\]
Indeed,
\[
\log(1-1/n)
=-\int_0^{1/n}\frac{du}{1-u}
\ge-\frac1{n-1},
\]
which proves the inequality.

On $\{M=1\}$, let $G_*$ be the Gaussian vector associated with the unique success.
Conditional on the Bernoulli variables, $G_*$ has distribution $N(0,I_p)$.
For $G\sim N(0,I_p)$, independence of the coordinates gives
$\E\|G\|_2^2=p$ and $\operatorname{Var}(\|G\|_2^2)=2p$.
Chebyshev's inequality,
$\PP(|V-\E V|\ge u)\le\operatorname{Var}(V)/u^2$,
therefore gives
\[
\PP\left(\frac p2\le\|G\|_2^2\le2p\right)
\ge1-\frac8p
\ge\frac12,
\qquad p\ge16.
\]
Consequently, the event
\[
E_{\mathrm G}
=
\left\{
M=1,\quad
\frac p2\le\|G_*\|_2^2\le2p
\right\}
\]
satisfies
\begin{equation}\label{eq_master_gaussian_event_probability}
\PP(E_{\mathrm G})\ge\frac1{2e},
\qquad p\ge16.
\end{equation}

On $E_{\mathrm G}$, the sample contains $n-1$ zeros and one observation $aG_*$.
This holds for both constructions because the truncation leaves $G_*$ unchanged on this event.
By \eqref{eq_bernoulli_cumulant_master_one}, \eqref{eq_master_single_outlier_cumulant}, and \eqref{eq_master_population_cumulant_bound},
\[
\begin{aligned}
\Delta(X)
&\ge\|\widehat{\bcK}_d\|-\|\bcK_d(X)\|\\
&\ge\frac{a^d}{n}
\left\{
\frac34\left(\frac p2\right)^{d/2}-H_d
\right\}.
\end{aligned}
\]
Choose
\[
p_0(d)=
\left\lceil
\max\left\{
16,\,
2\left(\frac{8H_d}{3}\right)^{2/d}
\right\}
\right\rceil.
\]
For $p\ge p_0(d)$,
$H_d\le(3/8)(p/2)^{d/2}$, so
\[
\Delta(X)
\ge
\frac{3}{8\,2^{d/2}}
\frac{a^dp^{d/2}}n
\qquad\text{on }E_{\mathrm G}.
\]
The definitions of $a$ give
\[
a^dp^{d/2}
=
\begin{cases}
4^{-d/\beta}
p^{d/2}L^{d/\beta-d/2},
&0<\beta\le2,\\[3pt]
\left(8\,2^{(\beta-2)/2}\right)^{-d/\beta}
p^{d/\beta},
&2<\beta\le d.
\end{cases}
\]
Therefore, defining
\[
b_{\beta,d}
=
\frac{3}{8\,2^{d/2}}
\begin{cases}
4^{-d/\beta},
&0<\beta\le2,\\[3pt]
\left(8\,2^{(\beta-2)/2}\right)^{-d/\beta},
&2<\beta\le d,
\end{cases}
\]
we obtain
\begin{equation}\label{eq_master_dimension_probability}
\PP_X^{\otimes n}
\left(\Delta(X)\ge b_{\beta,d}\rho_1\right)
\ge\frac1{2e},
\qquad p\ge p_0(d),\quad n\ge N_d.
\end{equation}

\medskip
\noindent\textbf{3. Rare-outlier lower bound for the logarithmic scale.}

Fix $0<\delta\le1/16$.
Use the rank-one Bernoulli construction from Step 1 with
\[
q=\frac{2\delta}{n},
\qquad
A_q=[\log(1+1/q)]^{1/\beta},
\qquad
X=A_qBe_1.
\]
Since $q\le1/(8n)$ and $n\ge N_d\ge4D_d$,
\[
D_dq\le\frac1{32}.
\]
By \eqref{eq_bernoulli_cumulant_master_bounds},
\[
|g_d(q)|
\le q+D_dq^2
\le2q
=\frac{4\delta}{n}
\le\frac1{4n}.
\]
On $\{M=1\}$, equations \eqref{eq_sample_bernoulli_cumulant_error} and \eqref{eq_bernoulli_cumulant_master_one} consequently yield
\[
\Delta(X)
\ge\frac{A_q^d}{2n}
=
\frac1{2n}
\left[\log\left(1+\frac n{2\delta}\right)\right]^{d/\beta}.
\]

For $q\in[0,1]$ and every integer $m\ge0$,
$(1-q)^m\ge1-mq$.
This follows by induction from
$(1-q)(1-mq)=1-(m+1)q+mq^2\ge1-(m+1)q$.
Applying this inequality gives
\[
\begin{aligned}
\PP(M=1)
&=nq(1-q)^{n-1}\\
&\ge nq\big(1-(n-1)q\big)\\
&\ge2\delta(1-2\delta)\\
&\ge\frac74\delta.
\end{aligned}
\]
Furthermore, for every $x\ge0$,
$(1+x/2)^2\ge1+x$, so
$\log(1+x/2)\ge\frac12\log(1+x)$.
Taking $x=n/\delta$ gives
\[
\left[\log\left(1+\frac n{2\delta}\right)\right]^{d/\beta}
\ge
2^{-d/\beta}
[\log(1+n/\delta)]^{d/\beta}.
\]
Thus
\begin{equation}\label{eq_master_rare_probability}
\PP_X^{\otimes n}
\left(
\Delta(X)\ge2^{-1-d/\beta}\rho_2
\right)
\ge\frac74\delta,
\qquad n\ge N_d.
\end{equation}

\medskip
\noindent\textbf{4. Binomial fluctuation lower bound.}

Define
\[
q_0=\frac1{8(D_d+1)},
\qquad
A_0=[\log(1+1/q_0)]^{1/\beta},
\qquad
\sigma^2=q_0(1-q_0),
\]
and put
\[
\varepsilon=\frac{q_0}{2},
\qquad
t_0=\frac4{\varepsilon^2},
\qquad
C_d=
\max\left\{
N_d,\left\lceil\frac{64}{q_0}\right\rceil
\right\}.
\]
Here $q_0\le1/16$, so
\eqref{eq_bernoulli_cumulant_master_bounds} implies
\begin{equation}\label{eq_master_bernoulli_derivative}
g_d'(u)\ge1-D_du\ge\frac12,
\qquad 0\le u\le2q_0.
\end{equation}
We use the distribution
\[
X=A_0Be_1,
\qquad B\sim\operatorname{Bernoulli}(q_0),
\]
which belongs to $\mathcal F_{\beta,p}$ by Step 1.
Write
\[
t=\log(1/\delta)\ge\log16,
\qquad
\rho_3=\sqrt{\frac tn}.
\]

We first consider $\log16\le t\le t_0$.
Let $M\sim\operatorname{Binomial}(n,q_0)$ and define
\[
Z_n=\frac{M-nq_0}{\sigma\sqrt n}.
\]
Then $\E Z_n=0$ and $\E Z_n^2=1$.
Writing $M-nq_0$ as the sum of $n$ independent centered Bernoulli variables gives
\[
\E Z_n^4
=
\frac{
n\E(B-q_0)^4+3n(n-1)\sigma^4
}{
n^2\sigma^4
}.
\]
Since $|B-q_0|\le1$, we have
$\E(B-q_0)^4\le\E(B-q_0)^2=\sigma^2$.
Consequently,
\[
\E Z_n^4
\le3+\frac1{n\sigma^2}
\le4,
\]
where the last inequality follows from
$n\ge C_d\ge64/q_0$ and
$n\sigma^2=nq_0(1-q_0)\ge1$.

Let $F=\{|Z_n|\ge1/2\}$.
Cauchy--Schwarz gives
\[
\begin{aligned}
1=\E Z_n^2
&\le\frac14+\E[Z_n^2\mathbf1_F]\\
&\le\frac14+
(\E Z_n^4)^{1/2}\PP(F)^{1/2}.
\end{aligned}
\]
It follows that
\[
\PP(F)
\ge\frac{(3/4)^2}{\E Z_n^4}
\ge\frac9{64}.
\]
On the other hand, Chebyshev's inequality gives
\[
\PP\left(\left|\frac Mn-q_0\right|>q_0\right)
\le\frac{\sigma^2}{nq_0^2}
\le\frac1{nq_0}
\le\frac1{64}.
\]
Therefore, the event
\[
E_0=
\left\{
\frac{\sigma}{2\sqrt n}
\le\left|\frac Mn-q_0\right|
\le q_0
\right\}
\]
satisfies
\begin{equation}\label{eq_master_moderate_event}
\PP(E_0)\ge\frac18.
\end{equation}

On $E_0$, both $q_0$ and $M/n$ belong to $[0,2q_0]$.
Integrating \eqref{eq_master_bernoulli_derivative} over the interval between these two values yields
\[
|g_d(M/n)-g_d(q_0)|
\ge\frac12\left|\frac Mn-q_0\right|
\ge\frac{\sigma}{4\sqrt n}.
\]
By \eqref{eq_sample_bernoulli_cumulant_error},
\[
\Delta(X)
\ge\frac{A_0^d\sigma}{4\sqrt n}
\ge\frac{A_0^d\sigma}{4\sqrt{t_0}}\rho_3,
\]
where the last inequality uses $t\le t_0$.
Since $\delta\le1/16$,
$\frac18\ge\frac74\delta$.
Thus
\begin{equation}\label{eq_master_moderate_fluctuation}
\PP_X^{\otimes n}\left(
\Delta(X)\ge
\frac{A_0^d\sigma}{4\sqrt{t_0}}\rho_3
\right)
\ge\frac74\delta,
\qquad \log16\le t\le t_0.
\end{equation}

Next, suppose $t\ge t_0$ and $n\ge t$.
Set
\[
s=\varepsilon\sqrt{\frac tn},
\qquad
q_1=q_0+s.
\]
Since $n\ge t$,
$0<s\le\varepsilon=q_0/2$, so
$q_0<q_1\le3q_0/2<1/2$.
Let $\PP_0$ and $\PP_1$ be the laws of $n$ independent Bernoulli variables with success probabilities $q_0$ and $q_1$, respectively, and let $M$ denote their sum.
Consider the event
\[
E_1=
\left\{
\left|\frac Mn-q_1\right|\le\frac s2
\right\}.
\]
Under $\PP_1$,
$\E_1(M/n)=q_1$ and
$\operatorname{Var}_1(M/n)=q_1(1-q_1)/n\le1/(4n)$.
Chebyshev's inequality gives
\[
\PP_1(E_1)
\ge1-\frac1{ns^2}
=1-\frac1{\varepsilon^2t}
\ge\frac34,
\]
because $t\ge t_0=4/\varepsilon^2$.

Both Bernoulli sample laws have the same finite support, and their log likelihood ratio is
\[
\log\frac{d\PP_1}{d\PP_0}
=
nD(q_1\|q_0)+(M-nq_1)\lambda,
\]
where
\[
\begin{aligned}
D(q_1\|q_0)
&=
q_1\log\frac{q_1}{q_0}
+(1-q_1)\log\frac{1-q_1}{1-q_0},\\
\lambda
&=
\log\frac{q_1(1-q_0)}{q_0(1-q_1)}.
\end{aligned}
\]
Using $\log x\le x-1$ for $x>0$, we obtain
\[
\begin{aligned}
D(q_1\|q_0)
&\le
q_1\left(\frac{q_1}{q_0}-1\right)
+(1-q_1)\left(\frac{1-q_1}{1-q_0}-1\right)\\
&=\frac{s^2}{q_0(1-q_0)}
\le\frac{2s^2}{q_0}.
\end{aligned}
\]
Likewise, since $q_1<1/2$ and $q_0<1/2$,
\[
\begin{aligned}
0<\lambda
&=
\log\left(1+\frac{s}{q_0}\right)
+\log\left(1+\frac{s}{1-q_1}\right)\\
&\le\frac{s}{q_0}+2s
\le\frac{2s}{q_0}.
\end{aligned}
\]
On $E_1$, we have $M-nq_1\le ns/2$.
Therefore,
\[
\begin{aligned}
\log\frac{d\PP_1}{d\PP_0}
&\le\frac{2ns^2}{q_0}
+\frac{ns}{2}\frac{2s}{q_0}\\
&=\frac{3ns^2}{q_0}
=\frac{3\varepsilon^2t}{q_0}
\le\frac t4.
\end{aligned}
\]
The last inequality follows from
$\varepsilon=q_0/2$ and $q_0\le1/16$, which imply
$3\varepsilon^2/q_0=3q_0/4\le1/4$.

Changing measure on $E_1$ gives
\[
\begin{aligned}
\PP_0(E_1)
&=
\E_1\left[
\frac{d\PP_0}{d\PP_1}\mathbf1_{E_1}
\right]\\
&\ge e^{-t/4}\PP_1(E_1)\\
&\ge\frac34e^{-t/4}
\ge\frac74e^{-t}
=\frac74\delta.
\end{aligned}
\]
For the penultimate inequality, $t\ge\log16$ implies
$e^{3t/4}\ge16^{3/4}=8$, and hence
$(3/4)e^{3t/4}\ge6>7/4$.

On $E_1$,
\[
q_0+\frac s2
\le\frac Mn
\le q_0+\frac{3s}{2}
\le2q_0.
\]
Integrating \eqref{eq_master_bernoulli_derivative} over $[q_0,M/n]$ yields
\[
g_d(M/n)-g_d(q_0)
\ge\frac12\left(\frac Mn-q_0\right)
\ge\frac s4.
\]
Thus, under the original law $\PP_0$ corresponding to
$X=A_0Be_1$,
\[
\Delta(X)
\ge\frac{A_0^ds}{4}
=\frac{A_0^d\varepsilon}{4}\rho_3
\qquad\text{on }E_1.
\]
We conclude that
\begin{equation}\label{eq_master_large_fluctuation}
\PP_X^{\otimes n}\left(
\Delta(X)\ge\frac{A_0^d\varepsilon}{4}\rho_3
\right)
\ge\frac74\delta,
\qquad t\ge t_0,\quad n\ge t.
\end{equation}

It remains to consider $t\ge t_0$ and $n<t$.
Because $\beta\le d$ and $t\ge\log16>1$,
\[
\rho_2
=
\frac{[\log(1+n/\delta)]^{d/\beta}}n
\ge\frac{t^{d/\beta}}n
\ge\frac tn
>
\sqrt{\frac tn}
=\rho_3.
\]
Hence the logarithmic scale already dominates the fluctuation scale in this regime, and the rare-outlier construction from Step 3 applies.

\medskip
\noindent\textbf{5. Proof of the unified confidence lower bound.}

Let
\[
c_{\beta,d}^{\star}
=
\min\left\{
b_{\beta,d},\,
2^{-1-d/\beta},\,
\frac{A_0^d}{4}
\min\left\{
\frac{\sigma}{\sqrt{t_0}},\varepsilon
\right\}
\right\}>0,
\]
and use $C_d$ and $p_0(d)$ defined above.

Fix $p\ge p_0(d)$, $n\ge C_d$, and $0<\delta\le1/16$.
If $\rho_1$ is the largest of the three scales, use the Gaussian or truncated-Gaussian construction in \eqref{eq_master_dimension_probability}.
Its success probability is at least $1/(2e)$, which is greater than $7/64\ge7\delta/4$.

If $\rho_2$ is the largest scale, use the rare-outlier construction in \eqref{eq_master_rare_probability}, whose success probability is at least $7\delta/4$.

Finally, suppose $\rho_3$ is strictly larger than both $\rho_1$ and $\rho_2$.
If $t\le t_0$, use \eqref{eq_master_moderate_fluctuation}.
If $t>t_0$, the inequality $\rho_3>\rho_2$ forces $n\ge t$, so \eqref{eq_master_large_fluctuation} applies.

In every case, the selected distribution belongs to $\mathcal F_{\beta,p}$ and satisfies
\[
\PP_X^{\otimes n}\left(
\Delta(X)\ge
c_{\beta,d}^{\star}
\max\{\rho_1,\rho_2,\rho_3\}
\right)
\ge\frac74\delta.
\]
This proves Theorem~\ref{thm_sample_cumulant_master_general_beta}.

\medskip
\noindent\textbf{6. Deduction of the expectation lower bound.}

Apply Theorem~\ref{thm_sample_cumulant_master_general_beta} with $\delta=1/16$.
For the resulting distribution $X\in\mathcal F_{\beta,p}$,
\[
\PP_X^{\otimes n}\left(
\Delta(X)\ge
c_{\beta,d}^{\star}
\max\left\{
\rho_1,\,
\frac{[\log(1+16n)]^{d/\beta}}n,\,
\sqrt{\frac{\log16}{n}}
\right\}
\right)
\ge\frac7{64}.
\]
Since $\Delta(X)\ge0$, the elementary inequality
$\E W\ge a\PP(W\ge a)$ for any nonnegative random variable $W$ and $a>0$ gives
\[
\begin{aligned}
\E\Delta(X)
&\ge
\frac{7c_{\beta,d}^{\star}}{64}
\max\left\{
\rho_1,\,
\frac{[\log(1+16n)]^{d/\beta}}n,\,
\sqrt{\frac{\log16}{n}}
\right\}\\
&\ge
\frac{7c_{\beta,d}^{\star}}{64n}
\max\left\{
p^{d/(\beta\vee2)}
[\log(1+n)]^{d(1/\beta-1/2)_+},
[\log(1+n)]^{d/\beta}
\right\},
\end{aligned}
\]
where the last inequality uses $\log(1+16n)\ge\log(1+n)$.
Taking $c_{\beta,d}=7c_{\beta,d}^{\star}/64$ proves
Theorem~\ref{prop_sample_cumulant_expectation_lower_general_beta}.

\medskip
\noindent\textbf{7. Deduction of the exponential-tail lower bound.}

Fix $\zeta>0$ and take
\[
\delta=e^{-\zeta p}.
\]
For every
\[
p\ge
\max\left\{
p_0(d),\frac{\log16}{\zeta}
\right\},
\]
we have $0<\delta\le1/16$, so the unified confidence lower bound applies.
Since
\[
\log(1+n/\delta)
=\log(1+ne^{\zeta p})
\ge\zeta p,
\]
the second scale satisfies
\[
\rho_2
\ge\zeta^{d/\beta}\frac{p^{d/\beta}}n.
\]
Moreover, the third scale is
\[
\rho_3
=
\sqrt{\frac{\zeta p}{n}}
=
\sqrt{\zeta}\sqrt{\frac pn}.
\]
Consequently,
\[
\max\{\rho_1,\rho_2,\rho_3\}
\ge
\min\{\zeta^{d/\beta},\sqrt{\zeta}\}
\max\left\{
\frac{p^{d/\beta}}n,\sqrt{\frac pn}
\right\}.
\]
Define
\[
c_{\beta,d,\zeta}
=
c_{\beta,d}^{\star}
\min\{\zeta^{d/\beta},\sqrt{\zeta}\}>0.
\]
Theorem~\ref{thm_sample_cumulant_master_general_beta} then gives a distribution $X\in\mathcal F_{\beta,p}$ such that
\[
\begin{aligned}
&\PP_X^{\otimes n}\left(
\|\widehat{\bcK}_d-\bcK_d(X)\|
\ge
c_{\beta,d,\zeta}
\max\left\{
\frac{p^{d/\beta}}n,\sqrt{\frac pn}
\right\}
\right)\\
&\hspace{3cm}
\ge\frac74e^{-\zeta p}
>e^{-\zeta p}.
\end{aligned}
\]
The constant $C_d$ is independent of $\beta$ and $\zeta$, while the threshold on $p$ may depend on $d$ and $\zeta$.
This proves Theorem~\ref{thm_sample_cumulant_lower_general_beta}.
\end{proof}

\subsection{Proof of Theorem \ref{thm_lower_bound_detection} and Theorem \ref{thm_moment_detection_appendix}}

\begin{proof}[Proof of Theorem~\ref{thm_lower_bound_detection} and Theorem~\ref{thm_moment_detection_appendix}]
Use the family $Q,P_{u,\theta}$ constructed in the proof of Theorem~\ref{thm_lower_bound_cumulant_general_d}, with $\theta=\frac14(\sqrt{p/n}\wedge1)$, and put $a=\theta b_d$.
The Gaussian law $Q$ and every $P_{u,\theta}$ belong to $\mathcal F$.
For cumulants, $\bcK_d(Q)=0$ and \eqref{eq_lower_bound_gaussian_signals} gives $\|\bcK_d(P_{u,\theta})\|=a$.
Here a moment or cumulant of a probability law denotes that of a random vector with this law.
Consequently, for $0<\rho\le a$ and every test $\psi$, \eqref{eq_lower_bound_testing_risk} gives
\[
\begin{aligned}
&\sup_{X\in\mathcal F:\,\bcK_d(X)=0}\PP_X^{\otimes n}(\psi=1)
+\sup_{X\in\mathcal F:\,\|\bcK_d(X)\|\ge\rho}\PP_X^{\otimes n}(\psi=0)\\
&\qquad\ge
Q^{\otimes n}(\psi=1)
+\E_{u\sim\pi_p}P_{u,\theta}^{\otimes n}(\psi=0)\\
&\qquad=
Q^{\otimes n}(\psi=1)+\overline P_\theta(\psi=0)
\ge\frac12.
\end{aligned}
\]
Taking the infimum over $\psi$ proves Theorem~\ref{thm_lower_bound_detection} with $c_d=b_d/4$ and $\eta_d=1/2$.
If $d$ is odd, $\bcM_d(Q)=0$, and \eqref{eq_lower_bound_gaussian_signals} gives $\bcM_d(P_{u,\theta})=a u^{\otimes d}$.
The identical argument with $\bcM_d$ in place of $\bcK_d$ proves Theorem~\ref{thm_moment_detection_appendix}.
\end{proof}

\subsection{Proof of Theorem \ref{thm_upper_bound_sample_moments_appendix} and Theorem \ref{thm_upper_bound_sample_cumulants}}

The proof has four ingredients: scalar concentration from Lemma~\ref{lemma_subexpo_concentration}, a product-net bound for the tensor spectral norm, an empirical absolute-moment bound from Minkowski's and H\"older's inequalities, and the moment-partition expansions with product telescoping.
We first derive the scalar and tensor concentration bounds for arbitrary $\beta>0$, and then impose $0<\beta\le d$ to prove the two sample-tensor theorems.

\begin{proof}
\textbf{Scalar moment concentration.}
For an integer $r\ge1$ and fixed $u_1,\ldots,u_r\in\mathbb S^{p-1}$, let $W_i=\prod_{j=1}^r\langle X_i,u_j\rangle$ and $\alpha_r=\min\{\beta/r,1\}$.
The arithmetic--geometric mean inequality and H\"older's inequality give
\[
\E\exp\big(|W_i|^{\beta/r}\big)
\le
\E\exp\left(\frac1r\sum_{j=1}^r|\langle X_i,u_j\rangle|^\beta\right)
\le
\prod_{j=1}^r\left(\E\exp\big(|\langle X_i,u_j\rangle|^\beta\big)\right)^{1/r}
\le2.
\]
If $\beta/r>1$, the inequality $x\le1+x^{\beta/r}$ for $x\ge0$ also implies
\[
\E\exp(|W_i|/4)
\le e^{1/4}\left(\E\exp\big(|W_i|^{\beta/r}\big)\right)^{1/4}
\le(2e)^{1/4}<2.
\]
Consequently, $\E\exp(|W_i/4|^{\alpha_r})\le2$ in both cases.
To justify centering when $0<\alpha\le1$, suppose $\E\exp(|V|^\alpha)\le2$.
Integration of $\PP(|V|>t)\le2e^{-t^\alpha}$ yields $\E|V|\le m_\alpha:=2\Gamma(1+1/\alpha)$.
Choose $\theta_\alpha=\log(2)/(m_\alpha^\alpha+\log(2))\in(0,1)$.
Subadditivity of $x^\alpha$ and Jensen's inequality yield
\[
\E\exp\big(\theta_\alpha|V-\E V|^\alpha\big)
\le
e^{\theta_\alpha m_\alpha^\alpha}\left(\E\exp(|V|^\alpha)\right)^{\theta_\alpha}
\le2.
\]
Applying this to $V=W_i/4$ proves $\|W_i-\E W_i\|_{\Phi_{\alpha_r}}\le C_{\beta,r}$.
The variables $W_i-\E W_i$ are independent across $i$ and centered.
Lemma~\ref{lemma_subexpo_concentration}, applied with $a_i=1/n$, $\|a\|^2=1/n$, and $\max_i|a_i|=1/n$, therefore gives, for every $t>0$,
\begin{equation}\label{eq_scalar_moment_concentration_upper}
\PP\left(\left|\frac1n\sum_{i=1}^n W_i-\E W_i\right|>t\right)
\le
2\exp\left(-c_{\beta,r}\min\{nt^2,(nt)^{\alpha_r}\}\right).
\end{equation}
Only independence across observations is used; the projections within $W_i$ may be dependent.

\medskip
\noindent\textbf{Uniform raw-moment bounds.}
Let $\widehat{\bcM}_r=n^{-1}\sum_{i=1}^nX_i^{\otimes r}$ and $\bcM_r=\E X^{\otimes r}$.
Take an $\varepsilon_r$-net $\mathcal N_r$ of $\mathbb S^{p-1}$ with $\varepsilon_r=1/(2r)$ and $|\mathcal N_r|\le(1+4r)^p$.
Such a net is obtained from a maximal $\varepsilon_r$-separated set: the balls of radius $\varepsilon_r/2$ about its points are disjoint and lie in the ball of radius $1+\varepsilon_r/2$, giving the stated cardinality by volume comparison.
For any order-$r$ tensor $\bcT$, choose $v_j\in\mathcal N_r$ with $\|u_j-v_j\|\le\varepsilon_r$ for each unit vector $u_j$.
Replacing the arguments one at a time and using multilinearity gives
\[
\left|\left\langle\bcT,\bigotimes_{j=1}^r u_j\right\rangle
-\left\langle\bcT,\bigotimes_{j=1}^r v_j\right\rangle\right|
\le\sum_{j=1}^r\|u_j-v_j\|\,\|\bcT\|
\le\tfrac12\|\bcT\|.
\]
Taking the supremum over $u_1,\ldots,u_r$ proves
\[
\|\bcT\|
\le2\max_{v_1,\ldots,v_r\in\mathcal N_r}
\left|\left\langle\bcT,\bigotimes_{j=1}^r v_j\right\rangle\right|.
\]
For any $L\ge1$, set $t=C_{\beta,r}\max\{\sqrt{Lp/n},(Lp)^{1/\alpha_r}/n\}$ in \eqref{eq_scalar_moment_concentration_upper}, with the constant large enough that its exponent is at least $Lp$.
Indeed, $nt^2\ge C_{\beta,r}^2Lp$ and $(nt)^{\alpha_r}\ge C_{\beta,r}^{\alpha_r}Lp$ before adjusting the constant.
A union bound over $\mathcal N_r^r$ gives
\begin{equation}\label{eq_uniform_raw_moment_concentration_upper}
\PP\left(
\|\widehat{\bcM}_r-\bcM_r\|
>C_{\beta,r}\max\left\{\sqrt{\frac{Lp}{n}},\frac{(Lp)^{1/\alpha_r}}n\right\}
\right)
\le2\exp\big(rp\log(1+4r)-Lp\big).
\end{equation}
We now assume $0<\beta\le d$ and write $b=\max\{\sqrt{p/n},p^{d/\beta}/n\}$.
For every $r\le d$, we have $1/\alpha_r=\max\{r/\beta,1\}\le d/\beta$, so \eqref{eq_uniform_raw_moment_concentration_upper} gives
\begin{equation}\label{eq_simultaneous_raw_moment_upper}
\|\widehat{\bcM}_r-\bcM_r\|\le C_{\beta,d}b,
\qquad r=1,\ldots,d,
\end{equation}
simultaneously with probability at least $1-\exp(-c_{\beta,d}p)$, after choosing $L$ sufficiently large depending only on $d$ and absorbing its fixed powers into $C_{\beta,d}$.
In particular, $\beta>r$ is handled by the $\alpha_r=1$ case of the concentration lemma, rather than by applying the order-$r$ sample theorem outside its assumptions.
The case $r=d$ proves the asserted raw-moment bound.

\medskip
\noindent\textbf{An empirical absolute-moment bound.}
Let $P_nf=n^{-1}\sum_{i=1}^nf(X_i)$ and define
\[
H_d=\sup_{\|u\|=1}\big(P_n|\langle X,u\rangle|^d\big)^{1/d},
\qquad
B_{\beta,d}=2\Gamma(1+d/\beta).
\]
The tail assumption gives, for every unit vector $u$,
\[
\E|\langle X,u\rangle|^d
\le2d\int_0^\infty t^{d-1}e^{-t^\beta}\,dt
=B_{\beta,d}.
\]
For each fixed sample, $u\mapsto(P_n|\langle X,u\rangle|^d)^{1/d}$ is a seminorm by Minkowski's inequality.
If $\mathcal N$ is a $1/2$-net with $|\mathcal N|\le5^p$, this seminorm property gives
\[
H_d\le2\max_{u\in\mathcal N}\big(P_n|\langle X,u\rangle|^d\big)^{1/d}.
\]
In fact, approximating a unit vector by a point in $\mathcal N$ leaves a remainder whose seminorm is at most $H_d/2$.
Since $\E\exp((|\langle X,u\rangle|^d)^{\beta/d})\le2$ and $\beta/d\le1$, the centering calculation and Lemma~\ref{lemma_subexpo_concentration} apply to $|\langle X_i,u\rangle|^d$ as well.
The same choice of scalar threshold and a union bound over $\mathcal N$ prove
\begin{equation}\label{eq_empirical_absolute_moment_upper}
H_d^d\le C_{\beta,d}(1+b).
\end{equation}
All events in \eqref{eq_simultaneous_raw_moment_upper} and \eqref{eq_empirical_absolute_moment_upper} can be imposed simultaneously with failure probability at most $\exp(-c_{\beta,d}p)$.
For example, choose $L$ larger than $\max_{r\le d}r\log(1+4r)$ and $\log(5)$ by at least $\log(2(d+1))+1$; the sum of the net failure probabilities is then at most $e^{-p}$.
The threshold constants may be increased to accommodate this fixed choice of $L$.
We work on this simultaneous event below.

\medskip
\noindent\textbf{Cumulants and central moments when $b>1$.}
For a block $B\subseteq[d]$ and unit vectors $u_1,\ldots,u_d$, H\"older's inequality on the empirical probability measure gives
\[
\left|P_n\prod_{j\in B}\langle X,u_j\rangle\right|
\le\prod_{j\in B}\big(P_n|\langle X,u_j\rangle|^d\big)^{1/d}.
\]
When $|B|<d$, the unused H\"older exponent is assigned to the constant function $1$.
Multiplying this inequality over the blocks of any partition $\pi\in\mathcal P([d])$ bounds the resulting product by $H_d^d$.
Applying the moment--cumulant inversion identity \eqref{eq_cumulant_moment_inversion} to the empirical distribution therefore yields
\[
\|\widehat{\bcK}_d\|
\le A_dH_d^d,
\qquad
A_d=\sum_{\pi\in\mathcal P([d])}(|\pi|-1)!.
\]
The population version of H\"older's inequality gives $\|\bcK_d\|\le A_dB_{\beta,d}$.
Thus \eqref{eq_empirical_absolute_moment_upper} and $b>1$ imply $\|\widehat{\bcK}_d-\bcK_d\|\le C_{\beta,d}b$.

For central moments, Jensen's inequality gives $\|\bar X\|\le H_d$.
Minkowski's inequality consequently gives
\[
\sup_{\|u\|=1}\big(P_n|\langle X-\bar X,u\rangle|^d\big)^{1/d}
\le H_d+\|\bar X\|\le2H_d.
\]
H\"older's inequality for arbitrary unit contraction vectors then implies
\[
\left\|\frac1n\sum_{i=1}^n(X_i-\bar X)^{\otimes d}\right\|\le2^dH_d^d.
\]
The corresponding population bound is $\|\E(X-\mu)^{\otimes d}\|\le2^dB_{\beta,d}$, because $\|\mu\|\le B_{\beta,d}^{1/d}$.
The triangle inequality and \eqref{eq_empirical_absolute_moment_upper} prove the central-moment error bound when $b>1$.

\medskip
\noindent\textbf{Cumulants and central moments when $b\le1$.}
For $r\le d$, H\"older's inequality gives $\|\bcM_r\|\le B_{\beta,d}^{r/d}$, and \eqref{eq_simultaneous_raw_moment_upper} now implies
\[
\|\widehat{\bcM}_r\|+\|\bcM_r\|\le C_{\beta,d},
\qquad
\|\widehat{\bcM}_r-\bcM_r\|\le C_{\beta,d}b.
\]
For fixed unit vectors $u_1,\ldots,u_d$ and a nonempty block $B\subseteq[d]$, set
\[
a_B=P_n\prod_{j\in B}\langle X,u_j\rangle,
\qquad
b_B=\E\prod_{j\in B}\langle X,u_j\rangle.
\]
The preceding bounds imply $|a_B|+|b_B|\le C_{\beta,d}$ and $|a_B-b_B|\le C_{\beta,d}b$, uniformly over all such vectors and blocks.
For $\pi=\{B_1,\ldots,B_m\}$, use the exact telescoping identity
\[
\prod_{\ell=1}^m a_{B_\ell}-\prod_{\ell=1}^m b_{B_\ell}
=\sum_{k=1}^m
\left(\prod_{\ell<k}a_{B_\ell}\right)
(a_{B_k}-b_{B_k})
\left(\prod_{\ell>k}b_{B_\ell}\right).
\]
Each summand has absolute value at most $C_{\beta,d}b$.
The moment--cumulant inversion identity \eqref{eq_cumulant_moment_inversion}, followed by the supremum over the contraction vectors, therefore proves $\|\widehat{\bcK}_d-\bcK_d\|\le C_{\beta,d}b$.

For the central moment, the expansion valid for every integer $d\ge1$ is
\begin{equation}\label{eq_expansion_central_moments}
\E\prod_{j=1}^d\langle X-\mu,u_j\rangle
=\sum_{S\subseteq[d]}(-1)^{d-|S|}
\left(\E\prod_{j\in S}\langle X,u_j\rangle\right)
\prod_{j\notin S}\E\langle X,u_j\rangle,
\end{equation}
where an empty product is $1$.
The empirical central moment has the identical expansion with every expectation replaced by $P_n$.
Each factor is a raw-moment contraction of order at most $d$, or the constant $1$.
The same telescoping identity bounds the difference of each pair of corresponding summands by $C_{\beta,d}b$.
Summing over the $2^d$ subsets and taking the supremum over $u_1,\ldots,u_d$ proves the central-moment error bound.
Together with the raw-moment bound and the simultaneous-event probability, this proves both theorems.
\end{proof}

\subsection{Proof of Theorem \ref{thm_rate_op_moment_multilinear_form_estimator}}

The proof uses three ingredients: Lemma~\ref{lemma_subexpo_concentration} for the empirical branch; Cauchy--Schwarz and H\"older's inequalities for the truncation bias; and variance-sensitive Bernstein concentration, followed by an explicit bound on the contribution of the truncation radius.

\begin{proof}
Fix $\zeta>0$ and put $W_i=\prod_{j=1}^d\langle X_i,u_j\rangle$ and $\alpha=\min\{\beta/d,1\}$.
The scalar calculation in the proof of Theorem~\ref{thm_upper_bound_sample_moments_appendix} establishes $\|W_i-\E W_i\|_{\Phi_\alpha}\le C_{\beta,d}$ for every $\beta>0$.
These variables are independent across observations and centered, so the hypotheses of Lemma~\ref{lemma_subexpo_concentration} hold with coefficients $a_i=1/n$.
Equivalently, \eqref{eq_scalar_moment_concentration_upper} implies, for every $x>0$,
\begin{equation}\label{eq_logfree_empirical_directional_tail}
\PP\left(
\left|\frac1n\sum_{i=1}^nW_i-\E W_i\right|
>C_{\beta,d}\left\{\sqrt{\frac xn}+\frac{x^{1/\alpha}}n\right\}
\right)\le2e^{-x}.
\end{equation}
Indeed, for the displayed threshold $t$, both $nt^2$ and $(nt)^\alpha$ are at least a sufficiently large constant times $x$.
In the empirical branch, $n>\max\{p,p^{2d/\beta-1}\}=p^{2/\alpha-1}$, and hence $p^{1/\alpha}/n\le\sqrt{p/n}$.
Taking $x=\zeta p$ in \eqref{eq_logfree_empirical_directional_tail} proves the asserted empirical-branch bound.
When $n\le p$, the estimator is zero and H\"older's inequality gives
\[
|\E W_i|
\le\prod_{j=1}^d\big(\E|\langle X_i,u_j\rangle|^d\big)^{1/d}
\le2\Gamma(1+d/\beta).
\]
Here and below the directional moment bound follows by integrating the tail $\PP(|\langle X,u\rangle|>t)\le2e^{-t^\beta}$.
This proves the zero-branch bound deterministically.

\medskip
\noindent\textbf{Truncation bias and variance.}
Suppose $p<n\le p^{2d/\beta-1}$, and let
\[
A_i=\left\{\max_{j\le d}|\langle X_i,u_j\rangle|\le R\right\},
\qquad Z_i=W_i\mathbf1_{A_i}.
\]
The estimated multilinear form is $n^{-1}\sum_iZ_i$.
For any $k>0$, tail integration gives
\[
\E|\langle X,u\rangle|^k
\le2k\int_0^\infty t^{k-1}e^{-t^\beta}\,dt
=2\Gamma(1+k/\beta).
\]
H\"older's inequality therefore proves
\[
\E W_i^2
\le\prod_{j=1}^d\big(\E|\langle X_i,u_j\rangle|^{2d}\big)^{1/d}
\le V_{\beta,d}:=2\Gamma(1+2d/\beta).
\]
In particular, $\operatorname{Var}(Z_i)\le\E Z_i^2\le V_{\beta,d}$, even though the almost-sure bound $R^d$ may grow with $n/p$.
The union bound gives $\PP(A_i^c)\le2d e^{-R^\beta}$, so Cauchy--Schwarz yields
\begin{equation}\label{eq_logfree_directional_bias}
|\E Z_i-\E W_i|
\le\sqrt{2dV_{\beta,d}}\,e^{-R^\beta/2}
=\sqrt{2dV_{\beta,d}}\,(en/p)^{-c_{\beta,\zeta}/2}
\le C_{\beta,d}\sqrt{p/n}.
\end{equation}
The last inequality uses $n>p$ and $c_{\beta,\zeta}\ge2$.

\medskip
\noindent\textbf{Variance-sensitive concentration.}
We use Bernstein's inequality in the following form: if independent centered real variables $Y_i$ satisfy $|Y_i|\le b$ and $\sum_i\E Y_i^2\le v$, then, for every $x>0$,
\begin{equation}\label{eq_logfree_bernstein}
\PP\left(\left|\sum_iY_i\right|>
\sqrt{2vx}+\frac{2bx}{3}\right)\le2e^{-x}.
\end{equation}
For completeness, $j!\ge2\cdot3^{j-2}$ for $j\ge2$ and $\E|Y_i|^j\le b^{j-2}\E Y_i^2$ imply, for $0<\lambda<3/b$,
\[
\log\E e^{\lambda Y_i}
\le\sum_{j=2}^\infty\frac{\lambda^j\E|Y_i|^j}{j!}
\le\frac{\lambda^2\E Y_i^2}{2(1-\lambda b/3)}.
\]
The first inequality uses centering and $\log(1+a)\le a$ for $a\ge0$.
Independence, the exponential Markov inequality, and the choice $\lambda=t/(v+bt/3)$ for $v>0$ give
\[
\PP\left(\sum_iY_i>t\right)
\le\exp\left(-\frac{t^2}{2(v+bt/3)}\right).
\]
Applying this argument to $-Y_i$ gives the two-sided version.
The choice $t=\sqrt{2vx}+2bx/3$ satisfies $t^2\ge2x(v+bt/3)$, proving \eqref{eq_logfree_bernstein}; if $v=0$, all the centered variables vanish almost surely.

Apply \eqref{eq_logfree_bernstein} to $Y_i=Z_i-\E Z_i$.
The variables are independent, $|Y_i|\le2R^d$, and $\sum_i\E Y_i^2\le nV_{\beta,d}$.
Combining the resulting bound with \eqref{eq_logfree_directional_bias}, for every $x>0$ we obtain
\begin{equation}\label{eq_logfree_truncated_directional_tail}
\PP\left(
\left|\frac1n\sum_{i=1}^nZ_i-\E W_i\right|
>C_{\beta,d}\sqrt{\frac pn}
+\sqrt{\frac{2V_{\beta,d}x}{n}}+\frac{4R^dx}{3n}
\right)\le2e^{-x}.
\end{equation}
The radius $R$ is fixed by $p,n$ and is not changed with the tail parameter $x$.

To control its contribution, write $a=d/\beta$ and $y=\log(n/p)\ge0$.
Then
\[
\frac{R^d(p/n)}{\sqrt{p/n}}
=c_{\beta,\zeta}^{a}(1+y)^ae^{-y/2}.
\]
The logarithmic derivative of $(1+y)^ae^{-y/2}$ is $a/(1+y)-1/2$.
Its maximum on $[0,\infty)$ is therefore $1$ if $a\le1/2$, and $(2a)^ae^{-a+1/2}$ if $a>1/2$.
Consequently,
\begin{equation}\label{eq_logfree_truncation_absorption}
R^d\frac pn\le C_{\beta,d,\zeta}\sqrt{\frac pn},
\qquad n>p.
\end{equation}
Take $x=\zeta p$ in \eqref{eq_logfree_truncated_directional_tail} and use \eqref{eq_logfree_truncation_absorption}.
All three terms in its threshold are bounded by $C_{\beta,d,\zeta}\sqrt{p/n}$.
Together with the empirical and zero branches, this proves the theorem.
\end{proof}

\subsection{Proof of Theorem \ref{thm_rate_op_cumulant_estimator} and Theorem \ref{thm_rate_op_moment_estimator_appendix}}

The four nontrivial steps are the symmetric tensor norm identity in Lemma~\ref{lemma_symmetric_tensor_approximation} and a diagonal-net bound, uniform directional tail bounds and their integration, control of empty feasible sets using the estimator's sample-size branches, and product telescoping in the moment--cumulant formula together with H\"older's inequality.

\begin{proof}
We first prove the moment-estimation assertions, including the higher error moments needed for cumulants.
Fix the maximum order $d\ge1$, let $1\le r\le d$, and use the notation of Section~\ref{sec_optimal_estimator}.
The choice $r=d$ gives the standalone moment estimator in Theorem~\ref{thm_rate_op_moment_estimator_appendix}.
Put
\[
s=\sqrt{p/n}\wedge1,
\qquad
\Delta_r=\max_{u\in\mathcal N_r}
|\widetilde m_{r,u}-\langle\bcM_r,u^{\otimes r}\rangle|,
\qquad
E_r=\|\widehat{\bcM}_r^\star-\bcM_r\|.
\]
All nets are deterministic and satisfy $|\mathcal N_r|\le(1+4r)^p$.
The tolerance is $\tau_r=A_{\beta,r,d}s$, with its coefficient at least the constant from Theorem~\ref{thm_rate_op_moment_multilinear_form_estimator} at order $r$ and $\zeta=3+\log(1+4d)$.
H\"older's inequality and the directional moment bound give
\begin{equation}\label{eq_logfree_population_moment_bound}
\|\bcM_r\|
\le\sup_{u_1,\ldots,u_r\in\mathbb S^{p-1}}
\prod_{j=1}^r\big(\E|\langle X,u_j\rangle|^r\big)^{1/r}
\le B_{\beta,r}:=2\Gamma(1+r/\beta).
\end{equation}

\medskip
\noindent\textbf{Diagonal-net control and feasibility.}
Lemma~\ref{lemma_symmetric_tensor_approximation} implies that, for every real symmetric order-$r$ tensor $\bcT$,
\begin{equation}\label{eq_logfree_symmetric_diagonal_norm}
\|\bcT\|=\sup_{\|v\|=1}|\langle\bcT,v^{\otimes r}\rangle|.
\end{equation}
To match the lemma to this identity, for a unit rank-one tensor $\bcA$ one has
\[
\min_{a\in\R}\|\bcT-a\bcA\|_F^2
=\|\bcT\|_F^2-\langle\bcT,\bcA\rangle^2.
\]
A symmetric best rank-one approximation therefore attains the maximum absolute multilinear contraction on a diagonal direction.
For $r=1$, \eqref{eq_logfree_symmetric_diagonal_norm} is simply the dual representation of the Euclidean norm.

For $v\in\mathbb S^{p-1}$ choose $u\in\mathcal N_r$ with $\|v-u\|\le1/(2r)$.
The exact identity
\[
v^{\otimes r}-u^{\otimes r}
=\sum_{j=1}^r
u^{\otimes(j-1)}\otimes(v-u)\otimes v^{\otimes(r-j)}
\]
and multilinearity bound the absolute contraction against $\bcT$ by $r\|v-u\|\,\|\bcT\|\le\|\bcT\|/2$.
Taking the supremum in \eqref{eq_logfree_symmetric_diagonal_norm} proves
\begin{equation}\label{eq_logfree_diagonal_net}
\|\bcT\|\le2\max_{u\in\mathcal N_r}|\langle\bcT,u^{\otimes r}\rangle|.
\end{equation}
Both $\widehat{\bcM}_r^\star$ and $\bcM_r$ are symmetric, so this bound applies to their difference.

When $\mathcal S_r$ is nonempty, it is a closed convex subset of a finite-dimensional tensor space.
The squared Frobenius norm is coercive and strictly convex, so its minimum on $\mathcal S_r$ exists and is unique.
Feasibility and \eqref{eq_logfree_diagonal_net} give the deterministic implication
\begin{equation}\label{eq_logfree_feasible_error}
\mathcal S_r\ne\varnothing
\quad\Longrightarrow\quad
E_r\le2(\Delta_r+\tau_r).
\end{equation}
Indeed, each net contraction of $\widehat{\bcM}_r^\star-\bcM_r$ is bounded in absolute value by $\tau_r+\Delta_r$ after adding and subtracting $\widetilde m_{r,u}$.
If $\Delta_r\le\tau_r$, then $\bcM_r\in\mathcal S_r$ and $E_r\le4\tau_r$.

For $n>p$, the union bound and Theorem~\ref{thm_rate_op_moment_multilinear_form_estimator} yield
\begin{equation}\label{eq_logfree_good_event}
\PP(\Delta_r>\tau_r)
\le2\exp\big((\log(1+4r)-\zeta)p\big)
\le2e^{-3p}.
\end{equation}
It follows that $E_r\le4\tau_r$ for every $r\le d$ simultaneously with probability at least $1-2d e^{-3p}$.
When $n\le p$, the directional estimates are all zero, zero is feasible and minimizes the Frobenius norm, and $\widehat{\bcM}_r^\star=0$.
Equation~\eqref{eq_logfree_population_moment_bound} then proves the asserted high-probability moment bounds deterministically.

\medskip
\noindent\textbf{Higher moments of the directional error.}
Suppose $n>p$ and put $\alpha_r=\min\{\beta/r,1\}$.
Apply \eqref{eq_logfree_empirical_directional_tail} at order $r$ with $x=(1+\log(1+4r))p+z$, and take a union bound over $\mathcal N_r$.
For every $z\ge0$, in the empirical branch this gives
\begin{equation}\label{eq_logfree_uniform_empirical_tail}
\PP\left(
\Delta_r>C_{\beta,r,d}\left\{
\sqrt{\frac{p+z}{n}}+\frac{(p+z)^{1/\alpha_r}}n
\right\}\right)\le2e^{-z}.
\end{equation}
Here the factor $e^{-p}$ left after the union bound has been discarded, and the fixed factor $1+\log(1+4r)$ is absorbed into the threshold constant.
The same union bound applied to \eqref{eq_logfree_truncated_directional_tail} at order $r$ gives, in the truncated branch,
\begin{equation}\label{eq_logfree_uniform_truncated_tail}
\PP\left(
\Delta_r>C_{\beta,r,d}\left\{
\sqrt{\frac pn}+\sqrt{\frac{p+z}{n}}+R_r^r\frac{p+z}{n}
\right\}\right)\le2e^{-z}.
\end{equation}
The first term is the truncation bias, and $R_r$ remains fixed as $z$ varies.

Since $p+z\le p(1+z)$, the empirical-branch condition $n>p^{2/\alpha_r-1}$ bounds both threshold terms in \eqref{eq_logfree_uniform_empirical_tail} by a constant times $s(1+z)^{1/\alpha_r}$.
In the truncated branch, \eqref{eq_logfree_truncation_absorption} at order $r$ gives $R_r^rp/n\le C_{\beta,r,d}s$, and $1/\alpha_r\ge1$ bounds all three terms in \eqref{eq_logfree_uniform_truncated_tail} the same way.
Thus in both branches,
\begin{equation}\label{eq_logfree_uniform_directional_tail}
\PP\big(\Delta_r>C_{\beta,r,d}s(1+z)^{1/\alpha_r}\big)
\le2e^{-z},\qquad z\ge0.
\end{equation}

For any fixed $q\ge1$, put $a=1/\alpha_r$ and $A=C_{\beta,r,d}s$.
Integrating \eqref{eq_logfree_uniform_directional_tail} using $t=A(1+z)^a$ yields
\begin{align}
\E\Delta_r^q
&=\int_0^\infty qt^{q-1}\PP(\Delta_r>t)\,dt\notag\\
&\le A^q+2aqA^q\int_0^\infty(1+z)^{aq-1}e^{-z}\,dz\notag\\
&\le C_{\beta,r,d,q}s^q.
\label{eq_logfree_directional_error_moments}
\end{align}
The integral equals $e\int_1^\infty t^{aq-1}e^{-t}\,dt\le e\Gamma(aq)$, which is finite and depends only on the fixed parameters.

\medskip
\noindent\textbf{Empty feasible sets and expected error.}
On $\{\mathcal S_r=\varnothing\}$, the fallback estimator is zero, so $E_r\le B_{\beta,r}$ by \eqref{eq_logfree_population_moment_bound}.
On its complement, use \eqref{eq_logfree_feasible_error} and $(x+y)^q\le2^{q-1}(x^q+y^q)$ for nonnegative $x,y$ and $q\ge1$.
Consequently,
\begin{equation}\label{eq_logfree_error_moment_reduction}
\E E_r^q
\le C_q\big(\E\Delta_r^q+\tau_r^q\big)
+B_{\beta,r}^q\PP(\mathcal S_r=\varnothing).
\end{equation}
In the empirical branch the feasible set is never empty: the sample tensor $n^{-1}\sum_iX_i^{\otimes r}$ is symmetric and exactly matches $\widetilde m_{r,u}$ at every net point.
Therefore the last term in \eqref{eq_logfree_error_moment_reduction} is zero in that branch, for every sample size allowed there.

In the truncated branch, $\mathcal S_r=\varnothing$ implies $\Delta_r>\tau_r$, so \eqref{eq_logfree_good_event} bounds its probability by $2e^{-3p}$.
This branch also satisfies $p<n\le p^{2r/\beta-1}$ and can occur only when $r/\beta>1$.
Writing $\nu=q(r/\beta-1)>0$, we have
\[
s^q=(p/n)^{q/2}\ge p^{-\nu},
\qquad
\frac{e^{-3p}}{s^q}\le p^\nu e^{-3p}
\le\left(\frac{\nu}{3e}\right)^\nu.
\]
The final inequality follows by maximizing $x^\nu e^{-3x}$ over $x>0$, whose derivative vanishes at $x=\nu/3$.
This proves that the empty-feasible-set contribution is also bounded by $C_{\beta,r,d,q}s^q$.
Combining \eqref{eq_logfree_directional_error_moments} and \eqref{eq_logfree_error_moment_reduction}, and using the deterministic zero estimator when $n\le p$, proves for all $p,n\ge1$ that
\begin{equation}\label{eq_logfree_moment_error_moments}
\E\|\widehat{\bcM}_r^\star-\bcM_r\|^q
\le C_{\beta,r,d,q}(\sqrt{p/n}\wedge1)^q,
\qquad q\ge1\text{ fixed}.
\end{equation}
Taking $r=d$ and $q=1$, together with the previously proved high-probability bound, completes Theorem~\ref{thm_rate_op_moment_estimator_appendix}.
No additive exceptional-event term remains in its expected risk.

\medskip
\noindent\textbf{Transfer to cumulants.}
Now assume $d\ge3$ and set $E=\max_{r\le d}E_r$.
Using \eqref{eq_logfree_moment_error_moments} with $q=d$ gives
\begin{equation}\label{eq_logfree_max_moment_error}
\E E^d\le\sum_{r=1}^d\E E_r^d\le C_{\beta,d}s^d.
\end{equation}
Let $B=\max\{1,B_{\beta,1},\ldots,B_{\beta,d}\}$.
Then $\|\bcM_r\|\le B$ and $\|\widehat{\bcM}_r^\star\|\le B+E$ for every $r\le d$.
For any partition $\pi=\{B_1,\ldots,B_m\}$ of $[d]$ and any unit contraction vectors $u_1,\ldots,u_d$, write
\[
a_\ell=\left\langle\widehat{\bcM}_{|B_\ell|}^\star,
\bigotimes_{j\in B_\ell}u_j\right\rangle,
\qquad
b_\ell=\left\langle\bcM_{|B_\ell|},
\bigotimes_{j\in B_\ell}u_j\right\rangle,
\]
with block indices in increasing order.
We have $|a_\ell-b_\ell|\le E$, $|b_\ell|\le B$, and $|a_\ell|\le B+E$.
The product identity
\[
\prod_{\ell=1}^m a_\ell-\prod_{\ell=1}^m b_\ell
=\sum_{k=1}^m
\left(\prod_{\ell<k}a_\ell\right)(a_k-b_k)
\left(\prod_{\ell>k}b_\ell\right)
\]
therefore bounds this partition's contraction error by $mE(B+E)^{m-1}$.
The inversion formula \eqref{eq_cumulant_moment_inversion} expresses the cumulant plug-in error as the sum of these partition differences with coefficients $(-1)^{m-1}(m-1)!$.
Summing their absolute bounds and taking the supremum over all unit contraction vectors yields
\begin{equation}\label{eq_logfree_cumulant_error_transfer}
\|\widehat{\bcK}_d^\star-\bcK_d\|
\le C_{\beta,d}\sum_{j=1}^d E^j.
\end{equation}
There are finitely many partitions for fixed $d$, and expanding $E(B+E)^{m-1}$ gives only powers $E^j$ with $1\le j\le m\le d$.

For $n>p$, \eqref{eq_logfree_good_event} and a union bound over $r\le d$ imply $E\le C_{\beta,d}s$ with probability at least $1-2d e^{-3p}$.
Since $s\le1$, \eqref{eq_logfree_cumulant_error_transfer} gives the claimed high-probability cumulant bound.
For $n\le p$, the moment reconstructions are all zero; \eqref{eq_logfree_population_moment_bound} and the same partition expansion give the required constant bound deterministically.
Finally, for every $1\le j\le d$, H\"older's inequality and \eqref{eq_logfree_max_moment_error} imply
\[
\E E^j\le(\E E^d)^{j/d}
\le C_{\beta,d}s^j\le C_{\beta,d}s.
\]
Taking expectations in \eqref{eq_logfree_cumulant_error_transfer} proves the asserted expected cumulant error.
No independence between the moment estimators is needed for these inequalities.
This completes Theorem~\ref{thm_rate_op_cumulant_estimator}.
\end{proof}

\subsection{Proof of Lemma \ref{lemma_ldp}}

The proof uses a Gaussian density comparison, the Hermite generating function, cancellation under the random-sign prior, and a count of the remaining multi-indices.
\begin{proof}
We first check integrability and the uniform tail bound used in Section~\ref{sec_hypothesis_test}.
The construction in Lemma~\ref{lemma_existance_1d} gives a density $f_W=h\phi$ with $0<\ell_d\le h\le B_d<\infty$, where $\phi$ is the standard Gaussian density and we may take $B_d\ge1$.
Fix $U$ and write $u=U/\sqrt p$ and $\lambda=a/(1+a)$.
Along $u$, the observation has law $\sqrt\lambda W+\sqrt{1-\lambda}G$, while the orthogonal component is independent standard Gaussian noise.
For independent standard Gaussians $G_0,G_1$, the density ratio of the scalar component relative to $N(0,1)$ is
\[
h_a(t)=\E\big[h(G_0)\mid \sqrt\lambda G_0+\sqrt{1-\lambda}G_1=t\big].
\]
This identity follows by integrating any bounded test function against the tilted joint law $h(G_0)d\mathbb P(G_0,G_1)$ and conditioning on the displayed linear combination.
It proves $\ell_d\le h_a\le B_d$, independently of $a,U$.
Hence the one-observation density ratio $dP_{U,a}/dN(0,I_p)$ lies between the same constants, and $0\le L\le B_d^n$ for the sample-level mixture.
In particular, $L\in L^2(\mathbb Q_{p,n})$ and all the polynomial moments used below are finite.
Also, for every unit vector $v$ and $K^2>2$,
\[
\E_{P_{U,a}}\exp\big((v^\top Y)^2/K^2\big)
\le1+B_d\left\{(1-2/K^2)^{-1/2}-1\right\}.
\]
Here we apply the density bound to the nonnegative function $\exp((v^\top y)^2/K^2)-1$.
Choosing $K=K_d$ sufficiently large makes the right side at most $2$, proving the claimed uniform directional sub-Gaussian bound.

Write $\alpha=(\alpha_{ij})\in\mathbb N_0^{n\times p}$, $k_i=\sum_j\alpha_{ij}$, $|\alpha|=\sum_i k_i$, and $\alpha!=\prod_{i,j}\alpha_{ij}!$.
Let $H_\alpha(Y)=\prod_{i,j}H_{\alpha_{ij}}((Y_i)_j)$, where $H_k$ denotes the probabilists' Hermite polynomial and $H_0=1$.
The normalized Hermite polynomials $H_\alpha/\sqrt{\alpha!}$ form an orthonormal basis of the degree-$D$ polynomial space under $\mathbb Q_{p,n}=N(0,I_{np})$ when $|\alpha|\le D$.
This is the Gaussian Hermite expansion used in \cite[Section~2.3]{kunisky2019notes}.
Consequently,
\[
\|L_{\le D}\|^2
=\sum_{|\alpha|\le D}\frac{(\E_{\mathbb P_{p,n,a}}H_\alpha(Y))^2}{\alpha!}.
\]

To compute a coefficient, conditional on $U$ the Gaussian component of $Y_i$ has covariance $I_p-\lambda uu^\top$.
The generating identity $\exp(tx-t^2/2)=\sum_{k\ge0}H_k(x)t^k/k!$ therefore gives
\[
\E\big[\exp(t^\top Y_i-\|t\|^2/2)\mid U\big]
=\E\exp\big(\sqrt\lambda\,(u^\top t)W-\lambda(u^\top t)^2/2\big).
\]
Differentiating at $t=0$ and comparing coefficients yields
\[
\E[H_{\alpha^{(i)}}(Y_i)\mid U]
=\lambda^{k_i/2}p^{-k_i/2}U^{\alpha^{(i)}}\E H_{k_i}(W).
\]
All derivatives are justified by the Gaussian density bound on $W$.
Conditional independence across observations now proves
\begin{equation}\label{eq_LD_bound1}
\|L_{\le D}\|^2
=\sum_{|\alpha|\le D}
\frac{\lambda^{|\alpha|}}{\alpha!p^{|\alpha|}}
\left(\prod_{i=1}^n\E H_{k_i}(W)\right)^2
\left(\E_U\prod_{j=1}^pU_j^{\sum_i\alpha_{ij}}\right)^2.
\end{equation}
The matching Gaussian moments give
\begin{equation}\label{eq_W_hermite_0}
\E H_0(W)=1,
\qquad
\E H_k(W)=0\quad(1\le k\le d-1).
\end{equation}
The independent sign coordinates also give
\begin{equation}\label{eq_exp_bound}
\E_U\prod_{j=1}^pU_j^{\sum_i\alpha_{ij}}
=\mathbf1_{\{\sum_i\alpha_{ij}\ \mathrm{is\ even\ for\ every}\ j\}}.
\end{equation}

For $m\ge1$, let $\mathcal A_m$ be the multi-indices with $|\alpha|=m$, every positive row sum at least $d$, and every column sum even.
Only these indices can contribute at total degree $m$.
Every such array has at most $m/d$ nonzero rows and $m/2$ nonzero columns.
For fixed support sizes $r,s\ge1$, choose the supports in at most $\binom nr\binom ps$ ways.
There are at most $(rs)^m$ arrays of total mass $m$ on those supports, since every array can be obtained by placing $m$ labeled units in its $rs$ cells.
Thus
\begin{equation}\label{eq_set_size}
\begin{aligned}
|\mathcal A_m|
&\le\sum_{r=1}^{\min(n,\lfloor m/d\rfloor)}
\sum_{s=1}^{\min(p,\lfloor m/2\rfloor)}
\binom nr\binom ps(rs)^m\\
&\le m^2 n^{m/d}p^{m/2}m^{2m}
\le n^{m/d}p^{m/2}(4m^2)^m.
\end{aligned}
\end{equation}
The last two inequalities use $\binom nr\le n^r$, $\binom ps\le p^s$, $rs\le m^2$, and $m^2\le4^m$ for $m\ge1$.
An empty sum is zero, so the argument also covers $m<d$ and imposes no restriction comparing $D$ with $n$ or $p$.

Finally, Lemma~\ref{lemma_existance_1d} gives $|\E H_k(W)|\le k!C_d^k$ for $k\ge1$.
For $\alpha\in\mathcal A_m$, the zero row sums contribute $\E H_0(W)=1$, while $k_i!\le m^{k_i}$ for each positive row sum.
It follows that
\begin{equation}\label{eq_W_cumulant_bound_1}
\prod_{i=1}^n|\E H_{k_i}(W)|^2
\le\prod_{i:k_i>0}(C_dm)^{2k_i}
=(C_dm)^{2m}.
\end{equation}
The degree-zero coefficient in \eqref{eq_LD_bound1} is exactly one.
Using $\alpha!\ge1$, \eqref{eq_set_size}, and \eqref{eq_W_cumulant_bound_1} bounds all remaining coefficients by
\[
\begin{aligned}
\|L_{\le D}\|^2
&\le1+\sum_{m=1}^D
\lambda^mp^{-m}|\mathcal A_m|(C_dm)^{2m}\\
&\le1+\sum_{m=1}^D
\left(\frac a{1+a}\frac{C'_dm^4n^{1/d}}{p^{1/2}}\right)^m.
\end{aligned}
\]
Renaming $C'_d$ proves the lemma.
\end{proof}

\subsection{Proof of Theorem \ref{thm_poly_test_moment_corrected}}

\begin{proof}
The computational complexity is immediate from the definition of $V_n$: one first computes the Gram matrix
\[
G_{ij}=X_i^\top X_j,
\]
and then averages the off-diagonal entries $G_{ij}^d$. This requires $O(n^2p)$ arithmetic operations.

We now analyze the mean and variance of $V_n$.

\paragraph{Step 1: Mean.}
By independence of $X_1$ and $X_2$,
\begin{align}
\E\big[(X_1^\top X_2)^d\big]
&=
\E\Big[\Big(\sum_{j=1}^p (X_{1})_j (X_{2})_j\Big)^d\Big] =
\sum_{i_1,\ldots,i_d\in[p]}
\E\big[(X_1)_{i_1}\cdots (X_1)_{i_d}\big]^2 =
\|\bcM_d\|_F^2. \label{eq_mean_Un_new}
\end{align}
Therefore,
\[
\E V_n=\|\bcM_d\|_F^2.
\]
In particular, under $H_0$,
\[
\E V_n=0,
\]
while under $H_1$,
\[
\E V_n=\|\bcM_d\|_F^2\ge \|\bcM_d\|^2\ge \kappa^2.
\]

\paragraph{Step 2: Hoeffding-type decomposition.}
Set
\[
\theta:=\E[(X_1^\top X_2)^d]=\|\bcM_d\|_F^2.
\]
For $x\in\R^p$, define
\[
g(x):=\E[(x^\top X_2)^d]-\theta
=\langle \bcM_d,x^{\otimes d}\rangle-\|\bcM_d\|_F^2.
\]
Also define, for $x,y\in\R^p$,
\[
R(x,y):=(x^\top y)^d-\theta-g(x)-g(y).
\]
Then
\[
V_n-\theta
=
\frac{2}{n}\sum_{i=1}^n g(X_i)
+
\frac{2}{n(n-1)}\sum_{1\le i<j\le n} R(X_i,X_j).
\]
Moreover,
\[
\E[g(X_1)]=0,
\]
and
\begin{align*}
\E[R(X_1,X_2)\mid X_1]
&=
\E[(X_1^\top X_2)^d\mid X_1]-\theta-g(X_1)-\E[g(X_2)]\\
&=
(\theta+g(X_1))-\theta-g(X_1)-0\\
&=0.
\end{align*}
Thus, 
\[
\E[g(X_1)]=0,
\qquad
\E[R(X_1,X_2)\mid X_1]=\E[R(X_1,X_2)\mid X_2]=0.
\]
Hence the two terms are orthogonal, and
\begin{equation}\label{eq_var_decomp_new}
\Var(V_n)
=
\frac{4}{n}\Var(g(X_1))
+
\frac{2}{n(n-1)}\Var(R(X_1,X_2)).
\end{equation}

We now bound the two variance terms.

\paragraph{Step 3: Bound on the second-order remainder.}
We first show that
\[
\E[(X_1^\top X_2)^{2d}]
\le C_d K^{4d}p^d.
\]
Conditioning on $X_1$, the scalar $X_1^\top X_2$ is sub-Gaussian with
\[
\|X_1^\top X_2\|_{\Phi_2\,|\,X_1}
\le K\|X_1\|,
\]
hence
\[
\E[(X_1^\top X_2)^{2d}\mid X_1]
\le C_d K^{2d}\|X_1\|^{2d}.
\]
Taking expectation gives
\[
\E[(X_1^\top X_2)^{2d}]
\le C_d K^{2d}\E\|X\|^{2d}.
\]
To bound $\E\|X\|^{2d}$, let $g\sim N(0,I_p)$ be independent of $X$. Since
\[
\E_g[(g^\top x)^{2d}]=c_d\|x\|^{2d},
\qquad \forall x\in\R^p,
\]
we have
\[
\E\|X\|^{2d}
=
c_d^{-1}\E[(g^\top X)^{2d}].
\]
Conditioning on $g$, the scalar $g^\top X$ is sub-Gaussian with
\[
\|g^\top X\|_{\Phi_2\,|\,g}\le K\|g\|,
\]
so
\[
\E[(g^\top X)^{2d}\mid g]
\le C_d K^{2d}\|g\|^{2d}.
\]
Therefore
\[
\E\|X\|^{2d}
\le C_d K^{2d}\E\|g\|^{2d}
\le C_d K^{2d}p^d,
\]
which implies
\begin{equation}\label{eq_kernel_2d_new}
\E[(X_1^\top X_2)^{2d}]
\le C_d K^{4d}p^d.
\end{equation}

Next, since
\[
R(X_1,X_2)
=
(X_1^\top X_2)^d-\theta-g(X_1)-g(X_2),
\]
we have
\[
R(X_1,X_2)^2
\le
4\Big(
(X_1^\top X_2)^{2d}
+\theta^2
+g(X_1)^2
+g(X_2)^2
\Big).
\]
Taking expectations and using Jensen's inequality,
\[
\E[g(X_1)^2]\le \E[(X_1^\top X_2)^{2d}],
\qquad
\theta^2\le \E[(X_1^\top X_2)^{2d}],
\]
we obtain
\[
\E[R(X_1,X_2)^2]
\le
16\,\E[(X_1^\top X_2)^{2d}]
\le
C_d K^{4d}p^d.
\]
Hence
\begin{equation}\label{eq_R_variance_new}
\Var(R(X_1,X_2))
\le \E[R(X_1,X_2)^2]
\le C_d K^{4d}p^d.
\end{equation}

\paragraph{Step 4: Type-I error under $H_0$.}
Under $H_0$, we have $\bcM_d=0$, so by definition of $g$,
\[
g(X_1)=\langle \bcM_d,X_1^{\otimes d}\rangle-\|\bcM_d\|_F^2=0.
\]
Thus the first variance term in \eqref{eq_var_decomp_new} vanishes, and by \eqref{eq_R_variance_new},
\[
\Var(V_n)
=
\frac{2}{n(n-1)}\Var(R(X_1,X_2))
\le
C_d K^{4d}\frac{p^d}{n(n-1)}.
\]
Since $\E V_n=0$, Chebyshev's inequality yields
\[
\PP_{H_0}(V_n>\tau_n)
\le
\frac{\Var(V_n)}{\tau_n^2}
\le
\frac{C_d K^{4d}p^d/n^2}{A^2p^d(\log n)^2/n^2}
=
\frac{C_d K^{4d}}{A^2(\log n)^2}
\to 0.
\]
Therefore
\[
\PP_{H_0}(\phi_n=1)\to 0.
\]

\paragraph{Step 5: Type-II error under $H_1$.}
Under $H_1$, the assumption \eqref{eq_zeta1_assumption_new} gives
\[
\Var(g(X_1))
=
\Var\!\big(\langle \bcM_d,X_1^{\otimes d}\rangle-\|\bcM_d\|_F^2\big)
=
\Var\!\big(\langle \bcM_d,X_1^{\otimes d}\rangle\big)
\le C_1 p^{d/2}.
\]
Combining this with \eqref{eq_var_decomp_new} and \eqref{eq_R_variance_new}, we obtain
\[
\Var(V_n)
\le
C\frac{p^{d/2}}{n}
+
C_d K^{4d}\frac{p^d}{n(n-1)}.
\]
Hence, whenever
\[
\frac{p^{d/2}\log n}{n}\to 0,
\]
we have in particular
\[
\Var(V_n)\to 0.
\]

On the other hand, under $H_1$,
\[
\E V_n=\|\bcM_d\|_F^2\ge \kappa^2.
\]
Since $\tau_n\to 0$, for all sufficiently large $n$ we have
\[
\tau_n\le \frac{\kappa^2}{2}.
\]
Therefore
\begin{align*}
\PP_{H_1}(\phi_n=0)
&=
\PP_{H_1}(V_n\le \tau_n) \\
&\le
\PP_{H_1}\Big(V_n-\E V_n\le -(\E V_n-\tau_n)\Big) \\
&\le
\PP_{H_1}\Big(|V_n-\E V_n|\ge \frac{\kappa^2}{2}\Big) \\
&\le
\frac{4\,\Var(V_n)}{\kappa^4}
\to 0
\end{align*}
by Chebyshev's inequality.
This proves
\[
\PP_{H_1}(\phi_n=0)\to 0.
\]
\end{proof}

\subsection{Proof of Theorem \ref{thm_poly_test_cumulant}}

\begin{proof}
We first use independence and the positive-semidefinite covariance bound to control the signal-dependent variance, then calculate the U-statistic variance by checking its cross terms, and finally apply Chebyshev's inequality to the two testing errors.
Fix a distribution satisfying the assumptions and write
\[
Z_i=\operatorname{vec}(\Psi_d(Y_i)),\qquad
\mu=\E Z_i=\operatorname{vec}(\bcK_d),\qquad
\varepsilon_i=Z_i-\mu,
\]
\[
\Gamma=\E[\varepsilon_i\varepsilon_i^\top],
\qquad
\theta=\|\mu\|^2=\|\bcK_d\|_F^2.
\]
Vectorization preserves the Frobenius inner product, and the assumed second moments ensure that all entries of $\Gamma$ are finite.
Since $Z_1,Z_2$ are independent, expansion of the squared inner product gives
\begin{equation}\label{eq_shrinking_feature_second_moment}
\begin{aligned}
\E\langle\Psi_d(Y_1),\Psi_d(Y_2)\rangle^2
&=\E(Z_1^\top Z_2)^2\\
&=\sum_{a,b=1}^{p^d}\big(\E[Z_{1a}Z_{1b}]\big)^2\\
&=\sum_{a,b=1}^{p^d}(\Gamma_{ab}+\mu_a\mu_b)^2\\
&=\operatorname{tr}(\Gamma^2)+2\mu^\top\Gamma\mu+\theta^2.
\end{aligned}
\end{equation}
All terms in the last line are nonnegative because $\Gamma$ is positive semidefinite.
Thus \eqref{eq_kernel_second_moment_cumulant} implies $\operatorname{tr}(\Gamma^2)\le C_0p^d$.
For the nonnegative eigenvalues $\lambda_j$ of $\Gamma$, $\max_j\lambda_j\le(\sum_j\lambda_j^2)^{1/2}$, so
\begin{equation}\label{eq_shrinking_feature_variance_bound}
\mu^\top\Gamma\mu
\le\|\Gamma\|\,\|\mu\|^2
\le\sqrt{\operatorname{tr}(\Gamma^2)}\,\theta
\le\sqrt{C_0}\,p^{d/2}\theta.
\end{equation}

Expanding $Z_i=\mu+\varepsilon_i$ in \eqref{eq_Un_cumulant} yields $\E V_n^{(\mathrm{cumu})}=\theta$ and
\begin{equation}\label{eq_shrinking_cumulant_decomposition}
V_n^{(\mathrm{cumu})}-\theta
=
\frac2n\sum_{i=1}^n\mu^\top\varepsilon_i
+
\frac{2}{n(n-1)}\sum_{1\le i<j\le n}\varepsilon_i^\top\varepsilon_j.
\end{equation}
The variance of the first term is $4\mu^\top\Gamma\mu/n$.
For every $i\ne j$, independence gives
\[
\E(\varepsilon_i^\top\varepsilon_j)^2
=\sum_{a,b=1}^{p^d}\E[\varepsilon_{ia}\varepsilon_{ib}]\E[\varepsilon_{ja}\varepsilon_{jb}]
=\operatorname{tr}(\Gamma^2).
\]
The pairwise terms have mean zero.
Terms indexed by disjoint pairs are independent; if two distinct pairs share an index, conditioning on the shared observation gives, for distinct $i,j,k$,
\[
\E\big[(\varepsilon_i^\top\varepsilon_j)(\varepsilon_i^\top\varepsilon_k)\mid\varepsilon_i\big]
=\varepsilon_i^\top(\E\varepsilon_j)(\E\varepsilon_k)^\top\varepsilon_i=0.
\]
A linear term and a pairwise term also have zero covariance: they are independent when their indices are disjoint, and otherwise conditioning on the shared observation leaves an independent centered factor.
There are $n(n-1)/2$ unordered pairs, so \eqref{eq_shrinking_cumulant_decomposition} gives the exact variance identity
\begin{equation}\label{eq_shrinking_cumulant_variance}
\begin{aligned}
\operatorname{Var}(V_n^{(\mathrm{cumu})})
&=\frac4n\mu^\top\Gamma\mu
+\frac{2}{n(n-1)}\operatorname{tr}(\Gamma^2)\\
&\le\frac{4\sqrt{C_0}\,p^{d/2}\theta}{n}
+\frac{2C_0p^d}{n(n-1)},
\end{aligned}
\end{equation}
where the inequality follows from \eqref{eq_shrinking_feature_variance_bound} and the trace bound.

For a random variable $T$ with finite variance and $t>0$, Chebyshev's inequality states that $\PP(|T-\E T|\ge t)\le\operatorname{Var}(T)/t^2$.
Under $H_0$, $\theta=0$, and applying this inequality at $t=\rho^2/2$ gives
\begin{equation}\label{eq_shrinking_cumulant_type_I}
\PP\big(\phi_{n,\rho}^{(\mathrm{cumu})}=1\big)
\le\frac{8C_0p^d}{n(n-1)\rho^4}.
\end{equation}
Under $H_1(\rho)$, Cauchy--Schwarz for the tensor inner product gives $\|\bcK_d\|_F\ge\|\bcK_d\|$, so $\theta\ge\rho^2$ and $\theta-\rho^2/2\ge\theta/2$.
The acceptance event is therefore contained in $\{|V_n^{(\mathrm{cumu})}-\theta|\ge\theta/2\}$.
Chebyshev's inequality and \eqref{eq_shrinking_cumulant_variance} yield
\begin{equation}\label{eq_shrinking_cumulant_type_II}
\begin{aligned}
\PP\big(\phi_{n,\rho}^{(\mathrm{cumu})}=0\big)
&\le\frac{4\operatorname{Var}(V_n^{(\mathrm{cumu})})}{\theta^2}\\
&\le\frac{16\sqrt{C_0}\,p^{d/2}}{n\rho^2}
+\frac{8C_0p^d}{n(n-1)\rho^4}.
\end{aligned}
\end{equation}
The same $C_0$ applies throughout both hypotheses.
Taking suprema in \eqref{eq_shrinking_cumulant_type_I} and \eqref{eq_shrinking_cumulant_type_II}, adding, and using $n(n-1)\ge n^2/2$ proves \eqref{eq_shrinking_cumulant_testing_risk} with $C=16\sqrt{C_0}+32C_0$.
Both terms on its right tend to zero whenever $n\rho_{p,n}^2/p^{d/2}\to\infty$.

Finally, the identity
\[
V_n^{(\mathrm{cumu})}
=
\frac{\left\|\sum_{i=1}^n\Psi_d(Y_i)\right\|_F^2
-\sum_{i=1}^n\|\Psi_d(Y_i)\|_F^2}{n(n-1)}
\]
follows by expanding the squared norm and removing the diagonal terms.
Evaluating the $n$ features, accumulating their sum and squared norms, and computing the displayed expression costs $O_d(np^d)$ arithmetic operations under the stated evaluation assumption.
Comparison with the prescribed threshold completes the test.
\end{proof}

\subsection{Proof of Proposition~\ref{prop_cumulant_estimation_reduction}}
\label{app_proof_cumulant_reduction}

We first give the quantitative reduction, using the classes $\mathcal F_q$ from Conjecture~\ref{conj_consistent}.

\begin{lemma}[Quantitative reduction]\label{lem_cumulant_estimation_reduction}
Fix $d\ge3$, $p\ge1$, and $2d\le m\le n$.
Suppose an estimator $\widehat{\bcK}_{dp,n}$ based on $n$ i.i.d.\ observations has worst-case expected spectral-norm error $r(dp,n)$ over $\mathcal F_{dp}$.
There is a randomized estimator based on $m$ i.i.d.\ observations from $P\in\mathcal F_p$ satisfying
\begin{equation}\label{eq_cumulant_reduction_risk}
\sup_{P\in\mathcal F_p}
\E_P\left\|\widehat{\bcK}^{\mathrm{red}}_{p,m;n}-\bcK_d(P)\right\|
\le C_d\frac nm\,r(dp,n)+C_de^{-c_dm},
\end{equation}
where $c_d,C_d>0$ depend only on $d$.
The reduction makes at most one call to $\widehat{\bcK}_{dp,n}$, has additional cost polynomial in $(n,p)$, and never evaluates a tensor spectral norm.
\end{lemma}

The proof has four ingredients.
\begin{enumerate}
\item The partition identity \eqref{eq_cumulant_moment_inversion} and Rademacher sign cancellation encode the target as a mixed-block cumulant.
\item Jensen's inequality preserves the directional sub-Gaussian bound, while H\"older's inequality gives a dimension-independent bound on the target norm.
\item A binomial exponential-moment bound and an explicit coupling control the finite sample budget without conditioning the auxiliary experiment.
\item Coordinate-block extraction contracts the tensor spectral norm, allowing the auxiliary error bounds to be transferred by rescaling.
\end{enumerate}

\begin{proof}
\textbf{Step 1: An exact cumulant encoding.}
Let $\mathcal P([d])$ and $\mu(\pi)=(-1)^{|\pi|-1}(|\pi|-1)!$ be as in \eqref{eq_cumulant_moment_inversion}, and put
\[
A_d=\sum_{\pi\in\mathcal P([d])}|\mu(\pi)|,
\qquad
B_d=2^{\lceil\log_2 A_d\rceil},
\qquad
a_d=\frac{1}{d^dB_d}.
\]
These quantities depend only on $d$.
The power-of-two normalization permits exact sampling with a bounded number of fair random bits.

For $P\in\mathcal F_p$, write $X\sim P$ and let $V_1,\ldots,V_d$ be independent copies of $X$.
Independently, draw $\Pi\in\mathcal P([d])\cup\{\star\}$ with
\[
\PP(\Pi=\pi)=\frac{|\mu(\pi)|}{B_d},
\qquad
\PP(\Pi=\star)=1-\frac{A_d}{B_d}.
\]
If $\Pi=\star$, set $Z=0\in\R^{dp}$.
Otherwise, write $\pi=\{B_1,\ldots,B_k\}$, ordering the blocks by their smallest elements, and let $b_\pi(j)$ be the index of the block containing $j$.
Draw independent Rademacher signs $\varepsilon_1,\ldots,\varepsilon_{d-1}$, independently of the $V_i$, and set
\[
\varepsilon_d
=
\operatorname{sgn}(\mu(\pi))\prod_{j=1}^{d-1}\varepsilon_j,
\qquad
Z_j=\frac{\varepsilon_j}{d}V_{b_\pi(j)},
\qquad
Z=(Z_1^\top,\ldots,Z_d^\top)^\top.
\]
Let $\widetilde P$ be the law of $Z$.
One draw from $\widetilde P$ uses at most $d$ independent observations from $P$; for the sample-budget construction below, we reserve a group of $d$ observations for each such draw.

For every nonempty proper subset $J\subsetneq[d]$,
\[
\E\left[\prod_{j\in J}\varepsilon_j\,\middle|\,\Pi=\pi\right]=0.
\]
Indeed, when $d\notin J$, this product contains a nonempty set of distinct independent Rademacher signs.
When $d\in J$, it equals $\operatorname{sgn}(\mu(\pi))\prod_{j\in[d-1]\setminus J}\varepsilon_j$, again with at least one independent sign because $J$ is proper.
The signs are independent of the $V_i$ conditional on $\Pi$.
Consequently, for arbitrary coordinate indices $i_j\in[p]$,
\begin{equation}\label{eq_reduction_proper_moments}
\E_{\widetilde P}\prod_{j\in J}(Z_j)_{i_j}=0,
\qquad
\varnothing\ne J\subsetneq[d].
\end{equation}
For the full product, $\prod_{j=1}^d\varepsilon_j=\operatorname{sgn}(\mu(\pi))$, so independence of the $V_i$ gives
\begin{equation}\label{eq_reduction_encoded_moment}
\begin{aligned}
\E_{\widetilde P}\prod_{j=1}^d(Z_j)_{i_j}
&=
\frac{1}{d^dB_d}
\sum_{\pi\in\mathcal P([d])}
\mu(\pi)\prod_{B\in\pi}
\E_P\prod_{j\in B}X_{i_j}\\
&=a_d\big(\bcK_d(P)\big)_{i_1,\ldots,i_d},
\end{aligned}
\end{equation}
where the last equality is exactly \eqref{eq_cumulant_moment_inversion}.

Define the coordinate-block extraction map $R$ by
\[
\big(R(\bcA)\big)_{i_1,\ldots,i_d}
=
\bcA_{i_1,\,p+i_2,\,\ldots,\,(d-1)p+i_d}.
\]
In the partition formula for the cumulant of one coordinate from each block of $Z$, every partition with two or more parts contains only proper mixed-block moments and therefore contributes zero by \eqref{eq_reduction_proper_moments}.
The one-part term is \eqref{eq_reduction_encoded_moment}.
Hence
\begin{equation}\label{eq_reduction_encoded_cumulant}
R\big(\bcK_d(\widetilde P)\big)=a_d\bcK_d(P).
\end{equation}
For any tensor $\bcA$, embedding each $u_j\in\mathbb S^{p-1}$ into the $j$th coordinate block gives a unit vector in $\R^{dp}$ and reproduces the contraction of $R(\bcA)$ against $u_1\otimes\cdots\otimes u_d$.
Taking the supremum over these vectors proves
\begin{equation}\label{eq_reduction_block_contraction}
\|R(\bcA)\|\le\|\bcA\|.
\end{equation}
This holds for arbitrary tensors, not only for tensors belonging to the parameter class.

\medskip
\noindent\textbf{Step 2: Tail bounds.}
Fix $v=(v_1^\top,\ldots,v_d^\top)^\top\in\mathbb S^{dp-1}$.
For fixed $\Pi=\pi$ and signs, finite Jensen's inequality for the convex function $s\mapsto e^{s^2/4}$ gives
\[
\begin{aligned}
\E\left[e^{(v^\top Z)^2/4}\,\middle|\,\Pi=\pi,\varepsilon_1,\ldots,\varepsilon_{d-1}\right]
&\le
\frac1d\sum_{j=1}^d
\E\exp\left(\frac{(v_j^\top V_{b_\pi(j)})^2}{4}\right)\\
&\le2.
\end{aligned}
\]
The last inequality holds because $\|v_j\|\le1$ and $P\in\mathcal F_p$.
No independence among the summands is needed for Jensen's inequality.
On $\{\Pi=\star\}$ the exponential equals $1$, so averaging proves $\widetilde P\in\mathcal F_{dp}$.

We also require a uniform bound on $\|\bcK_d(P)\|$.
We use Markov's inequality in the form $\PP(W>b)\le\E W/b$ for $W\ge0$ and $b>0$.
For $u\in\mathbb S^{p-1}$, applying this inequality to $e^{(u^\top X)^2/4}$ gives $\PP_P(|u^\top X|>s)\le2e^{-s^2/4}$.
Integrating this tail yields
\[
\E_P|u^\top X|^d
\le
2d\int_0^\infty s^{d-1}e^{-s^2/4}\,ds
=
2^{d+1}\Gamma(1+d/2)
=:M_d.
\]
For a nonempty $B\subseteq[d]$ and unit vectors $u_1,\ldots,u_d$, H\"older's inequality with exponent $d$ for each projection, and with the remaining exponent assigned to the constant $1$ when $|B|<d$, gives
\[
\left|\E_P\prod_{j\in B}u_j^\top X\right|
\le
\prod_{j\in B}\left(\E_P|u_j^\top X|^d\right)^{1/d}
\le M_d^{|B|/d}.
\]
Multiplication over the blocks of each partition, followed by the moment--cumulant formula and the supremum over the unit vectors, gives
\begin{equation}\label{eq_reduction_target_bound}
\sup_{P\in\mathcal F_p}\|\bcK_d(P)\|\le A_dM_d.
\end{equation}

\medskip
\noindent\textbf{Step 3: Dilution and a finite sample budget.}
For $0<\alpha\le1$, let $H_{P,\alpha}=(1-\alpha)\delta_0+\alpha\widetilde P$, where $\delta_0$ is the point mass at zero in $\R^{dp}$.
For every unit vector $v$, its directional exponential moment is at most $(1-\alpha)+2\alpha\le2$.
Thus $H_{P,\alpha}\in\mathcal F_{dp}$.
All proper mixed-block moments in \eqref{eq_reduction_proper_moments} remain zero, while the full mixed-block moment is multiplied by $\alpha$.
Applying the partition formula again yields the exact identity
\begin{equation}\label{eq_reduction_diluted_target}
R\big(\bcK_d(H_{P,\alpha})\big)=\alpha a_d\bcK_d(P).
\end{equation}
This identity concerns the extracted block; no linearity of the full cumulant tensor under mixtures is asserted.

Given $m$ observations from $P$, put
\[
L=\lfloor m/d\rfloor,
\qquad
k=\left\lceil\log_2\left(\frac{2n}{L}\right)\right\rceil,
\qquad
\alpha=2^{-k}.
\]
Since $m\ge2d$ and $m\le n$, these choices satisfy
\begin{equation}\label{eq_reduction_dilution_bounds}
\frac{m}{2d}\le L\le\frac md,
\qquad
\frac{L}{4n}<\alpha\le\frac{L}{2n},
\qquad
\frac1{\alpha a_d}\le\frac{8d}{a_d}\frac nm.
\end{equation}
Generate independent Bernoulli$(\alpha)$ indicators $I_1,\ldots,I_n$, independently of the observations, and set $J=\sum_{i=1}^n I_i$.
If $J>L$, return the zero tensor.
Otherwise, partition the first $dL$ observations into $L$ disjoint groups of size $d$.
For each active indicator, in increasing index order, use the next group and fresh independent encoding randomness to generate a vector with law $\widetilde P$.
Set the remaining auxiliary observations to zero.
Denote the resulting vectors by $Y_1,\ldots,Y_n$ and return
\begin{equation}\label{eq_reduction_estimator_definition}
\widehat{\bcK}^{\mathrm{red}}_{p,m;n}
=
\frac{1}{\alpha a_d}
R\big(\widehat{\bcK}_{dp,n}(Y_1,\ldots,Y_n)\big)
\qquad\text{on }\{J\le L\}.
\end{equation}
The auxiliary estimator uses independent internal randomness.

To analyze the construction, extend the original observations to an infinite i.i.d.\ sequence and apply the same procedure without stopping when $J>L$.
Call its auxiliary observations $Y_1^\circ,\ldots,Y_n^\circ$.
Conditional on the indicator vector, the active observations use disjoint independent groups and independent encoding randomness, and hence are independent with law $\widetilde P$.
Averaging over the independent indicators gives
\[
(Y_1^\circ,\ldots,Y_n^\circ)\sim H_{P,\alpha}^{\otimes n}.
\]
On $\{J\le L\}$, these ideal observations and the actual auxiliary observations agree exactly.
We use this coupling, rather than claiming that the observations conditioned on $J\le L$ are i.i.d.

Independence of the indicators gives $\E[2^J]=(1+\alpha)^n$.
Markov's inequality, $1+x\le e^x$, and $n\alpha\le L/2$ imply
\begin{equation}\label{eq_reduction_overflow}
\begin{aligned}
\PP(J>L)
&\le2^{-L}\E[2^J]
\le\exp(-L\log2+n\alpha)\\
&\le\exp\left(-\left(\log2-\frac12\right)L\right)
\le e^{-c_dm},
\end{aligned}
\end{equation}
where $c_d=(\log2-1/2)/(2d)>0$.

\medskip
\noindent\textbf{Step 4: Transfer of the error bounds and computational cost.}
Use the same internal random seed for the auxiliary estimator in the actual and ideal experiments, and write
\[
\widehat{\bcK}^{\circ}
=
\widehat{\bcK}_{dp,n}(Y_1^\circ,\ldots,Y_n^\circ).
\]
By \eqref{eq_reduction_diluted_target} and \eqref{eq_reduction_block_contraction}, on $\{J\le L\}$,
\begin{equation}\label{eq_reduction_pathwise_error}
\left\|\widehat{\bcK}^{\mathrm{red}}_{p,m;n}-\bcK_d(P)\right\|
\le
\frac1{\alpha a_d}
\left\|\widehat{\bcK}^{\circ}-\bcK_d(H_{P,\alpha})\right\|.
\end{equation}
On the complementary event the error equals $\|\bcK_d(P)\|$.
Therefore, using the nonnegativity of the ideal error, \eqref{eq_reduction_target_bound}, and \eqref{eq_reduction_overflow},
\[
\begin{aligned}
\E_P\left\|\widehat{\bcK}^{\mathrm{red}}_{p,m;n}-\bcK_d(P)\right\|
&\le
\frac1{\alpha a_d}
\E\left\|\widehat{\bcK}^{\circ}-\bcK_d(H_{P,\alpha})\right\|
+A_dM_d\PP(J>L)\\
&\le
\frac{8d}{a_d}\frac nm\,r(dp,n)+A_dM_de^{-c_dm}.
\end{aligned}
\]
Here the expectation involving $\widehat{\bcK}^{\circ}$ is under the unconditional product experiment $H_{P,\alpha}^{\otimes n}$.
Taking the supremum over $P$ proves \eqref{eq_cumulant_reduction_risk} with $C_d\ge\max\{8d/a_d,A_dM_d\}$.
The same coupling, \eqref{eq_reduction_pathwise_error}, and $H_{P,\alpha}\in\mathcal F_{dp}$ also give the probability bound
\begin{equation}\label{eq_cumulant_reduction_tail}
\begin{aligned}
&\sup_{P\in\mathcal F_p}
\PP_P\left(
\left\|\widehat{\bcK}^{\mathrm{red}}_{p,m;n}-\bcK_d(P)\right\|
>C_d\frac nm t
\right)\\
&\qquad\le
\sup_{H\in\mathcal F_{dp}}
\PP_H\left(
\left\|\widehat{\bcK}_{dp,n}-\bcK_d(H)\right\|>t
\right)+e^{-c_dm},
\qquad t>0.
\end{aligned}
\end{equation}

For fixed $d$, the partition table is finite and can be fixed in advance.
Because $B_d$ is a power of two and the partition weights have denominator $B_d$, $\Pi$ is sampled using exactly $\log_2 B_d$ fair bits, assigning $|\mu(\pi)|$ outcomes to each partition and the unused outcomes to $\star$.
Each Bernoulli$(\alpha)$ indicator is generated with $k=O(\log(n+1))$ fair bits by requiring all of them to be zero.
The remaining signs require at most $d-1$ bits per encoding.
Thus the randomness has a bounded polynomial cost, rather than merely an expected polynomial cost.
Apart from the auxiliary estimator, constructing the observations, extracting the block, and rescaling use $O_d(np+p^d+n\log(n+1))$ elementary operations.
The algorithm only copies, changes signs, and rescales observed coordinates, pads with zeros, counts indicators, and extracts tensor entries; no tensor norm is evaluated.
The stated polynomial-time preservation follows when $n$ is polynomially bounded and polynomial-time computable from $(m,p)$.
\end{proof}

We prove Proposition~\ref{prop_cumulant_estimation_reduction} by contraposition.
Suppose Conjecture~\ref{conj_optimal} fails, so that a single polynomial-time estimator achieves uniform expected error $o(q^{d/2}/n_q)$ along a polynomial-time computable sequence $q^{d/2}\ll n_q\ll q^{d-1}$.
For each $p$, apply Lemma~\ref{lem_cumulant_estimation_reduction} with auxiliary dimension $dp$, sample size $n=n_{dp}$, and worst-case expected error $r_p:=r(dp,n_{dp})$.
Since $d$ is fixed, the assumptions imply $p^{d/2}\ll n_{dp}\ll p^{d-1}$ and $n_{dp}r_p=o(p^{d/2})$.

Choose
\[
m_p=\left\lceil\sqrt{p^{d/2}(p+n_{dp}r_p)}\right\rceil.
\]
Since $d\ge3$, we have $p=o(p^{d/2})$, and hence $m_p=o(p^{d/2})$ and $m_p\to\infty$.
Moreover, $2d\le m_p\le n_{dp}$ for all sufficiently large $p$, and
\[
\frac{n_{dp}r_p}{m_p}
\le\sqrt{\frac{n_{dp}r_p}{p^{d/2}}}
\longrightarrow0.
\]
Applying \eqref{eq_cumulant_reduction_risk} therefore gives
\[
\sup_{P\in\mathcal F_p}\E_P\left\|\widehat{\bcK}^{\mathrm{red}}_{p,m_p;n_{dp}}-\bcK_d(P)\right\|
\le C_d\frac{n_{dp}r_p}{m_p}+C_de^{-c_dm_p}
\longrightarrow0.
\]
By Markov's inequality, for every fixed $\varepsilon>0$,
\[
\sup_{P\in\mathcal F_p}
\PP_P\left(\left\|\widehat{\bcK}^{\mathrm{red}}_{p,m_p;n_{dp}}-\bcK_d(P)\right\|>\varepsilon\right)
\longrightarrow0.
\]

The reduced estimator is implemented by a single algorithm that, given $p$ and $m$, computes $n_{dp}$ and applies the reduction whenever $2d\le m\le n_{dp}$, returning the zero tensor otherwise.
Since $n_{dp}$ is polynomial-time computable and polynomially bounded in $p$, this algorithm runs in time polynomial in $(m,p)$.
The sequence $m_p$ is used only to establish consistency and need not be computed by the algorithm.
Thus, if Conjecture~\ref{conj_optimal} fails, then Conjecture~\ref{conj_consistent} also fails, proving the proposition by contraposition.

\section{Lemmas}\label{app_lemmas}

\begin{lemma}\label{lemma_existance_1d}
    There exists a sub-Gaussian random variable $X \in \R$, such that its first $d-1$ cumulants are Gaussian cumulants while the $d$th cumulant is some non-zero constant $c_d$. 
    Moreover, we have $|\E H_m(X)| \leq m! C_d^m$ for all $m\in\mathbb{N}^+$, where $H_m(t)$ denotes the order-$m$ probabilists’ Hermite polynomial.  
\end{lemma}

\begin{proof}
    Let $\psi_m(t) = \operatorname{sign}(H_m(t)) \min\{|H_m(t)|, M\}$. Define $X_{\beta, \theta}$ by its density 
    \[
        f(x) = \frac{\exp\( \sum_{i\in[d-1]} \beta_i \psi_i(x) + \theta \psi_d(x)\) \phi(x)}{C_{\beta, \theta}}, 
    \]
    where $\beta \in \R^{d-1}$, $\theta \in\R$, and $\phi(x)$ is the density of standard normal distribution. 
    
    When $\theta = 0, \beta = 0$, $X_{\beta, \theta} \sim N(0,1)$. Denote
    \[
        F(\beta, \theta) = (\E H_1(X_{\beta, \theta}), \ldots, \E H_{d-1}(X_{\beta, \theta}))\T. 
    \]
        Let $J_\beta=D_\beta F(0,0)$, with its dependence on $M$ understood. Since the functions $\psi_j$ are bounded, differentiation under the integral sign gives
    \[
        (J_\beta)_{ij}
        =\mathrm{Cov}(H_i(Z),\psi_j(Z))
        =\E[H_i(Z)\psi_j(Z)],
        \qquad 1\le i,j\le d-1,
    \]
    where $Z\sim N(0,1)$. Since $|\psi_j(Z)|\le |H_j(Z)|$ and $\E|H_i(Z)H_j(Z)|<\infty$, dominated convergence and Hermite orthogonality imply
    \[
        J_\beta\longrightarrow \mathrm{diag}(1!,2!,\ldots,(d-1)!)
        \qquad\text{as }M\to\infty.
    \]
    Thus, $J_\beta$ is invertible for all sufficiently large $M$. For each such $M$, the implicit function theorem yields $\delta_M>0$ and a continuously differentiable map $\theta\mapsto\beta_M(\theta)$, defined for $|\theta|<\delta_M$, such that $\beta_M(0)=0$ and $F(\beta_M(\theta),\theta)=0$. We suppress the dependence on $M$ in the following calculations and fix a sufficiently large $M$ below.

    Thus, for all $\theta$ that is sufficiently small, we can always pick $\beta$ such that $\E H_i(X_{\beta, \theta}) = 0$ for all $i\in [d-1]$, which further implies that the cumulants match the Gaussian cumulants up to order $d-1$. 

    It remains to choose $\theta$ such that the $d$th cumulant of $X_\theta := X_{\beta(\theta),\theta}$ is nonzero. Since the moments of $X_\theta$ below order $d$ match those of a standard Gaussian and $H_d$ is monic with all remaining terms of degree less than $d$, the moment--cumulant relation gives $\kappa_d(X_\theta)=\E H_d(X_\theta)$.
    
    Let $G(\beta, \theta) = \E H_d(X_{\beta, \theta})$ and $g(\theta) = \E H_d(X_\theta)$. 
    
    Note that $g(0) = 0$ and differentiating at $0$ with the chain rule gives
    \[
      g'(0)
      = \partial_\theta G(0,0)
        + \sum_{j=1}^{d-1} \partial_{\beta_j} G(0,0)\, [\beta'(0)]_j.
    \]
    The partial derivatives can be written as covariances:
    \begin{align*}
      \partial_\theta G(0,0) &= \mathrm{Cov}\big(H_d(Z),\psi_d(Z)\big),\\
      \partial_{\beta_j} G(0,0) &= \mathrm{Cov}\big(H_d(Z),\psi_j(Z)\big).
    \end{align*}
    Moreover, $\beta'(0)$ is determined by differentiating the identity
    $F(\beta(\theta),\theta)=0$ at $\theta=0$:
    \[
      J_\beta\, \beta'(0) + F_\theta(0,0) = 0,
    \]
    where $F_\theta(0,0)$ has entries $\mathrm{Cov}(H_i(Z),\psi_d(Z))$. Hence
    \[
      \beta'(0) = - \big(J_\beta\big)^{-1} F_\theta(0,0).
    \]

    Using again that $\psi_m\to H_m$ as $M\to\infty$, we obtain
    \begin{align*}
      J_\beta &\to \mathrm{diag}(1!,2!,\dots,(d-1)!),\\
      F_\theta(0,0) &\to 0,\\
      \partial_\theta G(0,0) &\to \E H_d(Z)^2 = d!,\\
      \partial_{\beta_j} G(0,0) &\to 0,\quad j=1,\dots,d-1.
    \end{align*}
    Consequently,
    \[
      g'(0)
      = \partial_\theta G(0,0)
        + \sum_{j=1}^{d-1} \partial_{\beta_j} G(0,0)\,[\beta'(0)]_j
      \rightarrow d!.
    \]
    Since $g'(0)\to d!>0$ as $M\to\infty$, we may choose and fix a sufficiently large $M=M(d)$ such that $J_\beta$ is invertible and $g'(0)>0$. Since $g(0)=0$ and $g$ is continuously differentiable near zero, there exists $0<\theta_0<\delta_M$ such that $g(\theta_0)>0$.

    For $X:=X_{\theta_0} = X_{\beta(\theta_0),\theta_0}$, we have for $m\leq d-1$, $\kappa_m(X)$ equals the Gaussian cumulant of order $m$; and $\kappa_d(X) = g(\theta_0)\neq 0$.
    
    For this fixed $M$ and $(\beta,\theta_0)$, the tilt
    \[
      T(x) := \sum_{i=1}^{d-1}\beta_i \psi_i(x) + \theta_0 \psi_d(x)
    \]
    is bounded, say $|T(x)|\le B_d$ for all $x$. Hence there exist constants $0<c_d\leq C_d<\infty$ (depending only on $d$) such that
    \[
      c_d\phi(x) \leq f_X(x) \leq C_d\phi(x)
      \quad\text{for all }x\in\mathbb{R}.
    \]
    This implies in particular that $X$ is sub-Gaussian with a parameter depending only on $d$.

    {
    The density comparison also gives the required Hermite bound without a moment-dependent tail estimate.
    For $G\sim N(0,1)$, Hermite orthogonality gives $\E H_m(G)^2=m!$.
    Consequently, for every $m\ge1$, Cauchy--Schwarz yields
    \[
    |\E H_m(X)|\le\E|H_m(X)|
    \le C_d\E|H_m(G)|
    \le C_d\sqrt{m!}
    \le m!\widetilde C_d^{\,m},
    \]
    with $\widetilde C_d=\max\{C_d,1\}$.
    }
    This completes the proof.
\end{proof}

\begin{lemma}\label{lemma_kronecker_lower_bound}
    For unit vectors $u,v\in \R^p$ with $\langle u,v \rangle > 0$, we have $\| u^{\otimes d} - v^{\otimes d}\| \geq \|u-v\| / 2$. 
\end{lemma}
\begin{proof}
    Let
    $$
    \alpha:=\langle u, v\rangle , \quad 
    \delta:=\|u-v\|=\sqrt{2-2 \alpha}, \quad 
    w:=(u-v) / \delta.
    $$
    So we have \[
        u\T w = - v\T w = \delta / 2. 
    \]
    Choose the $d$ unit vectors
    $$
    x_1=\cdots=x_{d-1}:=u, \quad x_d:=w.
    $$
    Then we have 
    \[
        \| u^{\otimes d} - v^{\otimes d}\| 
        \geq \langle u^{\otimes d} - v^{\otimes d}, \bigotimes_{k = 1}^d x_k\rangle
         = \delta / 2 (1+ \alpha^{d-1}) \geq \delta / 2. 
    \]
    
\end{proof}

\begin{lemma}[Theorem 2 in \cite{friedlandNumberSingularVector2013}]\label{lemma_symmetric_tensor_approximation}
    For any symmetric tensor, there exists a best rank-one approximation that is symmetric.  
\end{lemma}


\begin{lemma}[Corollary 1.4 in \cite{gotzeConcentrationInequalitiesPolynomials2021}]\label{lemma_subexpo_concentration}
     Let $X_1, \ldots, X_n$ be a set of independent, centered random variables with {$\left\|X_i\right\|_{\Phi_\alpha} \leq M$} for some $\alpha \in(0,1]$. Let $a \in \mathbb{R}^n$. For any $t \geq 0$ it holds

$$
\mathbb{P}\left(\left|\sum_{i=1}^n a_i X_i\right| \geq t\right) \leq 2 \exp \left(-\frac{1}{C_\alpha} \min \left(\frac{t^2}{M^2\|a\|^2}, \frac{t^\alpha}{M^\alpha \max _i\left|a_i\right|^\alpha}\right)\right). 
$$

\end{lemma}

\end{document}